\documentclass{article}
\usepackage[colorinlistoftodos]{todonotes}

\input{insbox}
\usepackage{ifxetex} \ifxetex \usepackage{fontspec} \else
\usepackage[T1]{fontenc} \usepackage[utf8]{inputenc}
\usepackage{lmodern} \fi \usepackage{longtable, multicol}
\usepackage{graphicx} \usepackage{verbatim} \usepackage{enumerate}  \usepackage{dsfont}
\usepackage[colorlinks=true, linkcolor=blue, citecolor=blue, pagebackref=true]{hyperref} 
\usepackage{caption}
\usepackage{enumitem}
\usepackage[normalem]{ulem}
\usepackage{animate}

\usepackage{tikz} 
\usepackage{tikz-layers}
\usepackage{wrapfig} 
\usepackage{mwe}
\usepackage{tkz-graph}
\usetikzlibrary{3d,shapes,snakes,shapes,positioning, angles,calc,patterns,quotes}
\usetikzlibrary{arrows}
\usetikzlibrary{arrows.meta}
\usetikzlibrary{shapes.misc}
\usetikzlibrary{calc,arrows.meta}
\usepackage{rotating}
\usetikzlibrary{decorations.markings}
\tikzset{cross/.style={cross out, draw=black, fill=none, minimum size=2*(#1-\pgflinewidth), inner sep=0pt, outer sep=0pt}, cross/.default={2pt}}
\tikzstyle arrowstyle=[scale=1]
\tikzstyle directed=[postaction={decorate,decoration={markings,
    mark=at position .65 with {\arrow[arrowstyle]{stealth}}}}]
\tikzstyle reverse directed=[postaction={decorate,decoration={markings,
    mark=at position .65 with {\arrowreversed[arrowstyle]{stealth};}}}]

\newcommand{\AxisRotator}[1][rotate=0]{%
    \tikz [x=0.25cm,y=0.60cm,line width=.2ex,-stealth,#1] \draw (0,0) arc (-150:150:1 and 1);%
}

\usepackage{comment}

\usepackage{amsmath, amssymb, amsthm}
\usepackage{thmtools} 
\usepackage{cleveref}
\usepackage{mathtools}
\usepackage[margin=1in]{geometry}     
\usepackage{xcolor}

\usepackage{float}
\usepackage[caption = false]{subfig}

\allowdisplaybreaks

\theoremstyle{plain}
\newtheorem{theorem}{Theorem}[section]
\newtheorem{proposition}[theorem]{Proposition}
\newtheorem{lemma}[theorem]{Lemma}
\newtheorem{corollary}[theorem]{Corollary}

\theoremstyle{remark}
\newtheorem{remark}{Remark}[section]
\newtheorem{remark*}{Remark}

\theoremstyle{definition}
\newtheorem{question}{Question}
\theoremstyle{definition}
\newtheorem{definition}{Definition}
\newtheorem{example}{Example}

\newcommand{\NN}{\mathbb{N}}

\newcommand{\QQ}{\mathbb{Q}}

\newcommand{\RR}{\mathbb{R}}
\newcommand{\PP}{\mathbb{P}}

\newcommand{\ZZ}{\mathbb{Z}}

\newcommand{\calC}{\mathcal C}

\newcommand{\calF}{\mathcal F}

\newcommand{\calH}{\mathcal H}

\newcommand{\calL}{\mathcal L}

\newcommand{\calO}{\mathcal O}
\newcommand{\calP}{\mathcal{P}}

\newcommand{\calS}{\mathcal{S}}

\newcommand{\bzero}{\mathbf{0}}

\newcommand{\bt}{\mathbf{t}}

\newcommand{\bv}{\mathbf{v}}

\newcommand{\HH}{\mathbb{H}}
\newcommand{\Aff}{\mathrm{Aff}}
\newcommand{\Sp}{\mathrm{Sp}}
\newcommand{\Stab}{\mathrm{Stab}}
\newcommand{\cube}{\mathbf{K}}
\newcommand{\SL}{\mathrm{SL}}
\newcommand{\GL}{\mathrm{GL}}
\newcommand{\PSL}{\mathrm{PSL}}
\newcommand{\TT}{\mathbb{T}}
\renewcommand{\P}{\mathrm{P}}
\newcommand{\PGL}{\mathrm{PGL}}
\newcommand{\Id}{\mathrm{Id}}

\DeclarePairedDelimiter{\norm}{\lVert}{\rVert}
\DeclarePairedDelimiter{\abs}{\lvert}{\rvert}
\DeclarePairedDelimiter{\set}{\lbrace}{\rbrace}
\DeclarePairedDelimiter{\floor}{\lfloor}{\rfloor}

\DeclarePairedDelimiter{\parens}{\lparen}{\rparen}
\DeclarePairedDelimiter{\brackets}{\lbrack}{\rbrack}
\DeclarePairedDelimiter{\ideal}{\langle}{\rangle}

\DeclareMathOperator{\len}{length}
\DeclareMathOperator{\slope}{slope}
\DeclareMathOperator{\wid}{width}


\makeatletter

\newenvironment{customtheorem}[1]
  {\count@\c@theorem
   \global\c@theorem#1 %
    \global\advance\c@theorem\m@ne
   \theorem}
  {\endtheorem
   \global\c@theorem\count@}

\makeatother

\title{Straight-line trajectories on the Mucube}
\author{Andre Oliveira\footnote{Wesleyan University, \texttt{aoliveira@wesleyan.edu}} \and Felipe A. Ram\'irez\footnote{Wesleyan University, \texttt{framirez@wesleyan.edu}} \and Chandrika Sadanand\footnote{Bowdoin College, \texttt{c.sadanand@bowdoin.edu}} \and Sunrose T. Shrestha\footnote{Carleton College, \texttt{sshrestha@carleton.edu}}}

\date{}

\begin{document}
\maketitle
\begin{abstract}
  The dynamics of straight line flows on compact half-translation
  surfaces (surfaces formed by gluing Euclidean polygons edge-to-edge
  via translations possibly composed with rotation by $\pi$) has been
  widely studied due to their connections to polygonal billiards and
  Teichm\"uller theory. However, much less is known when the
  underlying surface is non-compact or infinite type. In this paper,
  we consider the straight line flow of the Mucube---an infinite
  $\ZZ^3$-periodic half-translation square-tiled surface---first
  written about by Coxeter and Petrie and more recently studied by
  Athreya--Lee and Guti\'errez-Romo--Lee--S{\'a}nchez. We give a
  geometric description of the flow's periodic and drift orbits in
  terms of the Mucube's rigid symmetries, and we give a complete
  characterization of the set of directions in which the straight line
  flow is periodic on the Mucube---first in terms of a genus one
  quotient and second in terms of an infinitely generated subgroup of
  $\SL_2(\ZZ)$. We use the latter characterization to obtain the Veech
  group (i.e. group of derivatives of affine diffeomorphisms) of the
  Mucube. Finally, we prove density of the sets of periodic and
  ergodic directions.
\end{abstract}
\vfill
\pagebreak

\tableofcontents
\pagebreak

\section{Introduction}\label{sec:intro}

The study of the behavior of straight line trajectories on compact,
finite-type {\em translation surfaces} (flat surfaces built by gluing
finitely many Euclidean polygons edge-to-edge via translations) was
first motivated by the desire to understand billiard trajectories in
rational polygons (polygons with angles that are rational multiples of
$\pi$). Katok--Zemlyakov \cite{katzem} showed that one can relate
billiard trajectories on a rational polygon to straight line
trajectories on a finite translation surface via a process that is now
popularly known as ``unfolding a billiard table.''

\begin{wrapfigure}{r}{0.35\textwidth}
\centering 
\subfloat[A part of the infinite staircase.]{
\includegraphics[scale=0.4]{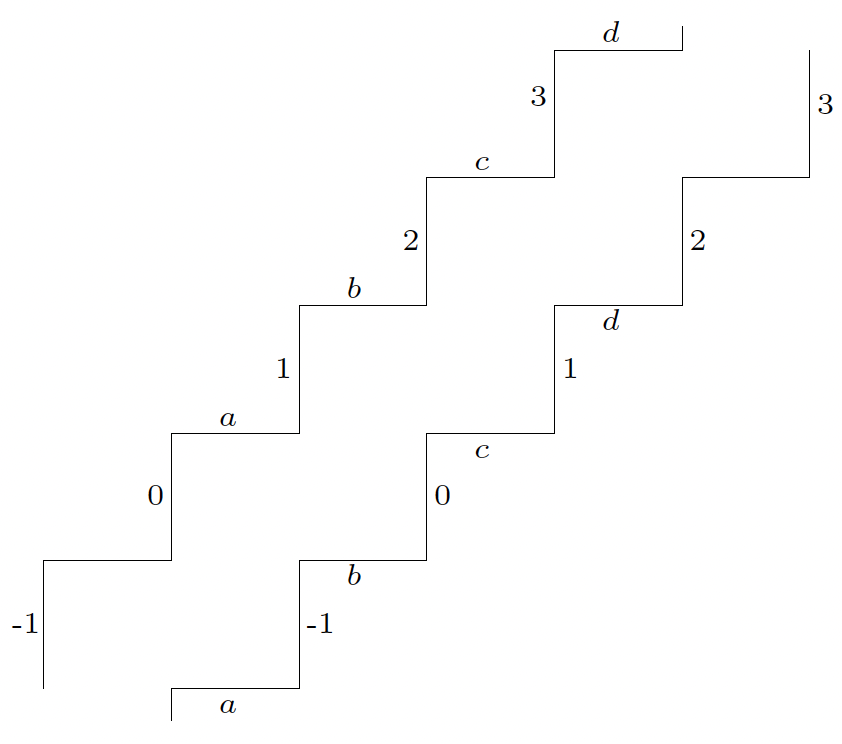}
\label{fig:infinitestaircase}
}

\subfloat[A part of the Ehrenfest wind-tree billiard.]{
\includegraphics[scale=0.5]{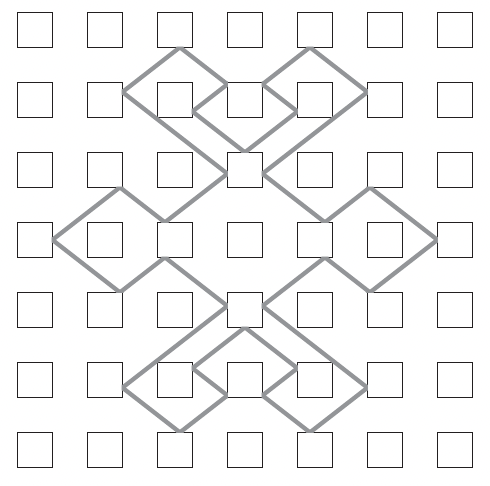}
\label{fig:windtree}
}
\caption{Two infinite settings to study the linear flow. Figures taken from~\cite{hoophubweiss} and \cite{hubleltrou}~respectively.}
\end{wrapfigure}

Since then, straight line trajectories on translation surfaces have
greatly interested many authors, spurring several deep results on the
dynamics and statistics of trajectories on translation surfaces. For
instance, it is known that a certain class of compact translation
surfaces called {\em Veech surfaces} (surfaces whose groups of
derivatives of affine diffeomorphisms form lattices in $\SL_2(\RR)$),
exhibit optimal dynamics for the straight line flow on the surface, in
the sense that in any given direction, either every orbit of the
straight-line flow is periodic or the flow is uniquely ergodic. This
is the celebrated Veech Dichotomy theorem \cite{masur1,
  veech1}. Another instance is the result proven by Masur
  \cite{masurDensity} that every compact translation surface has a dense set
  of directions in which there exists a closed straight line
  trajectory.

By comparison to the compact case, less is known about the dynamics of
the linear flow on general translation surfaces of infinite type. Even
the things that are known, are for specific families of surfaces. The
study of infinite-type translation surfaces was in part motivated by
Valdez \cite{valdez}, who showed that applying the Katok--Zemlyakov
unfolding construction to billiard trajectories on an irrational
polygon leads to trajectories on infinite translation surfaces
homeomorphic to the {\em Loch Ness monster}, a surface with infinite
genus and one end. Hooper \cite{hooper1, hooper2} created a family of
infinite Veech-type surfaces and, in particular, studied the dynamics
of pseudo Anosovs on such surfaces. Subsequently,
Hooper--Hubert--Weiss \cite{hoophubweiss} studied the linear flow on
the {\em infinite staircase} (a $\ZZ$-cover of the square torus, see
Figure \ref{fig:infinitestaircase}), providing the first example of an
infinite surface that satisfies a similar dynamical dichotomy such as
the one satisfied by Veech surfaces in the finite setting. General
$\ZZ$-covers of compact translation surfaces were studied by
Hooper--Weiss \cite{hoopweiss}, who found necessary and sufficient
conditions for recurrence of the straight line flows on such surfaces.

Another example of an infinite surface that has generated interest in
recent years is the translation surface arising from the {\em
  Ehrenfest wind-tree model}. The most basic family of Ehrenfest
wind-tree models are billiards in
$\RR^2$ with rectangular obstacles of varied dimensions placed at the
$\ZZ^2$ lattice points, see Figure \ref{fig:windtree}.  Using the
Katok--Zemlyakov construction and the theory of translation surfaces,
Hubert--Leli\`evre--Troubetzkoy \cite{hubleltrou} showed that the
unfolding of a wind-tree billiard is a $\ZZ^2$-cover of a compact
translation surface. By studying the straight line flow on the
corresponding infinite surface, they showed the existence of
completely periodic directions for the wind-tree billiard
(i.e. directions in which every non-singular billiard trajectory is
periodic) corresponding to a class of rational parameters of the
obstacles and also demonstrated the existence of rational parameters
for which the trajectories escape. Subsequently, M\'alaga
Sabogal--Troubetzkoy \cite{sabtrou1, sabtrou2, sabtrou3, sabtrou4} and
Fr\c aczek--Ulcigrai \cite{FraUlci} studied their minimality and
ergodicity properties while Pardo \cite{pardo} studied the asymptotics
for the number of closed billiard trajectories up to a fixed length.

In this note, we study the straight line flow of the {\em Mucube}, a
$\ZZ^3$-periodic half-translation surface homeomorphic to the Loch
Ness monster. In order to get an intuitive understanding of the
Mucube, consider a solid cube of sidelength two, $\cube$, centered at
the origin (the vertices are at $(\pm 1,\pm 1,\pm 1)$). Embedded in
$\cube$, consider the surface (with boundary) formed in the interior
by drilling out square cylinders through the $1 \times 1$ center
squares in each of the faces of $\cube$. The Mucube is the union of
all $(2\ZZ)^3$ translates of this surface. We will postpone further
descriptions of the Mucube, including the parallel fields and rigid
symmetries of its natural embedding in $\RR^3$, its half-translation
structure,
and its fundamental group to
Sections~\ref{sec:parallel-fields-m},~\ref{sec:rigidSym},~\ref{sec:MucubeXYHalftranslation}
and \ref{sec:MucubeFundamentalGroup}, respectively. For now, see
Figure \ref{fig:MucubeConstruction} for an illustration of this construction of
the Mucube.

Although first introduced in 1937 by Coxeter and Petrie \cite{coxeter}
as an example of an infinite polyhedral surface, the Mucube has
generated significant interest in recent years. In \cite{athlee},
Athreya--Lee study the affine symmetry group of the $\ZZ^3$ quotient
of the Mucube and quadratic asymptotics for counts of cylinders
curves. Additionally, Guit\'errez-Romo--Sanchez--Lee \cite{GutLeeSan}
consider the translation cover of this $\ZZ^3$ quotient and compute
the Zariski closure of its Kontsevich--Zorich monodromy group. Our
results about the Mucube, which we state in the following section are
more dynamically oriented.

\begin{figure}[h!]
\centering
\subfloat[The building block of the Mucube. The surface is shaded in light gray (equivalently, dark gray)]{

\centering
\begin{tikzpicture}[scale=0.7, shift={(0,0,4)}]
\coordinate (O) at (0,0,0);
\coordinate (A) at (0,6,0);
\coordinate (B) at (0,6,6);
\coordinate (C) at (0,0,6);
\coordinate (D) at (6,0,0);
\coordinate (E) at (6,6,0);
\coordinate (F) at (6,6,6);
\coordinate (G) at (6,0,6);

\coordinate (B2) at (1.5,4.5,6);
\coordinate (C2) at (1.5,1.5,6);
\coordinate (F2) at (4.5,4.5,6);
\coordinate (G2) at (4.5,1.5,6);

\coordinate (A3) at (1.5,4.5,4.5);
\coordinate (B3) at (1.5,4.5,6);
\coordinate (E3) at (4.5,4.5,4.5);
\coordinate (F3) at (4.5,4.5,6);

\coordinate (D4) at (4.5,1.5,4.5);
\coordinate (E4) at (4.5,4.5,4.5);
\coordinate (F4) at (4.5,4.5,6);
\coordinate (G4) at (4.5,1.5,6);

\coordinate (B5) at (1.5,6,4.5);
\coordinate (C5) at (1.5,4.5,4.5);
\coordinate (F5) at (4.5,6,4.5);
\coordinate (G5) at (4.5,4.5,4.5);

\coordinate (A6) at (1.5,6,1.5);
\coordinate (B6) at (1.5,6,4.5);
\coordinate (E6) at (4.5,6,1.5);
\coordinate (F6) at (4.5,6,4.5);

\coordinate (D7) at (4.5,4.5,1.5);
\coordinate (E7) at (4.5,6,1.5);
\coordinate (F7) at (4.5,6,4.5);
\coordinate (G7) at (4.5,4.5,4.5);

\coordinate (B8) at (4.5,4.5,4.5);
\coordinate (C8) at (4.5,1.5,4.5);
\coordinate (F8) at (6,4.5,4.5);
\coordinate (G8) at (6,1.5,4.5);

\coordinate (A9) at (4.5,4.5,1.5);
\coordinate (B9) at (4.5,4.5,4.5);
\coordinate (E9) at (6,4.5,1.5);
\coordinate (F9) at (6,4.5,4.5);

\coordinate (D10) at (6,1.5,1.5);
\coordinate (E10) at (6,4.5,1.5);
\coordinate (F10) at (6,4.5,4.5);
\coordinate (G10) at (6,1.5,4.5);

\coordinate (B11) at (0,4.5,4.5);
\coordinate (C11) at (0,1.5,4.5);
\coordinate (F11) at (A3);
\coordinate (G11) at (1.5,1.5,4.5);

\coordinate (A12) at (0,4.5,1.5);
\coordinate (B12) at (B11);
\coordinate (E12) at (1.5,4.5,1.5);
\coordinate (F12) at (F11);

\coordinate (B13) at (G11);
\coordinate (C13) at (1.5,0,4.5);
\coordinate (F13) at (D4);
\coordinate (G13) at (4.5,0,4.5);

\coordinate (D14) at (4.5,0,1.5);
\coordinate (E14) at (4.5,1.5,1.5);
\coordinate (F14) at (F13);
\coordinate (G14) at (G13);

\coordinate (A15) at (1.5,4.5,0);
\coordinate (B15) at (E12);
\coordinate (E15) at (4.5, 4.5, 0);
\coordinate (F15) at (A9);

\coordinate (D16) at (4.5,1.5,0);
\coordinate (E16) at (E15);
\coordinate (F16) at (F15);
\coordinate (G16) at (E14);

\begin{scope}[on glass layer]


\draw[] (C2) -- (B2) -- (F2) -- (G2) -- cycle;

\draw[fill=gray!80] (A3) -- (B3) -- (F3) -- (E3) -- cycle;

\draw[fill=gray!80] (D4) -- (E4) -- (F4) -- (G4) -- cycle;



\draw[fill=gray!80] (C5) -- (B5) -- (F5) -- (G5) -- cycle;

\draw[] (A6) -- (B6) -- (F6) -- (E6) -- cycle;

\draw[fill=gray!80] (D7) -- (E7) -- (F7) -- (G7) -- cycle;


\draw[fill=gray!80] (C8) -- (B8) -- (F8) -- (G8) -- cycle;

\draw[fill=gray!80] (A9) -- (B9) -- (F9) -- (E9) -- cycle;

\end{scope}

\begin{scope}[on above layer]
\draw[fill=gray!30] (G11) -- (C2) -- (B2) -- (F11) -- cycle;

\draw[fill=gray!30] (G11) -- (C2) -- (G2) -- (D4) -- cycle;


\draw[fill=gray!30] (E9) -- (D10) -- (E14) -- (A9) -- cycle;

\draw[fill=gray!30] (A6) -- (E6) -- (D7) -- (B15) -- cycle;

\end{scope}

\begin{scope} 

\draw[fill=gray!80] (B11) -- (C11) -- (G11) -- (F11) -- cycle;

\draw[fill=gray!80] (A12) -- (B12) -- (F12) -- (E12) -- cycle;


\draw[fill=gray!80] (B13) -- (C13) -- (G13) -- (F13) -- cycle;

\draw[fill=gray!80] (D14) -- (E14) -- (F14) -- (G14) -- cycle;

\draw[fill=gray!80] (A15) -- (B15) -- (F15) -- (E15) -- cycle;

\draw[fill=gray!80] (D16) -- (E16) -- (F16) -- (G16) -- cycle;
\end{scope}

\begin{scope}[on behind layer]

\draw[fill=gray!30] (G11) -- (C11) -- (0, 1.5, 1.5) -- (1.5, 1.5, 1.5) -- cycle;

\draw[fill=gray!30] (0, 1.5, 1.5) -- (1.5, 1.5, 1.5) -- (E12) -- (A12) --cycle;


\draw[fill=gray!30](1.5, 1.5, 1.5)-- (G16) -- (D16) -- (1.5, 1.5, 0) -- cycle;

\draw[](1.5, 1.5, 0) -- (A15) -- (E15) -- (D16) --  cycle;

\draw[fill=gray!30] (1.5, 1.5, 1.5) -- (G16) -- (D14) -- (1.5, 0, 1.5) -- cycle;

\draw[fill=gray!30] (1.5, 1.5, 1.5) -- (G11) -- (C13) -- (1.5, 0, 1.5) -- cycle;

\draw[fill=gray!30] (1.5, 1.5, 1.5) -- (1.5, 1.5, 0) -- (1.5, 4.5, 0)  -- (1.5, 4.5, 1.5) -- cycle;

\end{scope}

\begin{scope}[on background layer]
\draw[fill=gray!30] (D10) -- (E10) -- (F10) -- (G10) -- cycle;

\draw[fill=gray!30] (A6) -- (B6) -- (F6) -- (E6) -- cycle;
\end{scope}

\end{tikzpicture} \label{fig:mucubeFD}
}
\hspace{1cm}
\subfloat[A part of the Mucube formed by $\ZZ^3$ translates of the building block]{\includegraphics[scale=0.4]{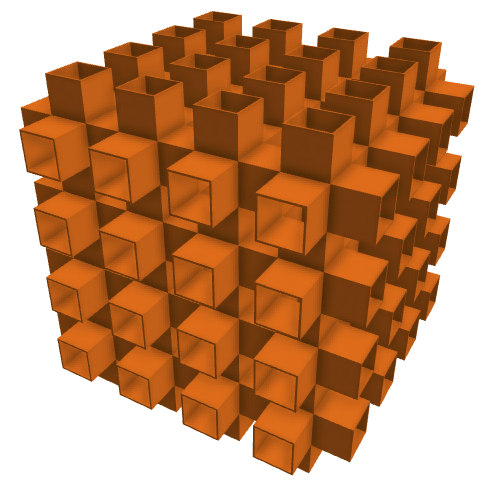}
}
\caption{The Mucube and its fundamental domain.}
\label{fig:MucubeConstruction}
\end{figure}
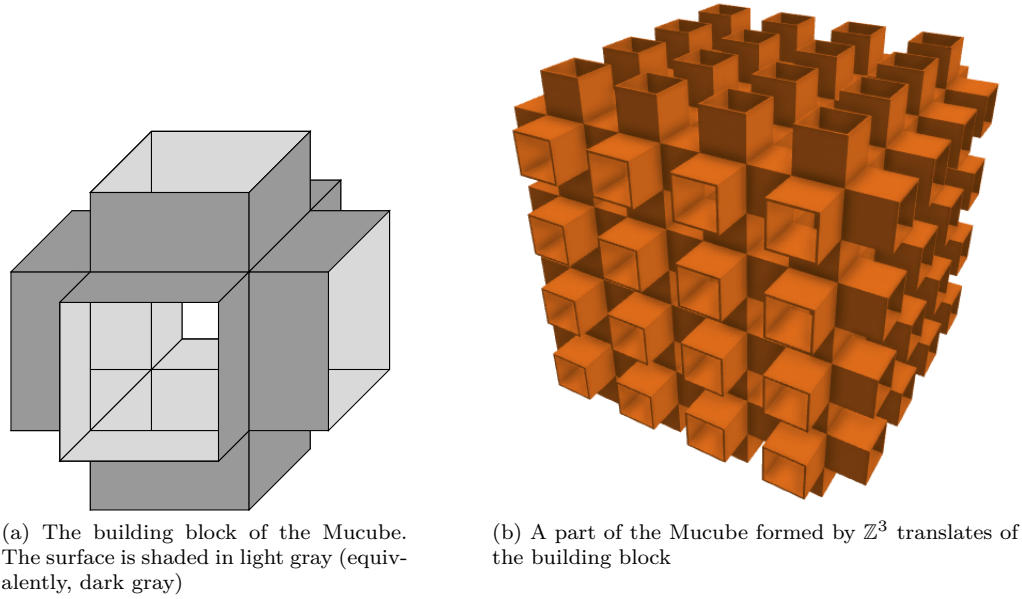

\subsection{Main results}

Given a point $x\in M$ and a direction $v\in T_x M$, we ask whether
the straight-line trajectory passing through $x$ in direction $v$ is
periodic. We specify $v$ by identifying the square face in which $x$
lies with the unit square $[0,1]^2\subset\RR^2$ and reporting the
components $v=(p,q)$ with respect to that identification. (It is a
consequence of the symmetries of $M$ that the answer to the question
of periodicity does not depend on which of the four natural
identifications we choose.) A moment's reflection reveals that a
trajectory can only be periodic if its slope $q/p$ is
rational. Furthermore, we will see that the question of periodicity
does not depend on the starting point $x$, as long as the forward and
backward trajectories do not encounter a corner
(Theorem~\ref{thm:cylinders}). The question, therefore, becomes: Which
integer vectors $(p,q)\in\ZZ^2$ (with relatively prime $x$ and $y$ coordinates) define periodic directions on $M$?

We give a complete algebraic characterization of the periodic
directions. In order to state it precisely, let
\begin{equation*}
  H := \left\langle\begin{bmatrix} 1 & 4 \\ 0 & 1\end{bmatrix}, \begin{bmatrix} 0 & -1 \\ 1 & 0\end{bmatrix}, \begin{bmatrix} 5 & -8 \\ 2 & -3\end{bmatrix}\right\rangle \leq \SL_2(\ZZ)
\end{equation*}
and 
\begin{equation*}
  \Gamma = \left\langle\begin{bmatrix}1 & 4 \\ 0 & 1\end{bmatrix}, \left\{h\begin{bmatrix} 0 & -1 \\ 1 & 0\end{bmatrix}h^{-1} : h \in H\right\} \right\rangle \leq H.
\end{equation*}
Given a matrix $N \in \SL_2(\ZZ)$, denote by $N \cdot (p,q)$ the
product $N\binom{p}{q}$.

\begin{restatable}[Characterization of periodic directions]{theorem}{veechgrpchar}\label{thm:veechgrpchar}
  Let $\Gamma\leq H \leq \SL_2(\ZZ)$ be as defined above.  A vector
  $(p,q)\in\ZZ^2$ with $\gcd(p,q) = 1$ defines a periodic direction in $M$ if and only if
  there exists $N \in \Gamma$ such that $N\cdot (1,0) = (p,q)$.

\end{restatable}

\begin{remark}
  It is not immediate from the algebraic characterization in Theorem
  \ref{thm:veechgrpchar} that $(p,q)$ defines a periodic direction in
  $M$ if and only if $(q,p)$ defines a periodic direction. However,
  this fact will be evident from the point of view of the rigid
  symmetries of the Mucube described in Section \ref{sec:rigidSym}.
\end{remark}

\begin{remark}
  Subsequently, the non-periodic directions defined by
  $(p,q) \in \ZZ^2$ will be called \emph{drift periodic directions}
  and the corresponding trajectories will be called \emph{drift
    periodic trajectories}. The reason for this is that in the
  mucube's embedding in $\RR^3$, any such trajectory is sent to itself
  by a translation---it drifts to infinity. (See
  Theorem~\ref{thm:cylinders}.) Drift trajectories project to periodic
  directions on the quotient surface pictured in \ref{fig:mucubeFD}.
\end{remark}

In addition to the characterization of periodic directions, we exhibit
a families of directions whose trajectories have nice properties such
as periodicity or recurrence (see Definition~\ref{def:recurrence}).

\begin{restatable}[Periodic/Recurrent Families]{theorem}{family}\label{thm:family}
  Suppose $\xi \in \RR$ has a finite or infinite continued fraction
  expansion of the form
  \begin{equation*} [4a_0; 4a_1, 4a_2, \dots, 4a_n, \dots ] := 4a_0 +
    \frac{1}{\displaystyle 4a_1 + \frac{1}{\displaystyle 4a_2
        +\frac{1}{\displaystyle \ddots \frac{1}{\displaystyle 4a_n + \frac{1}{\ddots}}}}},
   \end{equation*}
   where $a_0\in \ZZ$ and $a_i \in \ZZ\setminus\set{0}$.
   \begin{enumerate}[label=(\arabic*)]
   \item If the continued fraction
     is finite (i.e., $\xi=[4a_0: 4a_1, \dots, 4a_n]$), then every
     trajectory with slope $\xi$ is periodic. Moreover, the family of
     such trajectories contains members of arbitrarily large diameter
     (say, in the $\norm{\cdot}_\infty$-norm of $\RR^3$)
   
   \item \label{part:caseConePoint} If the continued fraction is infinite and
     $\lim a_na_{n+1} \neq -1$, then every trajectory with slope $\xi$
     emanating from a cone point is recurrent.
     
   \item \label{part:caseAll} If $\liminf a_n>0$ and $\limsup a_n >1$, then every
     cone-point avoiding trajectory with slope $\xi$ is recurrent.  If
     $\liminf\abs{a_n}>1$ and $\lim a_na_{n+1} \neq -4$, then every
     cone-point avoiding trajectory with slope $\xi$ is recurrent.
\end{enumerate}
\end{restatable}

\begin{remark}
  Part 2 applies when the sequence $(a_n)$ does not have a tail of
  alternating $1$s and $-1$s. Part 3 applies when either $(a_n)$ is
  eventually-positive and does not end in a tail of $1$s, or when
  $(\abs{a_n})$ is eventually $\geq 2$ and $(a_n)$ does not have a
  tail of alternating $2$s and $-2$s. Together, these conditions
  account for all but countably many possibilities. 
  
\end{remark}

Finally, as a consequence of the characterization of periodic slopes,
we are also able to understand the affine diffeomorhisms of the
Mucube. Given a half-translation surface, $S$, there is a well-defined
derivative map
$$D: \Aff(S) \rightarrow \PGL_2(\RR)$$
where $\Aff(S)$ is the group of affine diffeomorphisms of $S$. The
image $D(\Aff(S))$ is called the \emph{Veech group} of
$S$. Theorem~\ref{thm:veechgrpchar} yields the Veech group of $M$ as a
consequence.  For notational convenience, given a group
$G \leq \GL_2(\RR)$, we will denote by $\P G$ its projectivized image
in $\PGL_2(\RR)$.

\begin{restatable}[Veech group of $M$]{theorem}{veechgrpM}\label{thm:veechgrpM}
  The group $\P\Gamma \leq \PSL_2(\ZZ)$ is the Veech group of $M$.
\end{restatable}

\begin{remark}
  Athreya--Lee \cite{athlee} show that $H$, as defined above, is the
  Veech group\footnote{Some authors define the Veech group to be a
    subgroup of $\GL_2(\RR)$ while others take the projectivized image
    in $\PGL_2(\RR)$.} of the translation cover of the genus 3
  quotient of the Mucube by $\ZZ^3$. Since $\Gamma$ is infinite index
  in $H$, Theorem \ref{thm:veechgrpM} shows, in particular, that there
  are many elements of the Veech group of the genus 3 quotient that do
  not lift to the Veech group of $M$.
\end{remark}

Additionally, we show that the set of periodic directions is dense in
$S^1$.

\begin{restatable}[Density of periodic directions]{theorem}{densityofperiodics}\label{thm:density}
  The set of periodic directions of the Mucube is dense in the set of
  all directions.
\end{restatable}

The Mucube is a Hooper--Thurston--Veech surface, see
\cite{Hooper-Thurston-Veech}. Together with Theorem~\ref{thm:density}
and a recent computation of Hausdorff dimension of limit sets of Hecke
triangle groups by Fedosova \cite{Fedosova}, we obtain the following
corollary.

\begin{restatable}[Density of ergodic directions]{corollary}
  {ergodicdirections}\label{cor:ergodic directions}
  The Lebesgue measure is ergodic for the straight line flow in a
  dense set of directions with Hausdorff dimension greater than
  $0.68$.
\end{restatable}

\subsection{Analogous Results in Other Settings}

The majority of our main theorems about the Mucube pertain to
characterizing the periodic trajectories and periodic directions for the linear flow. We note that to date, approaches individualized to small families of surfaces have been required. Several other papers exemplify this. We have already mentioned the wind-tree model, and we share other notable examples below.

In 2017, Davis--Dods--Traub--Yang \cite{DDTY17} studied the existence
of periodic linear trajectories on platonic solids that start and end on the same vertex
without passing through any other vertices (also known as \emph{closed
  saddle connections}), proving non-existence of
such trajectories for the tetrahedron and the cube. Fuchs \cite{Fuchs}
proved that these trajectories exist neither on the octahedron nor the
icosahedron and conjectured their existence on the
dodecahedron. Indeed, Athreya--Aulicino \cite{AthAul} constructed a
closed saddle connection on the dodecahedron.  Subsequently,
Athreya--Aulicino--Hooper \cite{AthAulHoop} found 31 equivalence
classes of closed saddle connections on the dodecahedron, using the
affine diffeomorphism group of the translation cover of the
dodecahedron.

\begin{figure}[h!]
    \centering
    \includegraphics[scale=0.3]{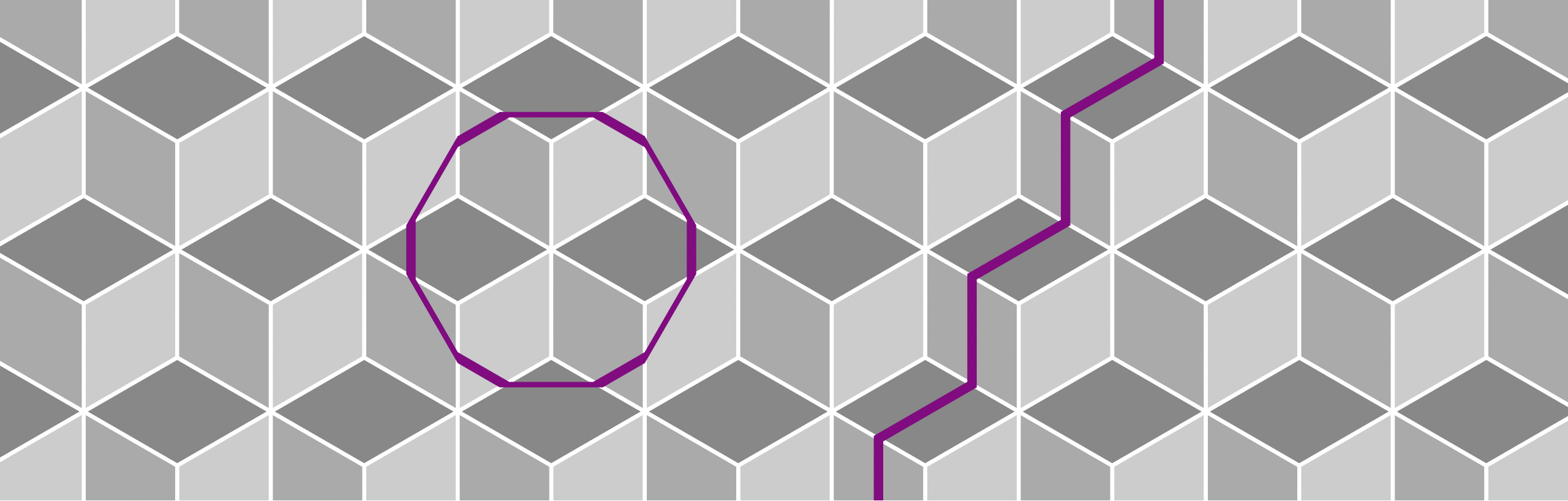}
    \caption{A part of the Necker Cube surface with periodic and drifting trajectories. Figure taken from \cite{HoopJav}.}
    \label{fig:neckercube}
\end{figure}

For the infinite staircase (see Figure~\ref{fig:infinitestaircase}), Hooper--Hubert--Weiss \cite{hoophubweiss}
proved that a direction with slope $\alpha \in \RR$ for the linear
flow is \emph{completely periodic} (i.e. the surface completely
decomposes into cylinders) if and only if
$\alpha = \frac{p}{q} \in \QQ$ with $\gcd(p,q)=1$ and one of $p$ or
$q$ is even. For the infinite translation surface obtained from
unfolding the standard Ehrenfest wind-tree billiard (with
$\frac{1}{2} \times \frac{1}{2}$ square obstacles),
Hubert--Leli{\`e}vre-Troubetzkoy \cite{hubleltrou} proved that a
direction with slope $\alpha \in \RR$ for the linear flow is
completely periodic if and only if $\alpha = \frac{p}{q} \in \QQ$ for
two odd integers $p$ and $q$. More recently, Hooper--Javornik
\cite{HoopJav} studied the \emph{Necker cube surface}, a periodic
surface built from unit squares in 3-space that are arranged so that
the squares are parallel to the coordinate planes in the pattern
pictured in Figure \ref{fig:neckercube}. They proved that a direction
with slope $\alpha \in \RR$ for the linear flow is completely periodic
if and only if $\alpha = \frac{p}{q} \in \QQ$ for two odd integers $p$
and $q$. Likewise, in forthcoming work, Fairchild--Lee--Shrestha study
the linear flow on the Mutetrahedron, another triply periodic surface
embedded in $\RR^3$.

\subsection{Further remarks/questions on the Mucube}

Our main results concern periodicity of the linear flow. However, we
have more questions, particularly concerning the non-periodic
directions of the linear flow.

We start with some remarks that dispose of a few of these questions
easily. The first of these concerns unboundedness of non-periodic
orbits.

\begin{remark}[Bounded orbits are periodic]\label{rem:boundedorbits}
  In the Mucube, every bounded linear orbit is periodic. To see this
  suppose $\calO$ is a bounded orbit. Then it is contained in some
  connected finite union $U$ of translated copies of the fundamental
  domain, and furthermore we may suppose that $U$ is large enough that
  $\calO$ is not dense in $U$. Now, by identifying opposite open ends
  of $U$, we obtain a compact surface $U/\sim$ tiled by squares. Such
  surfaces are known to be a Veech surfaces. Since $\calO$ is a
  non-dense orbit in $U/\sim$, the Veech dichotomy implies that it
  orbit must be periodic.
\end{remark}

Since non-periodic orbits are necessarily unbounded, it is a natural
next question to ask whether non-periodic orbits can be dense.

\begin{remark}[Existence of dense trajectories]
  It follows from density of periodic directions and machinery
  developed in \cite{Hooper-Thurston-Veech} that there is a dense set
  of directions with Hausdorff dimension at least $0.68$ in which the
  straight-line flow is ergodic with respect to Lebesgue (see
  Corollary~\ref{cor:ergodic directions}). Trajectories in these
  directions are dense. Furthermore, we know that the straight-line
  flow is not uniquely ergodic in these directions. See Theorem G.3
  from \cite{Hooper-Thurston-Veech}.
\end{remark}

Theorem~\ref{thm:density} and Corollary~\ref{cor:ergodic directions}
give existence of dense sets of periodic and ergodic directions
respectively, but they do not illuminate the behavior of a typical
trajectory in a typical direction. We pose this as a question.

\begin{question}\label{ques:typical}
    What is the typical behavior of a straight-line trajectory?
\end{question}

We note that Avila--Hubert \cite{avihub} and Fr\k{a}czek--Hubert
\cite{frahub} answer Question~\ref{ques:typical} for windtree
billiards (and generalized versions) by showing that in almost every
direction almost every trajectory is recurrent. Also for windtree
billiards, Delecroix \cite{del} give explicit examples of divergent trajectories and
subsequently Delecroix-Hubert-Leli\`evre~\cite{delhublel} compute
diffusion rates. It is natural to ask for such results in the context
of the Mucube as well:

\begin{question}
  Are there explicit examples of divergent, non-drift trajectories? What is the diffusion rate for the Mucube?
\end{question}

Relatedly, one also wonders what behavior non-dense, non-periodic
trajectories exhibit. For instance:

\begin{question}
  If a non-periodic trajectory is not dense, is it also nowhere dense?
\end{question}

Another question that has drawn recent attention is that of illumination. Given points $x$ and $y$ on a translation surface $S$, we say that $y$ is \emph{illuminated by $x$} if there exists a
straight-line trajectory from $x$ to $y$. Leli\`evre--Monteil--Weiss \cite{lelmontweiss} proved that for a
compact translation surface $S$ and any point $x \in S$, the set of
points which are not illuminated by $x$ is finite. Any point not illuminated by $x$ has infinitely many lifts to any periodic cover of $S$, and those lifts are non-illuminated by any lift of $x$. In Remark~\ref{rem:nonillumination} we give examples of points $x\in M$ that fail to illuminate infinitely many points of the Mucube.

\subsection{Structure}

We divide this paper into two parts. In
Part~\ref{part:geometricCharacterization}, we consider the natural
embedding of the Mucube in $\RR^3$ to obtain a geometric
characterization of periodic directions on the Mucube.  We start by by precisely expressing $M$ as a subset of $\RR^3$ in  Section~\ref{sec:subset} and establishing the notion of a parallel field on $M$ in Section~\ref{sec:parallel-fields-m}. We then proceed using the group of rigid symmetries of $\RR^3$ that preserve the
Mucube (see Section~\ref{sec:rigidSym}) to show that in a given
rational direction, the Mucube decomposes completely into isometric
cylinders or strips, and we obtain geometric information about these
cylinders/strips (see Theorem~\ref{thm:cylinders} in
Section~\ref{sec:cylinderDecomp}). Next, working towards a
characterization of periodic directions, we introduce two finite
quotient surfaces $X$ and $Y$ (Sections~\ref{sec:descentX} and
\ref{sec:descentToY}) and show that the periodic directions on the
Mucube are characterized by certain geometric conditions on their
images in $X$ and $Y$  (see Theorems~\ref{thm:displacementvector} and~\ref{thm:characterization} in
Section~\ref{sec:cylinderDecomp}). In Section~\ref{sec:fourey} we
give explicit examples and constructions of trajectories with various
geometric properties. In particular we introduce an operation between
periodic orbits that allows us to prove Theorem~\ref{thm:family}. This
concludes Part~\ref{part:geometricCharacterization}.

In Part~\ref{part:algebraicCharacterization}, we view the Mucube as a
square-tiled half-translation surface and use this structure coupled
with a homological approach to give a complete algebraic
characterization of the periodic directions. We start by introducing
preliminaries on translation surfaces, their Veech groups and
homology, and Fuchsian groups and their limit sets in
Section~\ref{sec:bg}. In Section~\ref{sec:MucubeXYHalftranslation} we
endow the Mucube and its quotients $X$ and $Y$ with half-translation
structures and in Section~\ref{sec:MucubeFundamentalGroup} we obtain
the fundamental group of the Mucube. In Section~\ref{sec:veechgrpchar}
we prove Theorem~\ref{thm:veechgrpchar}, a complete algebraic
characterization of the periodic directions on the Mucube, using a
Fuchsian group within $\SL_2(\ZZ)$. Subsequently, in
Section~\ref{sec:veech} we use Theorem~\ref{thm:veechgrpchar} to
obtain the Veech group of the Mucube (see
Theorem~\ref{thm:veechgrpM}). Finally, in Section~\ref{sec:density},
we use the algebraic characterization of periodic directions to prove
density of the set of periodic directions (see
Theorem~\ref{thm:density}) and, consequently, density of the set of
ergodic directions (see Corollary~\ref{cor:ergodic directions}). This
concludes the paper.

\subsection*{Acknowledgements}

We are deeply indebted to Moon Duchin for organizing the Polygonal
Billiards Research Cluster held at Tufts University in 2017 where this
project began and during which time our work was supported by Moon's
NSF grant number DMSCAREER-1255442. We thank Pat Hooper for
introducing us to the Mucube. We would also like to thank Barak Weiss,
Emily Stark, Angel Pardo, Marissa Loving, and Jane Wang for helpful
conversations at various stages of the project. S.S. was supported for
this research by an AMS-Simons Research Enhancement Grant for
Primarily Undergraduate Institution Faculty.

\part{Geometric Characterization}\label{part:geometricCharacterization}

We derive a series of results on the geometry of trajectories on $M$
by leveraging the symmetries of the natural embedding of $M$ in
$\mathbb{R}^3$. Taking quotients of $M$ by suitable subgroups of those
symmetries, we obtain compact, square-tiled surfaces. All trajectories
on $M$ \textit{descend} to trajectories on these quotients, and we
characterize periodic linear trajectories on $M$ by the behavior of
their descents.

Continuing this embedding and symmetry-based approach, we define a
process through which two given periodic directions give rise to a
third. Using this process, we produce a family of periodic
trajectories containing members of arbitrarily large diameter and a
family of recurrent trajectories with Hausdorff dimension
approximately 0.68.

\section{The Mucube $M$ as a subset of $\RR^3$}\label{sec:subset}

The Mucube $M$ has a natural embedding in $\RR^3$ where its faces are
all parallel to coordinate planes. Under this embedding, $M$ casts a
shadow on each coordinate plane that resembles a checker board. To be
precise, let
\begin{equation*}
  C = \overline{\set*{(x,y)\in\RR^2 : \floor{x-1/2} \textrm{ and } \floor{y-1/2} \textrm{ have opposite parity}}}.
\end{equation*}
Then we may take the Mucube to be
\begin{equation*}
  M = \set{(x,y,z)\in \RR^3 : (x,y)\in C, (x,z)\in C, \textrm{ and } (y,z)\in C}. 
\end{equation*}
Now let
\begin{equation*}
  F = M \cap \cube,
\end{equation*}
recalling that $\cube = [-1, 1]^3$. Then we have
\begin{equation*}
  M = \bigcup_{\bt\in\ZZ^3} F_\bt
\end{equation*}
where $F_\bt := F + 2\bt$. In this embedding, $\ZZ^3$ acts on $M$ by
translations $\bv\mapsto \bv + 2\bt$, and $F_\bzero = F$ is a
fundamental domain for the action.

\section{Parallel fields on $M$}\label{sec:parallel-fields-m}

The Mucube, $M$, is flat except for at its cone points---points where
six square faces meet at a shared corner. Let
\begin{equation*}
  \calC = \set*{\textrm{cone points of $M$}}.
\end{equation*}
and
\begin{equation*}
  M_0 = M - \calC.
\end{equation*}
Each cone point of $M$ has angle $3\pi$.  Therefore, if $x\in M_0$,
then the parallel transport of a vector $v\in T_x M$ around a path in
$M_0$ enclosing a cone point is $-v\in T_x M$, and the parallel
transport around an arbitrary closed path in $M_0$ is
$\pm v \in T_x M$. Consequently, there are no continuous parallel
vector fields on $M_0$, but there are parallel direction fields. Every
vector $(x,v)\in TM$ defines a parallel direction field $V$ on $M_0$,
by parallel transport. For each $y\in M_0$, $V(y)$ is the parallel
transport of $(x,v)$ to $y$, viewed as representing an element of the
projectivized unit tangent plane $\PP(T_y^1M)$.

In turn, a parallel field $V$ on $M_0$ determines a lamination of
$M_0$ by linear trajectories in direction $V$. That is,
\begin{equation*}
  M_0 = \bigcup_{x\in M} \calO_{x,V},
\end{equation*}
where $\calO_{x,V}$ is the linear trajectory passing through $x$ in
direction $V$. Note that it is possible for such a trajectory to have
a finite past or future, ending at cone points.

\section{Rigid symmetries of $M$}
\label{sec:rigidSym}

The Mucube has a rich collection of rigid symmetries, already apparent
when looking at Figure~\ref{fig:mucubeFD}(b). Its realization as
\begin{equation*}
  M = \bigcup_{\bt\in\ZZ^3} F_\bt
\end{equation*}
shows that $\ZZ^3$ acts on $M$ by rigid motions: the element
$\bt\in\ZZ^3$ acts by translating by $2\bt$. Let
$O\subset \mathrm{SO}(3)$ be the octahedral group---the group of
rotational symmetries of the cube $\cube$. We quickly see that $O$
also acts by rigid motions on $M$ in its natural embedding in $\RR^3$.

For $\theta\in O$ and $\bt\in\ZZ^3$, the composite $\bt\circ\theta$ acts
on $M$ by first rotating by $\theta$ and then translating by $2\bt$. The
acting group here is the semidirect product $G = \ZZ^3\rtimes O$ with
multiplication
\begin{equation*}
  (\bt_1, \theta_1)\cdot(\bt_2, \theta_2) = (\bt_1 + \theta_1 \bt_2,\theta_1\theta_2).
\end{equation*}
There is a convenient embedding $G\hookrightarrow \SL(4,\RR)$, defined by
\begin{equation*}
  (\bt, \theta) \mapsto
  \begin{bmatrix}
    \theta & 2\bt \\
    \bzero^T & 1
  \end{bmatrix}
\end{equation*}
and action by matrix multiplication, with $\RR^3$ embedded in $\RR^4$
by $\bv \mapsto (\bv, 1)$.

\begin{lemma}\label{lem:periodicrotations}
  The torsion elements of $G$ are rotations about lines in $\RR^3$
  having order no greater than $4$. If $\ell\subset\RR^3$ is the axis for a $\pi$ rotation in $G$, then $\ell\cap M$ is either empty, or it consists entirely of center points of square faces of $M$, or it consists entirely of midpoints of edges of square faces of $M$.
\end{lemma}

\begin{proof}
  Suppose
  \begin{equation*}
    g = \begin{bmatrix}
      \theta & 2\bt \\
      \bzero^T & 1
    \end{bmatrix}\in G
  \end{equation*}
  is an element of order $k$. Since 
\begin{equation*}
    g^k =   \begin{bmatrix}
      \theta & 2\bt \\
      \bzero^T & 1
    \end{bmatrix}^k = 
    \begin{bmatrix}
      \theta^k & (1 + \theta + \dots + \theta^{k-1})2\bt \\
      \bzero^T & 1
    \end{bmatrix},
  \end{equation*}
  we must have $\theta^k = I_3$, the identity, and we must have
  \begin{equation}\label{eq:btcycle}
    (1 + \theta + \dots + \theta^{k-1})\bt = \bzero.
  \end{equation}
  Therefore $2\bt$ lies in the plane that passes through the origin and is perpendicular to the axis of
  rotation of $\theta$  and so $g$ acts as a rotation 
  over some line in $\RR^3$ which is parallel to the axis of rotation
  of $\theta$. Furthermore, if $k'\geq 1$ is any other integer such
  that $\theta^{k'}=I_3$, then~(\ref{eq:btcycle}) holds with $k'$ in
  place of $k$, and so we will have $g^{k'}=I_4$. This means that the
  order of $g$ coincides with the order of $\theta$. Since all elements
  of $O$ have order $1, 2, 3$, or $4$, we are done with the first assertion.

  Now, suppose that $(\theta, \bt)$ is a $\pi$ rotation with $\bt=(m,n,k)$ and let $\ell$ denote its axis of rotation. If
  \begin{equation*}
      \theta = \begin{bmatrix}
          -1 & 0 & 0 \\
          0 & -1 & 0 \\
          0 & 0 & 1
      \end{bmatrix},
  \end{equation*}
  then $k$ must be $0$, and $\ell = \set{(m,n,z): z\in \RR}$. We have $\ell\cap M = \emptyset$ if $m\equiv n\pmod 2$, and $\ell\cap M = \set{(m,n,z) : z\in 1/2 + \ZZ}$ if $m\not\equiv n\pmod 2$. The latter set consists of centers of square faces of $M$. If
    \begin{equation*}
      \theta = \begin{bmatrix}
          0 & 1 & 0 \\
          1 & 0 & 0 \\
          0 & 0 & -1
      \end{bmatrix},
  \end{equation*}
  then we must have $m=-n$, and  $\ell=\set{(x+m, x-m, k) : x\in \RR}$ and $\ell\cap M = \set{(x+m, x-m, k) : x\in \ZZ}$, which consists entirely of midpoints as claimed.

  Every other $\pi$ rotation is treated similarly, with coordinates' roles interchanged.
\end{proof}

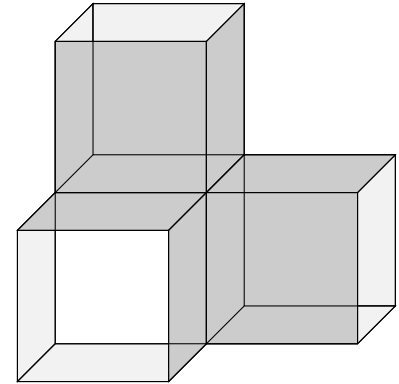
\begin{wrapfigure}[15]{r}{0.4\textwidth}
  \begin{center}
      \begin{tikzpicture}[sq/.style=
  {shape=regular polygon, regular polygon sides=4, draw, minimum width=2.828cm}]


\draw[fill=gray!10] (-1,-1) -- (-1, 1) -- (-1+0.5, 1+0.5) -- (-1+0.5, -1+0.5)--cycle; 
\draw[fill=gray!10] (1,-1) -- (1+0.5, -1+0.5) --(-1+0.5,-1+0.5) -- (-1,-1) --cycle; 
\draw[fill=gray!40]  (-1, 1) -- (-1+0.5, 1+0.5) -- (1+0.5, 1+0.5) -- (1,1) --cycle; 
\draw[fill=gray!40] (1,1) -- (1+0.5, 1+0.5) --(1+0.5,-1+0.5) -- (1,-1) --cycle; 



  \node[sq] at (0,0){};
  \node[sq] at (0.5,0.5){};
  \draw[thin] (-1,-1) -- ++(0.5, 0.5);
  \draw[thin] (-1,1) -- ++(0.5, 0.5);
  \draw[thin] (1,-1) -- ++(0.5, 0.5);
  \draw[thin] (1,1) -- ++(0.5, 0.5);
  
\begin{scope}[shift = {(0.5, 2.5)}]

\draw[fill=gray!10] (-1,-1) -- (-1, 1) -- (-1+0.5, 1+0.5) -- (-1+0.5, -1+0.5)--cycle; 
\draw[fill=gray!10]  (-1+0.5, -1+0.5) -- (-1+0.5, 1+0.5) -- (1+0.5, 1+0.5) -- (1+0.5,-1+0.5) --cycle; 
\draw[fill=gray!40]  (-1, -1) -- (-1, 1) -- (1, 1) -- (1,-1) --cycle; 
\draw[fill=gray!40] (1,1) -- (1+0.5, 1+0.5) --(1+0.5,-1+0.5) -- (1,-1) --cycle; 

 \node[sq] at (0,0){};
  \node[sq] at (0.5,0.5){};
  \draw[thin] (-1,-1) -- ++(0.5, 0.5);
  \draw[thin] (-1,1) -- ++(0.5, 0.5);
  \draw[thin] (1,-1) -- ++(0.5, 0.5);
  \draw[thin] (1,1) -- ++(0.5, 0.5);
\end{scope}

\begin{scope}[shift = {(2.5, 0.5)}]


\draw[fill=gray!10]  (-1+0.5, -1+0.5) -- (-1+0.5, 1+0.5) -- (1+0.5, 1+0.5) -- (1+0.5,-1+0.5) --cycle; 
\draw[fill=gray!10] (1,-1) -- (1+0.5, -1+0.5) --(-1+0.5,-1+0.5) -- (-1,-1) --cycle; 
\draw[fill=gray!40]  (-1, 1) -- (-1+0.5, 1+0.5) -- (1+0.5, 1+0.5) -- (1,1) --cycle; 

\draw[fill=gray!40]  (-1, -1) -- (-1, 1) -- (1, 1) -- (1,-1) --cycle; 


 \node[sq] at (0,0){};
  \node[sq] at (0.5,0.5){};
  \draw[thin] (-1,-1) -- ++(0.5, 0.5);
  \draw[thin] (-1,1) -- ++(0.5, 0.5);
  \draw[thin] (1,-1) -- ++(0.5, 0.5);
  \draw[thin] (1,1) -- ++(0.5, 0.5);
\end{scope}

\end{tikzpicture}
\caption{The shifted fundamental domain $F'$}
\label{fig:shiftedF}
\end{center}
  \end{wrapfigure}

The following lemma states that we can map any face of $M$ onto any
other face of $M$ with a rigid motion of $M$, and we can do the
same with the cone points. 

\begin{lemma}\label{lem:parallelism}~
  \begin{enumerate}
  \item The action of $G$ is transitive on the set of faces of $M$;
    the stabilizer of every face of $M$ under the action of $G$ has
    order $2$ and is generated by the $\pi$-rotation about the normal
    line through the center of that face. 
  \item The action of $G$ is transitive on the set of cone points of
    $M$. If $p\in F_\bt$ is a cone point, then its stabilizer under
    the action of $G$ has order $3$ and is generated by a
    $2\pi/3$-rotation about the diagonal in $\cube+ 2\bt$ containing
    $p$.
  \item Parallel fields on $M$ are fixed by the action of $G$.
  \end{enumerate}
\end{lemma}

\begin{proof}
  1. We may take
  \begin{equation*}
    F' = M\cap\brackets*{-\frac{1}{2},\frac{3}{2}}^3
  \end{equation*}
  as a fundamental domain for the $\ZZ^3$-action on $M$. (See
  Figure~\ref{fig:shiftedF}.)

  Each face of $M$ is uniquely identified with
  one of the faces of $F'$ by a translation. In turn, any face of $F'$
  can be mapped to any other face of $F'$ by an element of $O$. (One
  applies
  \begin{equation}\label{eq:xyzcycle}
    \theta_{2\pi/3} = \begin{bmatrix}
      0 & 1 & 0 \\
      0 & 0 & 1 \\
      1 & 0 & 0 
    \end{bmatrix}
  \end{equation}
  as needed, followed by some power of either
  \begin{equation}\label{eq:quarterturns}
    \theta_x = \begin{bmatrix}
      1 & 0 & 0 \\
      0 & 0 & -1 \\
      0 & 1 & 0 
    \end{bmatrix},\quad
    \theta_y = \begin{bmatrix}
      0 & 0 & -1 \\
      0 & 1 & 0 \\
      1 & 0 & 0 
    \end{bmatrix},\quad\textrm{or}\quad
    \theta_z  = \begin{bmatrix}
      0 & -1 & 0 \\
      1 & 0 & 0 \\
      0 & 0 & 1 
    \end{bmatrix},
  \end{equation}
  the quarter-turns.) This shows that $G$ acts transitively on the set
  of faces of $M$. The stabilizer of a face of $M$ must be a subgroup
  of the cyclic group of order $4$ (the orientation-preserving
  symmetries of a square), and one sees by inspection that it consists
  precisely of the identity and the $\pi$-rotation about the normal
  line to the square at its center.

  2. Taking now
  \begin{equation*}
    F = M\cap\brackets*{-1, 1}^3
  \end{equation*}
  as a fundamental domain for the $\ZZ^3$-action on $M$, as in
  Figure~\ref{fig:mucubeFD}, one sees that each cone point of $M$ is
  uniquely identified with one of the cone points of $F$ by a
  translation (by $-2\bt$ if the initial cone point lies in
  $F_\bt$). In turn, any cone point of $F$ can be mapped to any other
  cone point of $F$ by an element of $O$. (One may apply the quarter
  turns $\theta_x, \theta_y, \theta_z$ as needed.) This shows that $G$
  acts transitively on the set of cone points of $M$. The stabilizer
  of a cone point of $M$ is a conjugate of the stabilizer of the cone
  point of $F$ to which it is identified by the $\ZZ^3$-action. And
  this is precisely the stabilizer of a cube's corner in the
  octahedral group $O$. It has order $3$ and is generated by one of
  \begin{equation*}
    \theta_{2\pi/3},\quad \theta_z\theta_{2\pi/3}\theta_z^{-1},\quad \theta_z^2\theta_{2\pi/3}\theta_z^{-2},\quad\textrm{or}\quad \theta_z^3\theta_{2\pi/3}\theta_z^{-3},
  \end{equation*}
  the $2\pi/3$-rotations about diagonals of the cube $\cube$.

  3. Suppose $V$ is a parallel field on one of the rectangular faces
  of $F$, and consider its extension by parallel transport to an
  adjacent rectangular face. Note that there is a composition of quarter turns taking the first rectangular face to the second. Moreover, the extension of $V$ to the second face coincides with
  its image under the composition. Since $\theta_x, \theta_y, \theta_z$
  generate $O$, we have that $O$ fixes the parallel extension of $V$
  to all of $F$. Note that $M$ is obtained by gluing translated copies
  of $F$ along edges that lie in flat components of $M$. The parallel
  extension of $V$ (the parallel field on $F$) across such an edge
  coincides with its image under the corresponding
  translation. Therefore, the parallel extension of $V$ to all of $M$
  is fixed under the $\ZZ^3$-action, hence under $G$.
\end{proof}

\subsection{Orientation-reversing rigid motion and reflections}

The group $G$ is not the full set of rigid motions of $M$, but rather,
the orientation-preserving subgroup of the full group. Every element
of $G$ sends both connected components of $\RR^3\setminus M$ to
themselves: if they are designated the ``outside'' and the ``inside''
of $M$, then the outside goes to the outside and the inside goes to
the inside. On the other hand, the translation in $\RR^3$ by
$(1,1,1)$---which is not an element of $G$---maps $M$ to $M$ rigidly
and interchanges the inside and outside of $M$, thereby reversing
orientation. One way to phrase this is that the Mucube looks the same
from the inside as it does from the outside.

\begin{lemma}\label{lem:111}
  The translation in $\RR^3$ by $(1,1,1)$ sends any parallel direction
  field $V$ on $M_0$ to its perpendicular field $V^\perp$.
\end{lemma}

\begin{proof}
  The proof is a simple inspection. Pick an arbitrary square face
  $\calF\subset M$. Its image under the translation $h$ by $(1,1,1)$
  is another square face $h\calF = \calF'\subset M$. By
  Lemma~\ref{lem:parallelism}, there is a rigid motion of $g\in G$
  such that $g\calF = \calF'$. The reader will note that the
  restriction of $h$ to $\calF$ is equal to the composition of $g$
  with a quarter turn of $\calF'$. Therefore, since $g$ preserves
  parallel fields, a vector in $(x,v)\in T\calF$ is sent by $h$ to a
  vector that is perpendicular to the parallel field defined by
  $(x,v)$.
\end{proof}

It is also worth noting that reflection over any plane parallel to the
one of the coordinate planes of the form $x=2k$, $y=2k$ or $z=2k$ for
$k \in \ZZ$ is an orientation-reversing rigid isometry of $M$.

\section{Unfolding, strips, and cylinders}
Let $x\in M_0$ and $V$ a direction field on $M_0$. Suppose the face
$\mathcal F\subset M$ containing $x$ is given (one of) the natural
isometric coordinate chart $\psi: \calF \to [0,1]^2\subset\RR^2$. Then
$\psi(\calF\cap \calO_{x,V})$ is a line segment in
$[0,1]^2\subset\RR^2$, which we may extend linearly in both directions
either indefinitely or until it encounters $\ZZ^2\subset\RR^2$. Denote
by $\calL$ this extended line, and note that the parametrizing map
$\psi^{-1}:[0,1]^2\to\calF$ can be extended locally isometrically to
the union of unit squares through which $\calL$ passes. Let us denote
that extension $\phi$, so that $\phi\mid_{[0,1]^2} = \psi^{-1}$. Note
that $\phi(\calL) = \calO_{x,V}$. In this way, we have unfolded square
faces of $M$ through which $\calO_{x,V}$ passes. (See
Figure~\ref{fig:unfolding}.)

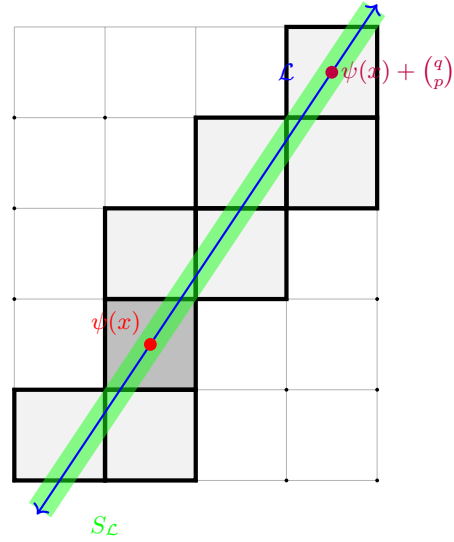
\begin{wrapfigure}[20]{r}{0.4\textwidth}
\vspace{-1cm}
  \centering
  \begin{tikzpicture}[scale=1.2, every node/.style={font=\small}]
    \def\x{0.5}   
    \def\y{0.5}   
    \def\slope{1.5}

    \draw[step=1, line width=0.3pt, gray!60] (-1,-1) grid (3,4);

    
    \fill[gray!10] (-1,-1) rectangle (0,0);
    \fill[gray!10]  (0,-1) rectangle (1,0);
    \fill[gray!50] (0,0) rectangle (1,1);
    \fill[gray!10]  (0,1) rectangle (1,2);
    \fill[gray!10]  (1,1) rectangle (2,2);
    \fill[gray!10]  (1,2) rectangle (2,3);
    \fill[gray!10]  (2,2) rectangle (3,3);
    \fill[gray!10]  (2,3) rectangle (3,4);

    \draw[ultra thick] (-1,-1) rectangle (0,0);
    \draw[ultra thick]  (0,-1) rectangle (1,0);
    \draw[ultra thick] (0,0) rectangle (1,1);
    \draw[ultra thick]  (0,1) rectangle (1,2);
    \draw[ultra thick]  (1,1) rectangle (2,2);
    \draw[ultra thick]  (1,2) rectangle (2,3);
    \draw[ultra thick]  (2,2) rectangle (3,3);
    \draw[ultra thick]  (2,3) rectangle (3,4);

    \foreach \i in {-1,0,1,2,3}
    \foreach \j in {-1,0,1,2,3}
    \fill (\i,\j) circle (0.6pt);


    \node[blue] at (2, 3.5) {$\calL$};
    \node[green] at (0, -1.5) {$S_\calL$};

    \draw[preaction={draw=green!90, line width=10pt, opacity=0.5, -, shorten >=2pt, shorten <=2pt}, blue, thick, <->]
    (-0.75, { \y + \slope*(-0.75 - \x) })
    -- ( 3, { \y + \slope*( 3 - \x) });

    \coordinate (x) at (\x,\y);
    \fill[red] (x) circle (2pt) node[above left] {$\psi(x)$};

    \fill[purple] (2.5, 3.5) circle (2pt) node[right] {$\psi(x) + \binom{q}{p}$};

  \end{tikzpicture}
  \caption{The unfolding $\calL$ of a linear trajectory
    $\calO\subset M$ and the square faces through which $\calO$
    passes. In green, the maximal $\ZZ^2$-avoiding strip $S_\calL$
    containing $\calL$. The darker gray square is
    $[0,1]^2\subset\RR^2$.}
  \label{fig:unfolding}
\end{wrapfigure}

Notice that if $\calO_{x,V}$ is periodic, then the image of $\calL$ in $\RR^2/\ZZ^2$ is a closed
linear trajectory in the torus, that is, there exist integers $(q,p)$ with $\gcd(p,q)=1$ such that $\calL$ is mapped to
itself by the translation
\begin{equation}\label{eq:Tpsi}
  T_\psi(v) = v + \binom{q}{p}. 
\end{equation}
Note that $\binom{q}{p}$ is determined by $\psi$ up to sign, and is
determined by $V$ up to permutation of coordinates. This implies the
following basic fact about trajectories on the Mucube:

\begin{proposition}\label{prop:PeriodicHasToBeRational}
    If $\calO$ is a periodic trajectory on the Mucube, then the slope of $\calO$ is rational. 
\end{proposition}

Suppose now that $\calO:=\calO_{x,V}$ avoids cone points and has
rational slope, in the sense that $\calL$ has rational slope or is
vertical in $\RR^2$. (Note that whether the slope of $\calL$ is
rational is independent of which of the natural identifications
$\calF \sim [0,1]^2$ we started with.) Then $\calL$ extends
bi-infinitely and stays a positive distance from $\ZZ^2$. For
$a\leq 0\leq b$, define the strip
\begin{equation*}
  S_{[a,b]} := \bigcup_{t\in[a,b]} \set*{\calL + (t,0)}.
\end{equation*}

Since $\calL$ stays a positive distance from $\ZZ^2$, it is contained
in a strip
\begin{equation}\label{eq:S_L}
  S_\calL := \overline{\bigcup_{\substack{a\leq 0 \leq b \\ S_{[a,b]}\subset \RR^2\setminus\ZZ^2}}S_{[a,b]}}
\end{equation}
which avoids $\ZZ^2$ except on its boundary. (Pictured in green in 
Figure~\ref{fig:unfolding}.) The image
\begin{equation*}
  S_\calO :=\phi(S_\calL)
\end{equation*}
is either a strip or a cylinder in $M$ embedded in $M$ by folding, and
meeting cone points on (and only on) the image of the boundary
$\phi(\partial S_\calL)$. We will call $S_\calO$ the maximal strip or
maximal cylinder containing the trajectory $\calO$.

\begin{remark}\label{ontotorus}
  Notice that $S_\calL$ covers $\RR^2/\ZZ^2$ under the quotient map
  $\RR^2\to \RR^2/\ZZ^2$. The two boundary components of $S_{\calL}$
  are sent to a single closed linear trajectory in the torus
  $\RR^2/\ZZ^2$ in the direction of $\calL$, passing through $(0,0)$.
\end{remark}

The following lemma introduces two rigid key motions of $M$ that arise from orbits in a given rational direction. The first of which is intertwined with $T$ via $\phi$. The second is intertwined, via $\phi$, with an involution $\iota$ that preserves the $\ZZ^2$ lattice and a given  maximal strip. 

\begin{lemma}\label{lem:intertwining}
  Let $V$ be a parallel field on $M_0$ with rational direction and
  $\calO$ a cone-point avoiding linear trajectory along $V$ passing
  through the point $x$.
\begin{enumerate}
\item Let $\psi$ and $T:=T_\psi$ be as in~\eqref{eq:Tpsi}. Then there
  is a unique rigid motion $g:=g_{\calO}\in G$ such that
  \begin{equation}\label{eq:intertwinedtranslation}
  g\circ \phi \mid_{S_{\calL}}= \phi\circ T \mid_{S_{\calL}},
\end{equation}
where $S_{\calL}$ is the maximal strip~\eqref{eq:S_L}. Furthermore,
$g^{\pm 1}$ is determined by $\calO$ and does not depend on the choice
of $x\in\calO$ or coordinatization $\psi$.
\item Let $\iota:\RR^2\to\RR^2$ be a $\pi$-rotation sending
  $\ZZ^2\to \ZZ^2$ and $S_{\calL}\to S_{\calL}$. Then there is an
  order $2$ rigid motion $h\in G$ such that
\begin{equation}\label{eq:intertwinedrotation}
  h\circ \phi \mid_{S_{\calL}}= \phi\circ \iota \mid_{S_{\calL}}.
\end{equation}
\end{enumerate}
\end{lemma}

\begin{proof}
  For the first part, let $\calF\subset M$ be the square face containing
  $x$, and let $g\in G$ be the rigid motion such that
  \begin{equation*}
    g\mid_{\calF} = \phi\circ T \circ \psi.
  \end{equation*}
  The element $g$ is guaranteed by Lemma~\ref{lem:parallelism} to exist. Then we have
  \begin{equation*}
    g \circ \phi \mid_{[0,1]^2} = \phi \circ T \mid_{[0,1]^2}. 
  \end{equation*}
  The maps on either side must have the same locally isometric
  extension to $S_\calL$, and so~(\ref{eq:intertwinedtranslation}) follows.

  For the second part, the unique fixed point $v$ of $\iota$ must lie
  on the center line through $S_\calL$ and, since that center line
  does not intersect $\ZZ^2$, $v$ must have at least one coordinate
  $\equiv 1/2\pmod 1$. That is, $v$ lies either at the center of an
  integer translate $\calH$ of $[0,1]^2$ or at the midpoint of one of
  its edges. Let $\calF' = \phi(\calH)$.

  If $\phi(v)$ is at the center of $\calF'$, then let $h\in G$ be the
  $\pi$-rotation sending $\calF'$ to itself.  On the other hand, if
  $\phi(v)$ is the midpoint of an edge of $\calF'$, let $h\in G$ be
  the $\pi$-rotation interchaging $\calF'$ with the square face
  sharing that edge. In either case, we have
  \begin{equation*}
    h\circ \phi \mid_{\calH}= \phi\circ \iota \mid_{\calH},
  \end{equation*}
  and again both sides must have the same locally isometric extension
  to $S_\calL$, which gives~(\ref{eq:intertwinedrotation}).
\end{proof}

\section{Cylinder Decompositions of $M$}\label{sec:cylinderDecomp}

Given a parallel direction field $V$ with rational slope, one obtains
a tiling of $M$ by maximal strips/cylinders: 
\begin{equation}\label{eq:tiling}
  M = \bigcup_{S\in \calS_V} S,
\end{equation}
where
\begin{equation*}
  \calS_V = \set{S : S = S_{\calO_{x,V}} \textrm{ for some } x \in M_0}.
\end{equation*}
In fact, more is true. We will show the following.

\begin{theorem}\label{thm:cylinders}~
  \begin{enumerate}[label=(\alph*)]

  \item Every periodic trajectory gives rise to a decomposition of $M$
    into isometric cylinders of area $4$. Each of these cylinders has
    an order-$4$ symmetry by a $\pi/2$-rotation of $\RR^3$ over an
    axis parallel to one of the coordinate axes; it is
      disjoint from all of its $(2\ZZ)^3$-translates; and
      it is disjoint from all of its $2\pi/3$-rotations.
  \item Every non-periodic trajectory in a rational direction gives
    rise to a decomposition of $M$ into isometric bi-infinite
    strips. Each of these strips is sent to itself by a translation by
    an element in $(2\ZZ)^3$.
  \end{enumerate}
\end{theorem}

We start by proving that ``tiles'' in
(\ref{eq:tiling}) are isometric to one another through rigid motions.

The idea is this. A linear trajectory in a rational direction lies
in either a maximal cylinder or a maximal strip, $S\subset M$, whose
boundary $\partial S$ passes through cone points of $M$. The images of
$S$ under the action of $G$ are isometric copies of $S$ that either
coincide with $S$ itself, or whose interiors are disjoint from
$S$. Furthermore, by the transitivity statements of
Lemma~\ref{lem:parallelism}, their union covers $M$.

\begin{proposition}\label{prop:Mcylinders}
  Let $V$ be a direction field on $M_0$ with rational slope. Then in
  the $V$ direction, either $M$ completely decomposes into isometric
  cylinders or $M$ decomposes into isometric bi-infinite strips.
\end{proposition}

\begin{proof}
  Let $\calO$ be a trajectory in $M_0$ in direction $V$, and $S_\calO$
  its maximal strip or cylinder. Then for any $g\in G$, the image
  $g\calO$ is also a linear trajectory in direction $V$, by
  Lemma~\ref{lem:parallelism}, and we have that
  \begin{equation*}
    g(S_{\calO}) = S_{g\calO}
  \end{equation*}
  is the maximal strip or cylinder around $g\calO$. Since $G$ acts by
  isometries, $S_{g\calO}$ is isometric to $S_\calO$. We will be done
  once we show that
  \begin{equation}\label{eq:monotiling}
    M = \bigcup_{g\in G} S_{g\calO} = \bigcup_{g\in G} g(S_{\calO}). 
  \end{equation}
  Let $x\in \calF \subset M$ where $\calF$ is the square face containing $x$. By Remark~\ref{ontotorus}, there is some
  point $y\in S_\calO$ holding the same relative position in its square
  face, say, $\calF'$, as $x$ holds in
  $\calF$. By~Lemma~\ref{lem:parallelism}, there is a rigid motion
  $g\in G$ such that $g\calF' = \calF$ and $gy=x$. Note, then, that
  $x\in g(S_\calO)$.  Since $x$ was an arbitrary
  point,~(\ref{eq:monotiling}) holds.
\end{proof}

Next, we prove that given a trajectory $\calO$ in a rational direction, its periodicity is detected by the associated rigid motion $g_\calO$ defined in Lemma~\ref{lem:intertwining}. 

\begin{lemma}\label{lem:torsion}
Let $V$ be a parallel field on $M_0$ with rational direction and $\calO$ a cone-point avoiding linear trajectory along $V$. Let $g:=g_\calO$ be determined as in Lemma~\ref{lem:intertwining}, noting that it is unique up to inversion. The trajectory $\calO$ is periodic if
  and only if $g$ is of finite order. In this case, the area of the maximal cylinder $S_{\calO}$
  is the order of $g$.
\end{lemma}

\begin{proof}
  Recall that $\phi$ defines a locally isometric covering
  \begin{equation*}
    \phi: S_{\calL} \to S_{\calO}.
  \end{equation*}
  If $g \in G$ is of order $k$, then
  by~(\ref{eq:intertwinedtranslation}) we have
  $\phi = \phi\circ T_{\psi}^k$. This implies that $\calO$ is
  periodic.

  Conversely, suppose the trajectory $\calO$ is periodic in
  $M$. Then there is some minimal $k$ such that
  $\phi = \phi\circ T_{\psi}^k$. Then,
  by~(\ref{eq:intertwinedtranslation}), we have
  $g^k\circ \phi = \phi$. That is, $g^k$ fixes every point of $\calO$ so
  must be the identity, and $g$ is torsion of order $k$.

  In this case, note that the covering map $\phi$ admits fundamental
  domain 
  
  \begin{equation*}
    S_{\calL}\cap \parens*{\RR\times[0,kp]}
  \end{equation*}
  if $p\neq 0$ and 
  \begin{equation*}
    S_{\calL}\cap \parens*{[0,kq]\times \RR}
  \end{equation*}
  otherwise where $p,q$ are such that $V$ has slope
  $\frac{p}{q}$. Either way, the parallelogram has area $k$, therefore
  $S_{\calO}$ has area $k$.
\end{proof}

For a periodic trajectory $\calO$, the previous lemma asserts that its maximal cylinder is given by the order of the associated rigid motion $g_\calO$. This fact can now be used to show that every maximal cylinder in the Mucube has the same area. 
\begin{proposition}\label{lem:areafour}
  Every maximal cylinder of the Mucube has area 4 and is mapped to itself by a $\pi/2$ rotation from $G$.
\end{proposition}

\begin{proof}
  It follows from Lemmas~\ref{lem:torsion}
  and~\ref{lem:periodicrotations} that every maximal cylinder of the
  Mucube has area 1, 2, 3, or 4.
  
  Let $S_\calO$ be a maximal cylinder in $M$, covered by
  $\phi:S_\calL \to S_\calO$ and assume that $\calO$ is the core curve
  of the cylinder. Let $\iota$ be an involution as in
  Lemma~\ref{lem:intertwining}, fixing the point $v\in S_\calL$, and
  let $h\in G$ be the corresponding rigid motion
  satisfying~(\ref{eq:intertwinedrotation}). Since
  $\calO$ is the core curve of $S_\calO$, it is sent to itself by $h$
  with the opposite orientation. Therefore, there are two points,
   $x=\phi(v)$ and $y \in \calO$, that are fixed by $h$,
  and they are antipodal in $\calO$. The fixed points of $h$ are
  either all center points of square faces, or all midpoints of edges,
  as established in Lemma~\ref{lem:periodicrotations}. This means that $x$ and $y$ occupy
  the same relative positions in their respective faces, hence there
  is a smallest integer $\ell\geq1$ such that
  $\phi(T^\ell (v))=y$, where $T$ is as in~(\ref{eq:Tpsi}).

  Let $g$ be the rigid motion associated to $T$ in the sense of
  Lemma~\ref{lem:intertwining}. By~(\ref{eq:intertwinedtranslation}),
  $\phi(T^\ell(v)) = g^\ell x = y$, and it follows from $x\neq y$ that
  $\ell < k$, where $k$ is the order of $g$ and also the smallest
  integer such that $\phi\circ T^k = \phi$.  On the other hand, the
  fact that $y$ is the half-way point of $\calO$ (from $x$'s
  perspective) implies that $g^{2\ell}$ is the identity. Hence
  $g^\ell$ is a $\pi$-rotation, and $g$ is a
  $\frac{\pi}{\ell}$-rotation. This means that $g$ is either a
  $\pi$-rotation or a $\pi/2$-rotation.

  Note that $g^\ell h$ is a $\pi$-rotation sending $\calO$ to itself
  with opposite direction and interchanging $x$ and $y$. It therefore
  has two antipodal fixed points $w,z\in \calO$, half-way between the
  arcs of $\calO$ connecting $x$ and $y$, and having the same relative
  positions as $x$ and $y$ do---if $x,y$ are the centers of faces,
  then so are $w, z$; if $x,y$ are midpoints of edges, 
  then so are
  $w,z$. Therefore, $w$ is the image under $\phi$ of the midpoint
  between $v$ and $T^\ell v$ in $\calL$, and that point projects to
  the same point on $\RR^2/\ZZ^2$ as $v$ and $T^\ell v$ do. Hence the
  components of $T^\ell v- v$ are not coprime, and so $\ell>1$.

  Therefore, $\ell=2$ and $g$ has order $4$, and $S_\calO$ has area
  $4$, by Lemma~\ref{lem:torsion}.
\end{proof}

\begin{wrapfigure}[24]{r}{0.5\textwidth}
  \centering
  \begin{tikzpicture}[scale=0.75]

    \node at (6.5,7.5){$w$};
    \draw[-latex] (6.77, 7.75) -- (6.85, 7.85);

    \node at (6.5,3.5){$x$};
    \draw[-latex] (6.77, 3.28) -- (6.85, 3.2);

    \node at (2.5,7.5){$y$};
    \draw[-latex] (2.28, 7.75) -- (2.2, 7.85);

    \node at (2.5,3.5){$z$};
    \draw[-latex] (2.28, 3.28) -- (2.2, 3.2);

    \node at (6.5,5.5){$w'$};
    \draw[-latex] (6.77, 5.75) -- (6.85, 5.85);

    \node at (4.5,5.5){$y'$};
    \draw[-latex] (4.28, 5.75) -- (4.2, 5.85);

    \node at (4.5,3.5){$z'$};
    \draw[-latex] (4.28, 3.28) -- (4.2, 3.2);

    \draw[thick, blue, <-] (-0.25, 5.35) -- (0, 5.35);
    \draw[thick, blue, ->] (10, 5.35) -- (10.25, 5.35);
    
    \foreach \k in {0, ..., 4}
    {
      \draw[thick, blue] (2*\k,5.35) -- (2*\k+0.85, 5.35);
      \draw[thick, dotted] (2*\k+0.85,5.35) -- (2*\k+2, 5.35);
      }

    \draw[thick, red, <-] (-0.25, 4.35) -- (0, 4.35);
    \draw[thick, red, ->] (10, 4.35) -- (10.25, 4.35);
    
    \foreach \k in {0, ..., 4}
    {
      \draw[thick, dotted] (2*\k,4.35) -- (2*\k+1, 4.35);
      \draw[thick, red] (2*\k+1,4.35) -- (2*\k+2, 4.35);
      }

    \foreach \k in {0, ..., 4}
    {
      \foreach \l in {0, ..., 4}
      {
        \fill[opacity=.3] (2*\k,2*\l) rectangle (2*\k+1,2*\l+1);
        \fill[opacity=.3] (2*\k+1,2*\l+1) rectangle (2*\k+2,2*\l+2);
        
        \fill[opacity=.5] (2*\k+1,2*\l+0.7) rectangle (2*\k+1.7,2*\l+1);
        \fill[opacity=.4] (2*\k-0.3, 2*\l) rectangle (2*\k,2*\l+0.7);

        \fill[opacity=0.7] (2*\k + 1.7, 2*\l+0.7) -- (2*\k+2, 2*\l+1)
        -- (2*\k+1.7, 2*\l+1) -- (2*\k + 1.7, 2*\l+0.7);

        \fill (2*\k- 0.3, 2*\l+0.7) -- (2*\k, 2*\l+1)
        -- (2*\k, 2*\l+0.7) -- (2*\k -0.3, 2*\l+0.7);

        \fill[opacity=.1] (2*\k+1, 2*\l+1) -- (2*\k+1, 2*\l + 2) --
        (2*\k+0.7, 2*\l + 1.7) -- (2*\k+0.7, 2*\l+1) -- (2*\k+1, 2*\l+1);

        \fill[opacity=.2] (2*\k,2*\l) -- (2*\k+1,2*\l) --
        (2*\k+0.7,2*\l-0.3) -- (2*\k,2*\l-0.3) -- (2*\k,2*\l);

        }

      }

      \foreach \l in {0, ..., 4}
      {
        \fill[opacity=0.4] (-0.3, 2*\l-0.3) -- (0, 2*\l)
        -- (-0.3, 2*\l) -- (-0.3, 2*\l-0.3);

        \fill[opacity=0.2] (-0.3, 2*\l -0.3) -- (0, 2*\l)
        -- (0, 2*\l-0.3) -- (0.3, 2*\l -0.3);
      }

      \foreach \k in {1, ..., 4}
      {
        \fill[opacity=0.4] (2*\k-0.3, -0.3)
        -- (2*\k, 0) -- (2*\k-0.3, 0) -- (2*\k-0.3, -0.3);

        \fill[opacity=0.2] (2*\k-0.3, -0.3) -- (2*\k, 0)
        -- (2*\k, -0.3) -- (2*\k - 0.3,  -0.3);
      }

      \draw[thin, purple, opacity=0.2] (4.375, 3) -- (4.325, 2.7);

      \draw[thin, purple, opacity=0.2] (4.325, 2.7) -- (4.575, 1.7);
      
      \draw[thin, green] (4.125, 2) -- (4.375, 3);

      \draw[thin, green] (4.575, 1.7) -- (4.85, 1.85);

      \draw[thin, green] (4.85, 1.85) -- (5, 1.825);

      \draw[thin, green] (5, 1.825) -- (6, 1.625);

      \draw[thin, purple, opacity=0.2] (6, 1.625) -- (5.7, 1.075);
      
      \draw[thin, purple, opacity=0.2] (5.7, 1.075) -- (4.7, 0.825);

      \draw[thin, purple, opacity=0.2] (4.7, 0.825) -- (4.85, 0.85);

      \draw[thin, purple, opacity=0.2] (4.85, 0.85) -- (4.875, 1);

      \draw[thin, purple, opacity=0.2] (4.075, -0.3) -- (3.825, 0.7);

      \draw[thin, purple, opacity=0.2] (3.825, 0.7) -- (3.85, 0.85);

      \draw[thin, purple, opacity=0.2] (3.85, 0.85) -- (4, 1.125);
      
      \draw[thin, green] (4.875, 1) -- (4.625, 0);

      \draw[thin, green] (4.625, 0) -- (4.075, -0.3);

      \draw[thin, green] (4, 1.125) -- (3, 1.375);

      \draw[thin, green] (3, 1.375) -- (2.7, 1.325);

      \draw[thin, purple, opacity=0.2] (2.7, 1.325) -- (3.7, 1.575);

      \draw[thin, purple, opacity=0.2] (3.7, 1.575) -- (3.85, 1.85);

      \draw[thin, purple, opacity=0.2] (3.85, 1.85) -- (3.9875, 1.925);

      \draw[thin, green] (3.9875, 1.925) -- (4.125, 2);

  \end{tikzpicture}
  \caption{Positions of $x,w,y,z$, from the proof of
    Corollary~\ref{lem:midpoints}, and $w',y',z'$ from
    Remark~\ref{rem:nonillumination}. In blue, the line bisecting the
    line segments $\overline{xw}$ and $\overline{yz}$. In red, the
    line bisecting the line segments $\overline{xw'}$ and
    $\overline{y'z'}$. Note that the red line does not intersect $M$.}
  \label{fig:xwyz}
\end{wrapfigure}
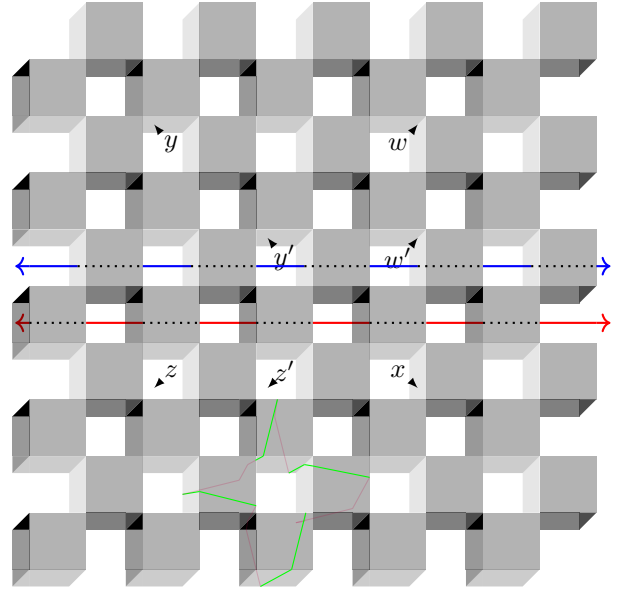
The fact that every maximal cylinder in the Mucube has area 4 has a few important consequences. One of these consequences is that the core curves of cylinders must pass through precisely 4 centers of squares. 

\begin{corollary}\label{lem:midpoints}
  The core curve of every maximal cylinder in $M$ passes through the
  center points of four square faces in $M$.
\end{corollary}
\begin{proof}
  In the proof of \Cref{lem:areafour} it is observed that the
  fixed points of $h$ are either center points of squares, or
  midpoints of edges. In the former case, the proof establishes that
  $x$ and $y=g^2x$ are the center points of their respective faces,
  and so are $w = gx$ and $z = g^3 x$.

  Suppose on the other hand that the fixed points of $h$ are midpoints
  of edges. Then the proof of \Cref{lem:areafour} shows that the
  points $x, w, y, z$ are midpoints of edges, as in
  Figure~\ref{fig:xwyz}.  The $\pi$-rotation over the axis bisecting
  the line segments $\overline{xw}$ and $\overline{yz}$ takes
  $x\mapsto w$ and $y\mapsto z$ and so takes $\calO$ to itself with
  opposite direction. It therefore must fix the points on $\calO$
  which are midway between $x$ and $w$ and midway between $y$ and $z$
  along the trajectory. Such a rigid motion can only fix points of $M$
  which are at center points of square faces. 
  
  Therefore, $\calO$ must contain the center point $p$ of some square
  face in $M$, as well as its images $gp, g^2p, g^3p$, which are also
  the centers of their respective squares.
\end{proof}

\begin{remark}\label{rem:nonillumination}
  The above proof leads to the observation that the Mucube is not
  illuminated: Given a point $x\in M$, there are points in $M$ that
  cannot be reached by traveling from $x$ along a cone-point avoiding
  linear trajectory. For example, suppose that Figure~\ref{fig:xwyz}
  is contained in $\RR^2\times [-1/2, 1/2]$.  We claim that there can
  be no linear trajectory from $x$ to the point $w' = w - (0,2, 0)$.

  To see the claim, note that $x$ is sent to $w'$ by a
  $\pi/2$-rotation $g' \in G$. Therefore, a linear trajectory from $x$
  to $w'$ would have to be periodic, visiting the points
  $x, w', y' = y + (2, -2, 0)$ and $z' = z+(2,0,0)$. It would also
  have to be taken to itself by the $\pi$-rotation over the line
  bisecting $\overline{xw'}$ and $\overline{y'z'}$. But that line does
  not intersect $M$, therefore its $\pi$-rotation has no fixed points
  in $M$. Since orientation-reversing automorphisms of circles must
  have two fixed points, there can be no linear trajectory from $x$ to
  $w'$.
\end{remark}

The fact that every core curve of a cylinder passes through precisely 4 centers of squares restricts the slope of the cylinder.

\begin{corollary}\label{cor:oddodd}
  A periodic slope cannot have both an odd numerator and an odd
  denominator.
\end{corollary}

\begin{proof}
  This is equivalent to the statement in Corollary~\ref{lem:midpoints}
  that the core curve of a periodic cylinder passes through centers of
  squares. A line in $\RR^2$ passing through the center of a unit
  square and having a slope whose numerator and denominator are odd,
  must intersect $\ZZ^2$. Conversely, a line in $\RR^2$ containing the
  center of some unit square and intersecting $\ZZ^2$ must have a
  slope with odd numerator and denominator (when expressed in reduced
  form).
\end{proof}

We end this sub-section by stating the proof of Theorem~\ref{thm:cylinders}.

\begin{proof}[Proof of Theorem~\ref{thm:cylinders}]
  (Part (a)) By Proposition~\ref{prop:Mcylinders} and
  \Cref{lem:areafour}, a periodic trajectory $\calO$ gives rise
  to a decomposition of $M$ into isometric cylinders of area $4$. Each
  must be disjoint from any $(2\ZZ)^3$-translate, for otherwise it
  would drift to $\infty$. By Corollary~\ref{lem:midpoints}, the core
  curve of any cylinder passes through $4$ center points of square
  faces of $M$, cycled by a $\pi/2$ rotation. Therefore, the cylinder
  cannot have a $2\pi/3$ symmetry, for such a mapping would introduce
  more centers of squares.

  (Part (b)) By Proposition~\ref{prop:Mcylinders}, a nonperiodic
  trajectory in a rational direction gives rise to a decomposition of
  $M$ into isometric bi-infinite strips. Let $S$ be one such strip,
  with core curve $\calO$. Let
  $g_\calO=(\bt, \theta)\in \ZZ^3 \rtimes O$ be the rigid motion
  associated to $\calO$ by Lemma~\ref{lem:intertwining}, and note that
  if $\theta$ has order $k$, then $g^k$ sends $S$ to itself by
  translating by $2\bv$, where
  \begin{equation*}
    \bv = (1 + \theta + \dots + \theta^{k-1})\bt.
  \end{equation*}
  This proves the theorem.
\end{proof}

In the following subsections, we elucidate the consequences of Theorem~\ref{thm:cylinders} for two natural compact quotients of $M$ and demonstrate characterizations of periodic directions on $M$ via these two quotients.

\subsection{Descent to $X$}\label{sec:descentX}

Let $X = M/\ZZ^3$, the quotient of $M$ by its $\ZZ^3$-action. Under
this map, any maximal cylinder $C\subset M$ is folded onto a
maximal cylinder in $X$. In fact, this folding is one-to-one on the
interior of $C$.

\begin{proposition}\label{prop:isometricprojection}
  Let $\pi: M \rightarrow X$ denote the quotient map. Then, if $C$ is
  any cylinder in $M$, then $\pi|_C:C \rightarrow \pi(C)$ is an
  isometry.
\end{proposition}
\begin{proof}
  First, note that the covering map $\pi$ is a local isometry, so it
  suffices to show that $\pi|_{C}$ is a single-sheeted covering map.

  From Theorem~\ref{thm:cylinders} and its proof we know that the
  core curve of $C$ passes through the midpoints of exactly 4 squares in
  $M$ which are in same $\pi/2$-rotation orbit. Since the deck group
  associated to the map $\pi$ is generated by translations along the
  three axes, the midpoints of these 4 squares projects to 4 distinct
  midpoints of squares in $X$. Hence, the core curve of $\pi(C)$
  passes through precisely 4 midpoints so that $\pi|_C$ must be single-sheeted.
\end{proof}

\begin{wrapfigure}{l}{0.35\textwidth}
\centering
\begin{tikzpicture}[scale=0.7, shift={(0,0,4)}]
\coordinate (O) at (0,0,0);
\coordinate (A) at (0,6,0);
\coordinate (B) at (0,6,6);
\coordinate (C) at (0,0,6);
\coordinate (D) at (6,0,0);
\coordinate (E) at (6,6,0);
\coordinate (F) at (6,6,6);
\coordinate (G) at (6,0,6);

\coordinate (B2) at (1.5,4.5,6);
\coordinate (C2) at (1.5,1.5,6);
\coordinate (F2) at (4.5,4.5,6);
\coordinate (G2) at (4.5,1.5,6);

\coordinate (A3) at (1.5,4.5,4.5);
\coordinate (B3) at (1.5,4.5,6);
\coordinate (E3) at (4.5,4.5,4.5);
\coordinate (F3) at (4.5,4.5,6);

\coordinate (D4) at (4.5,1.5,4.5);
\coordinate (E4) at (4.5,4.5,4.5);
\coordinate (F4) at (4.5,4.5,6);
\coordinate (G4) at (4.5,1.5,6);

\coordinate (B5) at (1.5,6,4.5);
\coordinate (C5) at (1.5,4.5,4.5);
\coordinate (F5) at (4.5,6,4.5);
\coordinate (G5) at (4.5,4.5,4.5);

\coordinate (A6) at (1.5,6,1.5);
\coordinate (B6) at (1.5,6,4.5);
\coordinate (E6) at (4.5,6,1.5);
\coordinate (F6) at (4.5,6,4.5);

\coordinate (D7) at (4.5,4.5,1.5);
\coordinate (E7) at (4.5,6,1.5);
\coordinate (F7) at (4.5,6,4.5);
\coordinate (G7) at (4.5,4.5,4.5);

\coordinate (B8) at (4.5,4.5,4.5);
\coordinate (C8) at (4.5,1.5,4.5);
\coordinate (F8) at (6,4.5,4.5);
\coordinate (G8) at (6,1.5,4.5);

\coordinate (A9) at (4.5,4.5,1.5);
\coordinate (B9) at (4.5,4.5,4.5);
\coordinate (E9) at (6,4.5,1.5);
\coordinate (F9) at (6,4.5,4.5);

\coordinate (D10) at (6,1.5,1.5);
\coordinate (E10) at (6,4.5,1.5);
\coordinate (F10) at (6,4.5,4.5);
\coordinate (G10) at (6,1.5,4.5);

\coordinate (B11) at (0,4.5,4.5);
\coordinate (C11) at (0,1.5,4.5);
\coordinate (F11) at (A3);
\coordinate (G11) at (1.5,1.5,4.5);

\coordinate (A12) at (0,4.5,1.5);
\coordinate (B12) at (B11);
\coordinate (E12) at (1.5,4.5,1.5);
\coordinate (F12) at (F11);

\coordinate (B13) at (G11);
\coordinate (C13) at (1.5,0,4.5);
\coordinate (F13) at (D4);
\coordinate (G13) at (4.5,0,4.5);

\coordinate (D14) at (4.5,0,1.5);
\coordinate (E14) at (4.5,1.5,1.5);
\coordinate (F14) at (F13);
\coordinate (G14) at (G13);

\coordinate (A15) at (1.5,4.5,0);
\coordinate (B15) at (E12);
\coordinate (E15) at (4.5, 4.5, 0);
\coordinate (F15) at (A9);

\coordinate (D16) at (4.5,1.5,0);
\coordinate (E16) at (E15);
\coordinate (F16) at (F15);
\coordinate (G16) at (E14);

\begin{scope}[on glass layer]


\draw[fill=gray!80] (A3) -- (B3) -- (F3) -- (E3) -- cycle;

\draw[fill=gray!80] (D4) -- (E4) -- (F4) -- (G4) -- cycle;


\draw[line width=0.5mm, red] (C2) -- (B2) -- (F2) -- (G2) -- cycle;


\draw[fill=gray!80] (C5) -- (B5) -- (F5) -- (G5) -- cycle;

\draw[] (A6) -- (B6) -- (F6) -- (E6) -- cycle;

\draw[fill=gray!80] (D7) -- (E7) -- (F7) -- (G7) -- cycle;

\draw[line width=0.5mm, blue] (A6) -- (B6) -- (F6) -- (E6) -- cycle;


\draw[fill=gray!80] (C8) -- (B8) -- (F8) -- (G8) -- cycle;

\draw[fill=gray!80] (A9) -- (B9) -- (F9) -- (E9) -- cycle;

\draw[line width=0.5mm, green] (F9) -- (E9) -- (D10) --  (G8)-- cycle;


\node[] at (0,0,3) {$I$}; 

\node[] at (5,0,0) {$J$}; 
\node[] at (0,5.25,0) {$K$}; 
\end{scope}

\begin{scope}[on above layer]
\draw[fill=gray!30] (G11) -- (C2) -- (B2) -- (F11) -- cycle;

\draw[fill=gray!30] (G11) -- (C2) -- (G2) -- (D4) -- cycle;


\draw[fill=gray!30] (E9) -- (D10) -- (E14) -- (A9) -- cycle;

\draw[fill=gray!30] (A6) -- (E6) -- (D7) -- (B15) -- cycle;

\end{scope}

\begin{scope} 

\draw[fill=gray!80] (B11) -- (C11) -- (G11) -- (F11) -- cycle;

\draw[line width=0.5mm, green] (B11) -- (C11);

\draw[fill=gray!80] (A12) -- (B12) -- (F12) -- (E12) -- cycle;

\draw[line width=0.5mm, green] (A12) -- (B12);


\draw[fill=gray!80] (B13) -- (C13) -- (G13) -- (F13) -- cycle;

\draw[line width=0.5mm, blue] (C13) -- (G13);

\draw[fill=gray!80] (D14) -- (E14) -- (F14) -- (G14) -- cycle;

\draw[line width=0.5mm, blue] (G14) -- (D14);

\draw[fill=gray!80] (A15) -- (B15) -- (F15) -- (E15) -- cycle;

\draw[fill=gray!80] (D16) -- (E16) -- (F16) -- (G16) -- cycle;

\end{scope}

\begin{scope}[on behind layer]

\draw[fill=gray!30] (G11) -- (C11) -- (0, 1.5, 1.5) -- (1.5, 1.5, 1.5) -- cycle;

\draw[fill=gray!30] (0, 1.5, 1.5) -- (1.5, 1.5, 1.5) -- (E12) -- (A12) --cycle;


\draw[fill=gray!30](1.5, 1.5, 1.5)-- (G16) -- (D16) -- (1.5, 1.5, 0) -- cycle;

\draw[fill=gray!30] (1.5, 1.5, 1.5) -- (G16) -- (D14) -- (1.5, 0, 1.5) -- cycle;

\draw[fill=gray!30] (1.5, 1.5, 1.5) -- (G11) -- (C13) -- (1.5, 0, 1.5) -- cycle;

\draw[fill=gray!30] (1.5, 1.5, 1.5) -- (1.5, 1.5, 0) -- (1.5, 4.5, 0)  -- (1.5, 4.5, 1.5) -- cycle;

\draw[line width=0.5mm, red](1.5, 1.5, 0) -- (A15) -- (E15) -- (D16) --  cycle;

\end{scope}

\begin{scope}[on background layer]
\draw[fill=gray!30] (D10) -- (E10) -- (F10) -- (G10) -- cycle;

\draw[fill=gray!30] (A6) -- (B6) -- (F6) -- (E6) -- cycle;

\end{scope}

\end{tikzpicture}
\caption{The fundamental domain $F_0$ with the curves basic curves $I$, $J$ and $K$ labelled.}
\label{fig:FwithBasicCurves}
\end{wrapfigure}
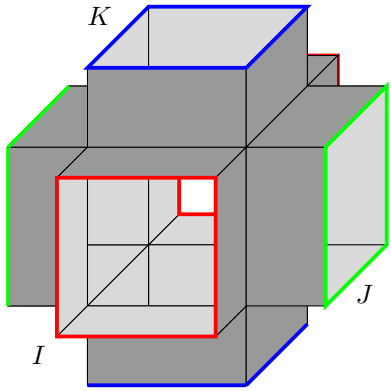

Now Theorem~\ref{thm:cylinders} tells us that the periodic directions
in the Mucube can only be those directions in which there is a
decomposition of $X$ into cylinders of area 4, each of whose core
curve passes through the center points of four square faces of
$X$. However, not all such directions lift as periodic directions in
the Mucube $M$. For instance, one finds by inspection that in the
direction $(5,2)$, the surface $X$ decomposes into cylinders of area
4, but in this direction the straight line flow of $M$ dissipates.

Let $I, J, K$ denote the three core curves defining the boundary of
the fundamental domain $F$, identified via the $\ZZ^3$ action, as
pictured in Figure \ref{fig:FwithBasicCurves}. To a smooth path
$\gamma:[0,T]\to X$ starting and ending off of $I, J, K$, and only
crossing $I,J,K$ transversally, we associate a \emph{displacement
  vector}
\begin{equation*}
  \bv_\gamma:[0,T]\to \ZZ_{\geq 0}^3  \qquad \bv_\gamma(t) = a(t)\mathbf{i} + b(t)\mathbf{j} + c(t)\mathbf{k},
\end{equation*}
where $a(t)$ is the number of times $\gamma$ crossed $I$ in the
$\mathbf{i}$ direction in the time interval $[0,t]\subset[0,T]$, and
so on. Note that any lift $\tilde \gamma: [0,T]\to M$ will cross the
lifts of $I, J, K$ whenever $\gamma$ crosses $I, J, K$, so if
$\tilde \gamma(0)\in F_0$, then
$\tilde \gamma(T) \in F_{\bv_\gamma(T)}$. Clearly, the lift of a path
$\gamma$ is periodic if and only if there exists some $T>0$ such that
$\gamma(0)=\gamma(T)$ (i.e., the path itself is periodic) and
$F_{\mathbf{v}_\gamma(T)} = F_0$, that is, if its displacement vector is
$\mathbf{v}_\gamma(T) = 0$.  We observe that for any
  $\theta\in O\subset G$,
  \begin{equation}\label{eq:coordpermute}
    \bv_{\theta\gamma}(T) = \theta \bv_\gamma(T). 
  \end{equation}
  
  After all, the lift $\widetilde{\theta\gamma}$ having
  $\widetilde{\theta\gamma}(0) \in F_0$ is precisely
  $\widetilde{\theta\gamma} = \theta \tilde{\gamma}$, therefore,
  $\widetilde{\theta\gamma}(T) \in \theta F_{\bv(T)} = F_{\theta\bv(T)}$.

  In fact, Theorem~\ref{thm:cylinders} leads to the following
  characterization of periodic trajectories in $M$.

\begin{theorem}[Periodicity characterization using
    $X$]\label{thm:displacementvector}
    Let $\gamma:[0,T] \to X$ be a straight-line trajectory with period
    $T$ and $\gamma(0) \notin I\cup J\cup K$, and let
    $\bv(t):=\bv_\gamma(t)$ denote its displacement at time $t$. The
    following are equivalent.
    \begin{enumerate}
    \item The trajectory $\gamma$ has a periodic lift to $M$.
    \item There is a $\pi/2$-rotation, $\theta$, over one of the
      coordinates ($\mathbf{i}, \mathbf{j}$, or $\mathbf{k}$) such
      that
      \begin{equation}\label{eq:dispcycle}
        \mathbf{v}\parens*{\frac{(i+1)T}{4}} - \mathbf{v}\parens*{\frac{iT}{4}} = \theta^i \mathbf{v}\parens*{\frac{T}{4}}
      \end{equation}
      for each $i=0, 1, 2, 3$. 
    \item In the direction of $\gamma$, $X$ decomposes
        into three maximal cylinders that are cycled by
        $2\pi/3$-rotations, and $a+b+c=0$.
    \end{enumerate}
\end{theorem}

\begin{proof}
  (1. $\implies$ 2.) If the lift $\tilde\gamma$ is periodic in $M$,
  then by Theorem~\ref{thm:cylinders} there is a $\pi/2$-rotation
  $g=(\bt, \theta)\in G$ taking the trajectory to itself and taking
  $\tilde\gamma(iT/4)$ to $\tilde\gamma((i+1)T/4)$. Then $\theta$ acts
  on the displacement vectors as described by~(\ref{eq:dispcycle}).

  (1. $\implies$ 3.) If the lift $\tilde\gamma$ is
    periodic in $M$, then by Theorem~\ref{thm:cylinders} its maximum
    cylinder has area $4$ and is disjoint from all of its $2\pi/3$
    rotations. This implies that the maximal cylinder of $\gamma$ is
    disjoint from its images under $2\pi/3$ rotations of $X$, hence
    that $X$ decomposes into three cylinders in direction $\gamma$,
    which are cycled by $2\pi/3$ rotations, as claimed. Of course,
    since $\tilde\gamma$ is periodic, we have $\bv(T)=(0,0,0)$, so
    $\bv(T)\cdot (1,1,1)=0$.

  (2. $\implies$ 1.) Note that~(\ref{eq:dispcycle}) implies
  \begin{equation*}
    \sum_{i=0}^3 \parens*{\mathbf{v}\parens*{\frac{(i+1)T}{4}} - \mathbf{v}\parens*{\frac{iT}{4}}} = \sum_{i=0}^3\theta^i \mathbf{v}\parens*{\frac{T}{4}}.
  \end{equation*}
  The left-hand side is $\bv(T)$ and the right-hand side is
  $0$. Therefore the lift $\tilde\gamma$ is periodic with period $T$.

  (3. $\implies$ 1.) By applying $\theta_{2\pi/3}$, we
    see that the three cylinders must have displacement vectors
    $(a, b, c)$, $(b, c, a)$, and $(c, a, b)$,
    by~(\ref{eq:coordpermute}). Two applications of
    $\theta_z^2\theta_{2\pi/3}$
    
    shows that $(-b, -c, a)$ and
    $(c, -a, -b)$ coincide with $(b,c,a)$ and $(c, b, a)$ in some
    order, leading to the linear system
    \begin{align*}
      a + b + c &= 0 \\
      a -b -c &=0 \\
      -a - b + c &= 0,
    \end{align*}
    which has only the trivial solution. Therefore,
    $\bv(T) = (a,b,c)=0$ and $\gamma$ has a periodic lift to $M$.  
\end{proof}

\subsection{Descent to $Y$}\label{sec:descentToY}
Suppose $V$ is a parallel direction field in $M_0$. Since it is
preserved by the $\ZZ^3$-action, it descends to a parallel field on
$X_0$
\begin{equation*}
  X_0 =  X - \set{\textrm{cone points}}.
\end{equation*}
If $V$ is a periodic direction in $M_0$, then the corresponding
cylinder decomposition of $M$ descends to a decomposition of $X$ into
three isometric cylinders of area $4$, which are cycled by the action
of
\begin{equation*}
  \theta_{2\pi/3} = \begin{bmatrix}
    0 & 0 & 1 \\ 1 & 0 & 0 \\ 0 & 1 & 0
  \end{bmatrix}.
\end{equation*}
Let $Y = X/\theta_{2\pi/3}$, the quotient of $X$ by the
$\ZZ/3\ZZ$-action generated by $\theta_{2\pi/3}$. The surface $Y$ is
best visualized by seeing
\begin{equation*}
  F' = M\cap\brackets*{-\frac{1}{2},\frac{3}{2}}^3
\end{equation*}
as the fundamental domain of $X$, as in
Figure~\ref{fig:rotation}. Then $\theta_{2\pi/3}$ is the rotation that
cycles the three basic cylinders.
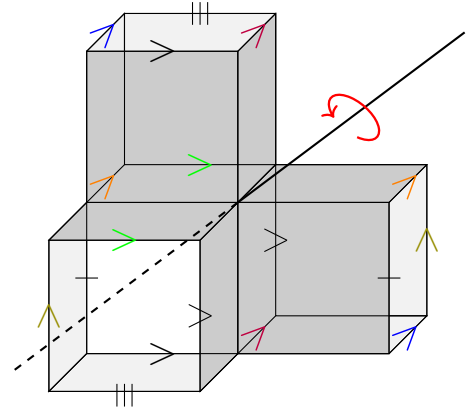
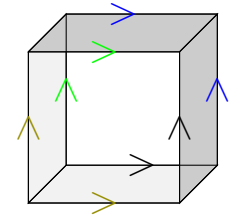
\begin{wrapfigure}[27]{r}{0.38\textwidth}
\vspace{-0.5cm}
\center
\subfloat[$X$ admits an order 3 rotation symmetry.]{
\begin{tikzpicture}[sq/.style=
  {shape=regular polygon, regular polygon sides=4, draw, minimum width=2.828cm}]


\draw[fill=gray!10] (-1,-1) -- (-1, 1) -- (-1+0.5, 1+0.5) -- (-1+0.5, -1+0.5)--cycle; 
\draw[fill=gray!10] (1,-1) -- (1+0.5, -1+0.5) --(-1+0.5,-1+0.5) -- (-1,-1) --cycle; 
\draw[fill=gray!40]  (-1, 1) -- (-1+0.5, 1+0.5) -- (1+0.5, 1+0.5) -- (1,1) --cycle; 
\draw[fill=gray!40] (1,1) -- (1+0.5, 1+0.5) --(1+0.5,-1+0.5) -- (1,-1) --cycle; 



  \node[sq] at (0,0){};
  \node[sq] at (0.5,0.5){};
  \draw[thin] (-1,-1) -- ++(0.5, 0.5);
  \draw[thin] (-1,1) -- ++(0.5, 0.5);
  \draw[thin] (1,-1) -- ++(0.5, 0.5);
  \draw[thin] (1,1) -- ++(0.5, 0.5);
  
\begin{scope}[shift = {(0.5, 2.5)}]

\draw[fill=gray!10] (-1,-1) -- (-1, 1) -- (-1+0.5, 1+0.5) -- (-1+0.5, -1+0.5)--cycle; 
\draw[fill=gray!10]  (-1+0.5, -1+0.5) -- (-1+0.5, 1+0.5) -- (1+0.5, 1+0.5) -- (1+0.5,-1+0.5) --cycle; 
\draw[fill=gray!40]  (-1, -1) -- (-1, 1) -- (1, 1) -- (1,-1) --cycle; 
\draw[fill=gray!40] (1,1) -- (1+0.5, 1+0.5) --(1+0.5,-1+0.5) -- (1,-1) --cycle; 

 \node[sq] at (0,0){};
  \node[sq] at (0.5,0.5){};
  \draw[thin] (-1,-1) -- ++(0.5, 0.5);
  \draw[thin] (-1,1) -- ++(0.5, 0.5);
  \draw[thin] (1,-1) -- ++(0.5, 0.5);
  \draw[thin] (1,1) -- ++(0.5, 0.5);
\end{scope}

\begin{scope}[shift = {(2.5, 0.5)}]


\draw[fill=gray!10]  (-1+0.5, -1+0.5) -- (-1+0.5, 1+0.5) -- (1+0.5, 1+0.5) -- (1+0.5,-1+0.5) --cycle; 
\draw[fill=gray!10] (1,-1) -- (1+0.5, -1+0.5) --(-1+0.5,-1+0.5) -- (-1,-1) --cycle; 
\draw[fill=gray!40]  (-1, 1) -- (-1+0.5, 1+0.5) -- (1+0.5, 1+0.5) -- (1,1) --cycle; 

\draw[fill=gray!40]  (-1, -1) -- (-1, 1) -- (1, 1) -- (1,-1) --cycle; 


 \node[sq] at (0,0){};
  \node[sq] at (0.5,0.5){};
  \draw[thin] (-1,-1) -- ++(0.5, 0.5);
  \draw[thin] (-1,1) -- ++(0.5, 0.5);
  \draw[thin] (1,-1) -- ++(0.5, 0.5);
  \draw[thin] (1,1) -- ++(0.5, 0.5);
\end{scope}


\draw[thick] (1.5,1.5) -- ++ (0.75*4,0.75*3) node [midway] {\AxisRotator[x=0.2cm,y=0.4cm,->,rotate=50, red]};

\draw[thick, dashed] (1.5,1.5) -- ++ (-4*0.75,-3*0.75);


\draw[thin] (-0.65,0.5) -- ++(0.3, 0);
\draw[thin] (3.5-0.15,0.5) -- ++(0.3, 0);


\draw[thin, rotate = 90 ] (-0.15,-0.35-0.5) -- ++(0.15,-0.3);
  \draw[thin,rotate = 90 ] (-0.15+0.15,-0.35-0.3-0.5) -- ++(0.15,0.3);
  
  \draw[thin, rotate = 90 ] (-0.15+1,-0.35-0.5-1) -- ++(0.15,-0.3);
  \draw[thin,rotate = 90 ] (-0.15+0.15+1,-0.35-0.3-0.5-1) -- ++(0.15,0.3);

  \draw[thin] (7.90-8,-0.65-0.25) -- ++(0,-0.3);
  \draw[thin] (8.00-8,-0.65-0.25) -- ++(0,-0.3);
  \draw[thin] (8.10-8,-0.65-0.25) -- ++(0,-0.3);
  
  \draw[thin] (0.90,-0.65+4.5) -- ++(0,0.3);
  \draw[thin] (1.00,-0.65+4.5) -- ++(0,0.3);
  \draw[thin] (1.10,-0.65+4.5) -- ++(0,0.3);


\node[rotate = 90, thick, font=\fontsize{15}{0}\selectfont, thick] at (-1,0){{\color{olive}$>$}};
\node[rotate = 90, thick, font=\fontsize{15}{0}\selectfont, thick] at (4,1){{\color{olive}$>$}};


\node[thick, font=\fontsize{15}{0}\selectfont, thick] at (0,1){{\color{green}$>$}};
\node[ thick, font=\fontsize{15}{0}\selectfont, thick] at (1,2){{\color{green}$>$}};


\node[thick, font=\fontsize{15}{0}\selectfont, thick] at (0.5,-0.5){{\color{black}$>$}};
\node[ thick, font=\fontsize{15}{0}\selectfont, thick] at (0.5,3.5){{\color{black}$>$}};


\node[rotate = 45, thick, font=\fontsize{15}{0}\selectfont, thick] at (3.75,1.75-2){{\color{blue}$>$}};

\node[rotate = 45, thick, font=\fontsize{15}{0}\selectfont, thick] at (3.75-4,1.75+2){{\color{blue}$>$}};


\node[rotate = 45, thick, font=\fontsize{15}{0}\selectfont, thick] at (3.75-2,1.75-2){{\color{purple}$>$}};

\node[rotate = 45, thick, font=\fontsize{15}{0}\selectfont, thick] at (3.75-4+2,1.75+2){{\color{purple}$>$}};
  
\node[rotate = 45, thick, font=\fontsize{15}{0}\selectfont, thick] at (3.75,1.75){{\color{orange}$>$}};

\node[rotate = 45, thick, font=\fontsize{15}{0}\selectfont, thick] at (3.75-4,1.75){{\color{orange}$>$}};

\end{tikzpicture}

}
\hspace{2cm}
\subfloat[The genus 1 quotient $Y$ under the order 3 rotation]{\begin{tikzpicture}
    [sq/.style=
  {shape=regular polygon, regular polygon sides=4, draw, minimum width=2.828cm}]


\draw[fill=gray!10] (-1,-1) -- (-1, 1) -- (-1+0.5, 1+0.5) -- (-1+0.5, -1+0.5)--cycle; 
\draw[fill=gray!10] (1,-1) -- (1+0.5, -1+0.5) --(-1+0.5,-1+0.5) -- (-1,-1) --cycle; 
\draw[fill=gray!40]  (-1, 1) -- (-1+0.5, 1+0.5) -- (1+0.5, 1+0.5) -- (1,1) --cycle; 
\draw[fill=gray!40] (1,1) -- (1+0.5, 1+0.5) --(1+0.5,-1+0.5) -- (1,-1) --cycle; 

  \node[sq] at (0,0){};
  \node[sq] at (0.5,0.5){};
  \draw[thin] (-1,-1) -- ++(0.5, 0.5);
  \draw[thin] (-1,1) -- ++(0.5, 0.5);
  \draw[thin] (1,-1) -- ++(0.5, 0.5);
  \draw[thin] (1,1) -- ++(0.5, 0.5);


\node[thick, font=\fontsize{15}{0}\selectfont, thick] at (0,1){{\color{green}$>$}};
\node[rotate=90, thick, font=\fontsize{15}{0}\selectfont, thick] at (-0.5,0.5){{\color{green}$>$}};

\node[rotate = 90, thick, font=\fontsize{15}{0}\selectfont, thick] at (-1,0){{\color{olive}$>$}};
\node[thick, font=\fontsize{15}{0}\selectfont, thick] at (0,-1){{\color{olive}$>$}};

\node[rotate=90, thick, font=\fontsize{15}{0}\selectfont, thick] at (1.5,0.5){{\color{blue}$>$}};
\node[thick, font=\fontsize{15}{0}\selectfont, thick] at (0.25,1.5){{\color{blue}$>$}};

\node[rotate=90, thick, font=\fontsize{15}{0}\selectfont, thick] at (1,0){{\color{black}$>$}};
\node[thick, font=\fontsize{15}{0}\selectfont, thick] at (0.5,-0.5){{\color{black}$>$}};

  \end{tikzpicture}}

\caption{Genus 3 to genus 1 quotient}
\label{fig:rotation}
\end{wrapfigure}

Note then that $I, J, K$ are identified in $Y$ as a single curve
$\gamma_0$. To a path in $Y$, we associate its algebraic intersection number 
\begin{equation*}
  i = \#\set{\textrm{upward crossings of } \gamma_0} - \#\set{\textrm{downward crossings of } \gamma_0}.
\end{equation*}
Each time a path in $X$ (or $M$) crosses any of $I, J, K$ in the
positive direction, the projected path in $Y$ crosses $\gamma_0$ in
the upward direction, and vice versa, hence we have
\begin{equation}\label{eq:iabc}
  i = a + b + c
\end{equation}
where $\bv=(a,b,c)$ is the displacement vector of the path in $X$.

We therefore have shown that any cylinder decomposition of $M$
descends to a decomposition of $Y$ into a single cylinder whose core
curve has algebraic intersection number $i=0$.

  In fact, Theorem~\ref{thm:displacementvector} shows that the
  converse is also true.

\begin{theorem}[Periodicity characterization using
  $Y$]\label{thm:characterization}
  A direction is periodic in $M$ if and only if $Y$ decomposes into a
  single cylinder in that direction and the algebraic intersection
  number of the core curve of this cylinder with the horizontal core
  curve is $0$.
\end{theorem}

\begin{proof}
  Suppose $V$ is a periodic direction on $M$. By
  Theorem~\ref{thm:displacementvector}(3), the corresponding cylinder
  decomposition on $M$ descends to a decomposition of $X$ into three
  area-four cylinders that are cycled by $\theta_{2\pi/3}$-rotations,
  and have displacement vector $\bv=(a,b,c)$ with $a + b +
  c=0$. Therefore, this decomposition descends to a one-cylinder
  decomposition of $Y$ whose core curve has algebraic intersection
  $i=0$, by~\eqref{eq:iabc}.

  On the other hand, a single-cylinder decomposition of $Y$ having
  algebraic intersection $i=0$ lifts to a decomposition of $X$ into
  three cylinders that are cycled by $2\pi/3$ rotations, and whose
  displacement vectors are orthogonal to $(1,1,1)$. By
  Theorem~\ref{thm:displacementvector}, this lifts to a decomposition
  of $M$ into periodic cylinders.
\end{proof}

\section{Exhibiting trajectories with certain properties}\label{sec:fourey}

In this section we describe a process through which two periodic
directions give rise to a third. Using this process, we prove the
following theorem:

\family*

To prove this theorem, we organize this section in the following way:

\begin{itemize}
\item In Section~\ref{sec:arbitrarilyLarge}, we demonstrate a family of periodic trajectories of arbitrarily large diameter. This family is characterized by the continued fraction expansions of their slopes (see Theorem~\ref{thm:diameter} and Proposition~\ref{prop:foureyslopes}). This will prove part 1 of Theorem~\ref{thm:family}.

\item In Section~\ref{sec:recurrence}, we will prove parts 2 and 3 of Theorem~\ref{thm:family} by demonstrating an uncountable family of directions 
with recurrent trajectories. Although such a set's existence follows from
  Corollary~\ref{cor:ergodic directions}, we construct one explicitly. Furthermore, it is known that the Hausdorff dimension of the slopes satisfying 2 or 3 in Theorem~\ref{thm:family} is approximately 0.68 (see Remark~\ref{rem:HausdorffDimension}).
\end{itemize}

\begin{wrapfigure}[12]{r}{0.5\textwidth}
\centering
\vspace{-1cm}
\begin{tikzpicture}[scale=0.6]
  \draw[directed] (-4,-2) -- (6,2);
  \draw[dotted] (-4,-2.5) -- (6,1.5);
  \draw[dotted] (-4,-1.5) -- (6,2.5);

    \draw[dotted] (-4,-4) -- (-4,1);
    \draw[dotted] (-4,1) -- (-2,1);
    \draw[dotted] (-2,1) -- (-2,-4);
    \draw[dotted] (-2,-4) -- (-4,-4);

    \draw[dotted] (-2,-4) -- (0,-4);
    \draw[dotted] (0,-4) -- (0,1);
    \draw[dotted] (0,1) -- (-2,1);

    \draw[dotted] (0,2) -- (0,-3);
    \draw[dotted] (0,-3) -- (2,-3);
    \draw[dotted] (2,-3) -- (2,2);
    \draw[dotted] (2,2) -- (0,2);
    
    \draw[dotted] (2,4) -- (2,-1);
    \draw[dotted] (2,-1) -- (4,-1);
    \draw[dotted] (4,-1) -- (4,4);
    \draw[dotted] (4,4) -- (2,4);

    \draw[dotted] (2,4) -- (2,-1);
    \draw[dotted] (2,-1) -- (4,-1);
    \draw[dotted] (4,-1) -- (4,4);
    \draw[dotted] (4,4) -- (2,4);

    \draw[dotted] (4,3) -- (4,-2);
    \draw[dotted] (4,-2) -- (6,-2);
    \draw[dotted] (6,-2) -- (6,3);
    \draw[dotted] (6,3) -- (4,3);

    \draw[directed] (-4,-2) -- (-2.95,1);
    \draw[directed] (-2.95,-4) -- (-1.2,1);
    \draw[directed] (-1.2,-4) -- (0.9,2);
    \draw[directed] (0.9,-3) -- (3.1,4);
    \draw[directed] (3.2,-1) -- (4.7,3);
    \draw[directed] (4.7,-2) -- (6,2);

    \node[] at (1.2,0.4) {$\calO$};
    \node[] at (-4.1,0) {$T_V(\calO)$};
  	
\end{tikzpicture}  	
\caption{$\calO$ and $T_V(\calO)$ are shown together with the cylinders in direction $V$ as well as the interior of the cylinder containing $\calO$.}
\label{odot}
\end{wrapfigure}
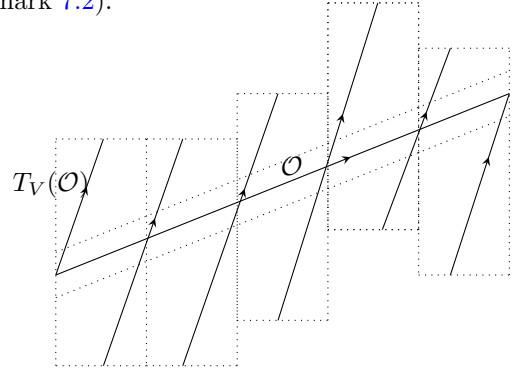

We now define the process through which two periodic directions give rise to a third. This process in the definition below is equivalent to the application
of multiple distinct left Dehn twists to a given periodic trajectory.

\begin{definition}\label{def:geomOperation}
  Let $V$ be a periodic direction and fix an orientation on the Mucube
  $M$. Let $\calO$ be a periodic trajectory not parallel to $V$. Note
  that $\calO$ is decomposed into finitely many segments of equal
  length by the intersection with cylinders in direction $V$. Replace
  each segment with a straight line segment such that
  \begin{enumerate}
  \item the new segment has the same end points
  \item the new segment is contained in the same cylinder
  \item the interior of the new segment does not intersect $\calO$
  \item the new segment departs the starting point in a direction
    counterclockwise to that of the old segment.
  \end{enumerate}
  Define $T_V(\calO)$ as the resulting periodic trajectory.
\end{definition}

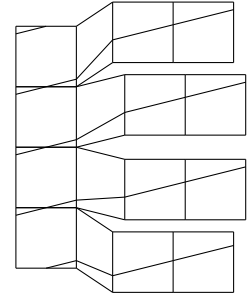
\begin{wrapfigure}[14]{r}{0.3\textwidth}
\centering
\begin{tikzpicture}[scale=0.8]
\begin{scope}[rotate=90,yscale=-1,xscale=1]
    \draw[] (-2,0) -- (-1,0);
    \draw[] (-1,0) -- (-1,1);
    \draw[] (-1,1) -- (-2,1);
    \draw[] (-2,1) -- (-2,0);

        \begin{scope}[xshift=1cm]
    \draw[] (-2,0) -- (-1,0);
    \draw[] (-1,0) -- (-1,1);
    \draw[] (-1,1) -- (-2,1);
    \draw[] (-2,1) -- (-2,0);
        \end{scope}

        \begin{scope}[xshift=2cm]
    \draw[] (-2,0) -- (-1,0);
    \draw[] (-1,0) -- (-1,1);
    \draw[] (-1,1) -- (-2,1);
    \draw[] (-2,1) -- (-2,0);
        \end{scope}

        \begin{scope}[xshift=3cm]
    \draw[] (-2,0) -- (-1,0);
    \draw[] (-1,0) -- (-1,1);
    \draw[] (-1,1) -- (-2,1);
    \draw[] (-2,1) -- (-2,0);
        \end{scope}

    \draw[] (-1,1) -- (-1.2,1.8);
    \draw[] (-1.2,1.8) -- (-0.2,1.8);
    \draw[] (-0.2,1.8) -- (0,1);

    \draw[] (-1.2,1.8) -- (-1.2,2.8);
    \draw[] (-1.2,2.8) -- (-0.2, 2.8);
    \draw[] (-0.2, 2.8) -- (-0.2, 1.8);

    \begin{scope}[yshift=1cm]
    \draw[] (-1.2,1.8) -- (-1.2,2.8);
    \draw[] (-1.2,2.8) -- (-0.2, 2.8);
    \draw[] (-0.2, 2.8) -- (-0.2, 1.8);
    \end{scope}
    \begin{scope}[xshift=-1cm]
    \draw[] (-1,1) -- (-1.4,1.6);
    \draw[] (-1.4,1.6) -- (-0.4,1.6);
    \draw[] (-0.4,1.6) -- (0,1);
    \end{scope}
    \begin{scope}[xshift=-1.2cm,yshift=-0.2cm]
       \draw[] (-1.2,1.8) -- (-1.2,2.8);
         \draw[] (-1.2,2.8) -- (-0.2, 2.8);
         \draw[] (-0.2, 2.8) -- (-0.2, 1.8);

         \begin{scope}[yshift=1cm]
            \draw[] (-1.2,1.8) -- (-1.2,2.8);
         \draw[] (-1.2,2.8) -- (-0.2, 2.8);
         \draw[] (-0.2, 2.8) -- (-0.2, 1.8);
            \end{scope} 
    \end{scope}
    \begin{scope}[yscale=1,xscale=-1]
        \draw[] (-1,1) -- (-1.2,1.8);
         \draw[] (-1.2,1.8) -- (-0.2,1.8);
         \draw[] (-0.2,1.8) -- (0,1);

        \draw[] (-1.2,1.8) -- (-1.2,2.8);
        \draw[] (-1.2,2.8) -- (-0.2, 2.8);
        \draw[] (-0.2, 2.8) -- (-0.2, 1.8);

        \begin{scope}[yshift=1cm]
        \draw[] (-1.2,1.8) -- (-1.2,2.8);
        \draw[] (-1.2,2.8) -- (-0.2, 2.8);
        \draw[] (-0.2, 2.8) -- (-0.2, 1.8);
        \end{scope}
        
        \begin{scope}[xshift=-1cm]
        \draw[] (-1,1) -- (-1.4,1.6);
        \draw[] (-1.4,1.6) -- (-0.4,1.6);
        \draw[] (-0.4,1.6) -- (0,1);
        \end{scope}

         \begin{scope}[xshift=-1.2cm,yshift=-0.2cm]
       \draw[] (-1.2,1.8) -- (-1.2,2.8);
         \draw[] (-1.2,2.8) -- (-0.2, 2.8);
         \draw[] (-0.2, 2.8) -- (-0.2, 1.8);

         \begin{scope}[yshift=1cm]
            \draw[] (-1.2,1.8) -- (-1.2,2.8);
            \draw[] (-1.2,2.8) -- (-0.2, 2.8);
            \draw[] (-0.2, 2.8) -- (-0.2, 1.8);
            \end{scope} 
    \end{scope}
    \end{scope}

    \draw[] (-2,0.5) -- (-2+0.125,1);
    \draw[] (-2+0.125,1) -- (-2.4+0.275, 1.6);
    \draw[] (-2.4+0.275, 1.6) -- (-2.4+0.525, 2.6);
    \draw[] (-2.4+0.525, 2.6) -- (-2.4+0.775,3.6);

    \draw[] (-1.125,0) --(-1+0.125,1);
    \draw[] (-1+0.125,1) -- (-1.2+0.375,1.8);
    \draw[] (-1.2+0.375,1.8) -- (-1.2+0.625,2.8);
    \draw[] (-1.2+0.625,2.8) -- (-1.2+0.875,3.8);

    \draw[] (-1+0.875,0) -- (0.125,1);
    \draw[] (0.125,1) -- (0.2+0.375,1.8);
    \draw[] (0.2+0.375,1.8) -- (0.2+0.625,2.8);
    \draw[] (0.2+0.625,2.8) -- (0.2+0.875,3.8);
    
    \draw[] (0.875,0) -- (1.125,1);
    \draw[] (1.125,1) -- (1.4+0.375,1.6);
    \draw[] (1.4+0.375,1.6) -- (1.4+0.625,2.6);
    \draw[] (1.4+0.625,2.6) -- (1.4+0.875,3.6);

    \draw[] (1+0.875,0) -- (2,0.5);
\end{scope}
\end{tikzpicture}
\caption{A trajectory with slope $1/4$ shown on a region of the Mucube. Note that this is the same trajectory as in Figure~\ref{fig:xwyz}.\protect\footnotemark }
\label{fig:fourey}
\end{wrapfigure}

\footnotetext{See \href{https://sites.google.com/view/sunroseshrestha/periodic-trajectory-on-mucube}{https://sites.google.com/view/sunroseshrestha/periodic-trajectory-on-mucube} for an animation of the trajectory.}

\emph{A priori}, the process described above results in a piece-wise
linear loop. However, since the cylinders are parallel (all in
direction $V$), the new linear segments are parallel as well. Indeed,
this forms a periodic straight line trajectory. See
Figure~\ref{odot}. We note that the angle that $T_V(\calO)$ makes with
the direction $V$ is smaller than the angle that $\calO$ makes with
$V$. 
\begin{example}
    Applying the construction in Definition~\ref{def:geomOperation} by taking $V$ to be the horizontal direction and $\calO$ to be a trajectory in the vertical direction yields a trajectory with slope $\frac{1}{4}$. See Figure~\ref{fig:fourey}.
\end{example}

\subsection{Arbitrarily large periodic trajectories}\label{sec:arbitrarilyLarge}

The next proposition describes the asymptotic growth of the length of a periodic trajectory with repeated application of $T_V$. When we write $f(k) \sim g(k)$ as $k \to \infty$ for functions $f$ and $g$, we mean $\lim_{k \rightarrow \infty} \frac{f(k)}{g(k)} = 1$. For a periodic direction $V$, $\len V$ will refer to the length of any periodic orbit in the direction $V$ (which is well defined, by Proposition~\ref{prop:Mcylinders}). Likewise, $\wid V$ will refer to the width of any cylinder in direction $V$.

\begin{proposition}\label{prop:AsymLengthOfTwist}
Let $V$ be a periodic direction and $\calO$ be a periodic trajectory with slope transverse to $V$. Let $\theta \in (0, \pi/2]$ be the angle between $V$ and $\calO$. Then, as $k \rightarrow \infty$,
    \begin{equation*}
        \len T_V^k(\calO) \sim   k \cdot \sin(\theta) \cdot\frac{\len V}{\wid V}\cdot
  \len \mathcal O
    \end{equation*}
\end{proposition}
\begin{proof}

First we compute that the number of cylinders in the $V$ direction that $\calO$ passes through is given by $\sin(\theta) \cdot  \frac{\len \calO}{\wid V}$ (see Figure~\ref{fig:NumCylindersCutByO}). 

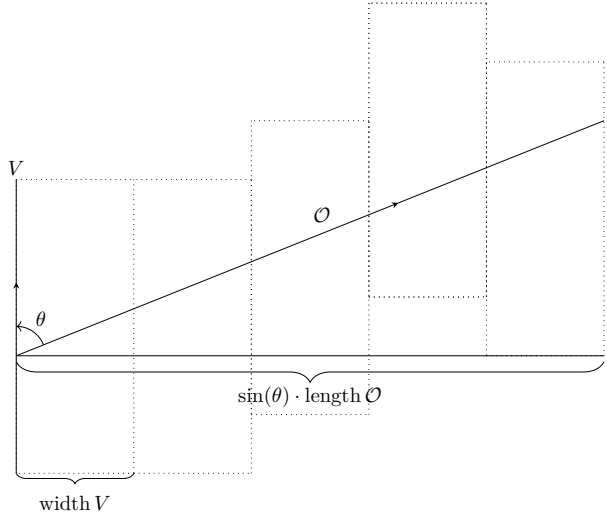
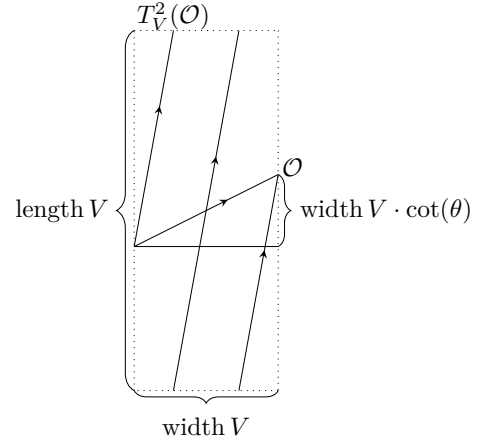
\begin{figure}[h!]
\centering
\subfloat[The number of cylinders in direction $V$ that the periodic trajectory $\cal O$ cuts through can be computed to be $\frac{\sin(\theta)\cdot \len \calO}{\wid V}$.]{
\resizebox{0.5\textwidth}{!}{
\begin{tikzpicture}
  \draw[directed] (-4,-2) -- (6,2);

    \draw[dotted] (-4,-4) -- (-4,1);
    \draw[dotted] (-4,1) -- (-2,1);
    \draw[dotted] (-2,1) -- (-2,-4);
    \draw[dotted] (-2,-4) -- (-4,-4);

    \draw[dotted] (-2,-4) -- (0,-4);
    \draw[dotted] (0,-4) -- (0,1);
    \draw[dotted] (0,1) -- (-2,1);

    \draw[dotted] (0,2) -- (0,-3);
    \draw[dotted] (0,-3) -- (2,-3);
    \draw[dotted] (2,-3) -- (2,2);
    \draw[dotted] (2,2) -- (0,2);
    
    \draw[dotted] (2,4) -- (2,-1);
    \draw[dotted] (2,-1) -- (4,-1);
    \draw[dotted] (4,-1) -- (4,4);
    \draw[dotted] (4,4) -- (2,4);

    \draw[dotted] (2,4) -- (2,-1);
    \draw[dotted] (2,-1) -- (4,-1);
    \draw[dotted] (4,-1) -- (4,4);
    \draw[dotted] (4,4) -- (2,4);

    \draw[dotted] (4,3) -- (4,-2);
    \draw[dotted] (4,-2) -- (6,-2);
    \draw[dotted] (6,-2) -- (6,3);
    \draw[dotted] (6,3) -- (4,3);

\draw[directed] (-4,-4) -- (-4,1);


\draw[] (-4,-2) -- (6,-2);

\coordinate (intersect) at (-4,-2);
\coordinate (V) at (-4,1);
\coordinate (O) at (6,2);
\pic [draw, ->, "$\theta$", angle eccentricity=1.5] {angle = O--intersect--V};

\draw [decorate,decoration={brace,amplitude=10pt}] (6,-2.1) -- (-4,-2.1)  node [black,midway, yshift=-0.6cm] { $\sin(\theta)\cdot \len \calO$};

\draw [decorate,decoration={brace,amplitude=5pt}] (-2,-4) -- (-4,-4)  node [black,midway, yshift=-0.5cm] { $\wid V$};

    \node[] at (1.2,0.4) {$\calO$};
    \node[] at (-4,1.2) {$V$};

\end{tikzpicture} 
\label{fig:NumCylindersCutByO}
}}
\hspace{1cm}
\subfloat[Here segments of $T^2_V(\calO)$ and $\calO$ within a single cylinder are shown together. The length of the segment shown of $T^2_V(\calO)$ can be computed to be $\sqrt{(\wid V)^2 + (2 \cdot \len V + \cot(\theta) \cdot \wid V)^2}$]{
\resizebox{0.4\textwidth}{!}{
\begin{tikzpicture}
  \draw[directed] (-4,-2) --++ (2,1);

    \draw[dotted] (-4,-4) -- (-4,1);
    \draw[dotted] (-4,1) -- (-2,1);
    \draw[dotted] (-2,1) -- (-2,-4);
    \draw[dotted] (-2,-4) -- (-4,-4);

    \draw[directed] (-4,-2) --++ (6/11,3);
    \draw[directed] (-4+6/11,-2+3-5) --++ (10/11,5);
    \draw[directed] (-4+6/11+10/11,-2+3-5+5-5) --++ (6/11,3);


\draw[] (-4,-2) -- (-2,-2);


\draw [decorate,decoration={brace,amplitude=5pt}] (-2,-4) -- (-4,-4)  node [black,midway, yshift=-0.5cm] { $\wid V$};

\draw [decorate,decoration={brace,amplitude=5pt}] (-2,-1) -- (-2,-2)  node [black,midway, xshift=1.5cm] { $\wid V \cdot \cot(\theta)$};

\draw [decorate,decoration={brace,amplitude=7pt}] (-4,-4) -- (-4,1)  node [black,midway, xshift=-1cm] { $\len V$};

    \node[] at (-1.8,-0.9) {$\calO$};
    \node[] at (-4+6/11,1.2) {$T^2_V(\calO)$};

\end{tikzpicture}  	
\label{fig:LengthOfTwist}
}}
\caption{The number of cylinders that $\calO$ passes through and the length of a segment of $T^2_V(\calO)$ within a single cylinder.}
\end{figure}

Next, in each cylinder that $\calO$ passes through, the length of the segment of $T^k_V(\calO)$ contained in that cylinder is given by (see Figure~\ref{fig:LengthOfTwist}),
$$\sqrt{(\wid V)^2 + (k \cdot \len V + \cot(\theta) \cdot \wid V)^2}.$$
Therefore, 
\begin{align*}
    \len T^k_V(\calO) &= \sin(\theta) \cdot  \frac{\len \calO}{\wid V}\cdot  \sqrt{(\wid V)^2 + (k \cdot \len V + \cot(\theta) \cdot \wid V)^2} \\ & \sim k \cdot \sin(\theta) \cdot  \frac{\len V}{\wid V} \cdot \len \calO \hspace{1cm} (k \rightarrow \infty)
\end{align*}
 where the asymptotic follows by noting $\theta, \len V, \len \calO$ and $\wid V$ are fixed.\end{proof}

Using Proposition~\ref{prop:AsymLengthOfTwist} we see that the length of periodic trajectories can get arbitrarily large. In fact, even the \emph{diameter} (say, in $\RR^3$) of periodic trajectories can get arbitrarily large, as the next result shows.

 \begin{theorem}[Periodic orbits of arbitrary diameter]\label{thm:diameter}
   There are periodic linear orbits of the Mucube, $M$, of
   arbitrarily large diameter (say, in the $\norm{\cdot}_\infty$-norm
      of $\RR^3$). 
 \end{theorem}

\begin{proof}
  Begin with a periodic trajectory $\calO$. Let $S$ be the union of
  the square faces of the Mucube that $\calO$ passes through. The
  trajectory $\calO$ must be transversal to either the vertical or the
  horizontal direction. Without loss of generality, suppose it is
  transversal to the vertical direction, $V$. Then $T_V(\calO)$ is a new
  periodic trajectory passing through every square face in $S$, as
  well as every square face of every basic cylinder intersecting
  $\calO$. Importantly, $T_V(\calO)$ passes through more square faces
  of $M$ than $\calO$ does.

  Iterating this process, we may produce periodic trajectories passing
  through arbitrarily many square faces of $M$. Since every bounded
  set in $\RR^3$ contains finitely many square faces of $M$, the
  result follows from the pigeonhole principle.
\end{proof}

\begin{remark}
  By a similar argument, beginning with a horizontal periodic
  trajectory and alternating applications of $T_V$ and $T_H$, we
  obtain a sequence of periodic trajectories $\{\calO_n\}$. This
  particular sequence has the property that $\calO_{n+1}$ has
  diameter exactly two greater than $\calO_{n}$.
\end{remark}

We now use the geometric operation in Definition
\ref{def:geomOperation} to give an explicit family of periodic
trajectories whose slopes have the following nice continued fraction
expansions, which for obvious reasons, we refer to as ``Fourey fractions.''

\begin{proposition}[Fourey fractions are periodic]\label{prop:foureyslopes}
  All rational numbers of the form
  \begin{equation*}
    [4a_0; 4a_1, 4a_2, \dots,4a_n] := 4a_0 + \frac{1}{\displaystyle 4a_1 + \frac{1}{\displaystyle 4a_2 +\frac{1}{\ddots +\frac{1}{4a_n}}}},
  \end{equation*}
  where $a_0\in\ZZ$ and $a_i \in\ZZ\setminus\set{0}$ for all
  $i\geq 1$, are slopes of periodic linear orbits of $M$.
\end{proposition}

\begin{proof}
  Let the direction $V$ be vertical with respect to a coordinatization
  of a square face of $M$. Observe that if $\calO$ is a trajectory of
  slope $s$, then for any $a\in\ZZ$, the trajectory $T_V^{a}(\calO)$
  has slope $4a + s$. Hence, if $s$ is a periodic slope, then so is
  $4a+s$. Since reciprocals of periodic slopes are also periodic, it
  follows that $\frac{1}{4a+s}$ is a periodic slope. The proposition
  follows by iterating this process finitely many times, beginning
  with the periodic slope $s=0$.
\end{proof}

One can also deduce Proposition~\ref{prop:foureyslopes} from the
algebraic description of periodic directions in
Theorem~\ref{thm:veechgrpchar}. See Remark~\ref{rem:thm1conseq}.

\subsection{Recurrent trajectories}\label{sec:recurrence}

With the aid of Proposition \ref{prop:foureyslopes}, we now construct
recurrent trajectories in a family of directions characterized by
their continued fraction expansions. In what follows, we consider the
set of possible slopes $\mathbb{R} \cup \infty$ to have the topology
of a circle (the one-point compactification of $\mathbb{R}$). That
way, it is homeomorphic to the circle of possible directions of
straight line flow.

\begin{definition}\label{def:recurrence}
  A linear trajectory $\gamma: [0, \infty) \rightarrow M$ on the
  Mucube is said to be \textbf{recurrent} if for every $\epsilon > 0$
  there exists a sequence $t_n \rightarrow +\infty$ such that for all
  $n$, $\gamma(t_n)$ is in the $\epsilon$-neighborhood of $\gamma(0)$.
\end{definition}

In other words, a trajectory is recurrent if it comes arbitrarily
close to its starting point infinitely often.

We recall the following basic facts from the theory of continued
fractions (see~\cite{RockettSzusz}). Given the formal expression
\begin{equation*}
  [a_0; a_1, a_2, a_3, \dots]
\end{equation*}
where the coefficients are treated as symbols, put $p_0 = a_0$,
$q_0 = 1$, $p_1 = a_1a_0 + 1$, $q_1= a_1$, and for $n\geq 1$,
\begin{gather}\label{eq:recurrencerelations}
  \begin{split}
          p_{n+1} &= a_{n+1}p_n + p_{n-1} \\
    q_{n+1} &= a_{n+1}q_n + q_{n-1}. 
  \end{split}
\end{gather}
Then for each $n\geq 0$, we have
\begin{equation*}
  \frac{p_n}{q_n} = [0; a_1, a_2, a_3, \dots, a_n]
\end{equation*}
and
\begin{equation}
  \label{eq:flipflop}
  p_{n-1}q_n-p_nq_{n-1} = (-1)^n.
\end{equation}
In particular, if the coefficients are integers, then so are
$p_n, q_n$, and we have $\gcd(p_n,q_n)=1$ for all $n\geq
1$. Furthermore, if $\abs{a_n}\geq 4$ for all $n\geq 1$, as is the
assumption in Theorem~\ref{thm:family}, then $\abs{q_n}$ is a strictly
increasing sequence. From~(\ref{eq:flipflop}) it follows that
\begin{equation*}
  \abs*{\frac{p_{n}}{q_{n}} - \frac{p_{n-1}}{q_{n-1}}} = \frac{1}{q_nq_{n-1}},
\end{equation*}
and the right-hand side forms a convergent series. Therefore, the
sequence of convergents $p_n/q_n$ converges to some real number, $\xi$.

In cases where $a_n\geq 1\, (n\geq 1)$, it is well known that the
convergents satisfy
\begin{equation}
  \label{eq:convapprox}
  \abs*{q_n\xi - p_n} < \frac{1}{q_{n+1}}
\end{equation}
for every $n\geq 0$. In fact, Hurwitz' theorem says that of any three
consecutive convergents, at least one of them satisfies the inequality
\begin{equation}\label{eq:hurwitzog}
  \abs{q_n\xi - p_n} < \frac{1}{\sqrt{5} q_n},
\end{equation}
and $1/\sqrt{5}$ is the smallest constant for which this is true
(see~\cite{niven}). We will prove an analogous theorem for the
continued fractions appearing in Theorem~\ref{thm:family}.

\begin{lemma}[Fourey Hurwitz]\label{lem:hurwitz}
  Let $k \in \ZZ_{\geq 1}$ and 
  \begin{equation*}
    \xi = [0; 4a_1, 4a_2, 4a_3, \dots]
  \end{equation*}
  where $a_i\in\ZZ\setminus\set{-(k-1), \dots, 0, \dots, k-1}$ for all $i\geq 1$. Then for every $n$, we have
  \begin{equation}\label{eq:hurwitz}
    \abs*{q_n\xi - p_n} < \frac{1}{2\sqrt{4k^2-1}q_n},
  \end{equation}
  and this is best possible. Furthermore, if the sequence
  $(a_n)$ does not have a tail of alternating $k$'s and $-k$'s, then
    \begin{equation*}
      \abs*{q_n\xi - p_n} < \frac{1}{4kq_n}
  \end{equation*}
  for infinitely many $n\geq 0$, and this is best possible.
\end{lemma}

\begin{proof}
  For each $n\geq 0$, let 
  \begin{equation*}
    \xi_n = [4a_n; 4a_{n+1}, 4a_{n+2}, \dots].
  \end{equation*}
  Then for each $n\geq 0$ we have
  \begin{align}
    q_n\xi - p_n
    &= q_n\frac{\xi_{n+1}p_n + p_{n-1}}{\xi_{n+1}q_n + q_{n-1}} - p_n \nonumber \\
    &= \frac{1}{\xi_{n+1}q_n + q_{n-1}}\parens*{q_n(\xi_{n+1}p_n + p_{n-1}) - p_n(\xi_{n+1}q_n + q_{n-1})} \nonumber \\
    &= \frac{p_{n-1}q_n-p_nq_{n-1}}{\xi_{n+1}q_n + q_{n-1}} \nonumber\\
    &\overset{(\ref{eq:flipflop})}{=} \frac{(-1)^n}{\xi_{n+1}q_n + q_{n-1}} \nonumber\\
    &= \parens*{\frac{(-1)^n}{\xi_{n+1} + \frac{q_{n-1}}{q_n}}}\frac{1}{q_n}.\label{eq:return}
  \end{align}
  Observe that
  \begin{equation*}
    \xi_{n+1} + \frac{q_{n-1}}{q_n}
    = [4a_{n+1}; 4a_{n+2}, 4a_{n+3}, \dots] + [0; 4a_n, 4a_{n-1}, 4a_{n-2}, \dots, 4a_1]. 
  \end{equation*}
  This is minimized (in absolute value) when
  
  \begin{equation*}
    [4a_{n+1}; 4a_{n+2}, 4a_{n+3}, \dots] = \pm [4k; -4k, 4k, -4k, \dots] = \pm \parens*{2k+\sqrt{4k^2-1}}
  \end{equation*}
  and
  
  \begin{equation*} [0; 4a_n, 4a_{n-1}, 4a_{n-2}, \dots, 4a_1] \sim
    \pm [0; -4k, 4k, -4k, \dots ] = \pm \parens*{\sqrt{4k^2-1}-2k}, 
  \end{equation*}
  that is, 
  \begin{equation*}
    \abs*{\xi_{n+1} + \frac{q_{n-1}}{q_n}}
    > 2\sqrt{4k^2-1}.
  \end{equation*}
  Returning this to~(\ref{eq:return}) proves~(\ref{eq:hurwitz}), and
  its optimality is exhibited by $[4k; \overline{-4k, 4k}]$.

  Suppose that the continued fraction does not have
  $\overline{[-4k,4k]}$ in its tail. Then either there are infinitely
  many $n$ for which $\abs{a_n}\geq k+1$, or there are infinitely many
  consecutive pairs $[\pm 4k, \pm 4k]$ appearing. For every $n$ such
  that $\abs{a_{n+1}}\geq k
  +1$, we have
  
  \begin{align*}
    \abs*{\xi_{n+1} + \frac{q_{n-1}}{q_n}}
    &= \abs*{[4a_{n+1}; 4a_{n+2}, 4a_{n+3}, \dots] + [0; 4a_n, 4a_{n-1}, 4a_{n-2}, \dots, 4a_1]} \\
    &\geq 4(k+1) - (2k-\sqrt{4k^2-1}) - (2k-\sqrt{4k^2-1})\\
    &= 4 + 2\sqrt{4k^2-1}.
  \end{align*}
  For every $n$ such that $a_n = a_{n+1} = \pm k$, we have
  
  \begin{align*}
    \xi_{n+1} + \frac{q_{n-1}}{q_n}
    &= [(\pm 4k); 4a_{n+2}, 4a_{n+3}, \dots] + [0; (\pm 4k), 4a_{n-1}, 4a_{n-2}, \dots, 4a_1] \\
    &= (\pm 4k) + \frac{1}{\xi_{n+2}} + \frac{1}{(\pm 4k) + \frac{q_{n-2}}{q_{n-1}}}
  \end{align*}
  and
  
  \begin{align*}
    \xi_{n} + \frac{q_{n-2}}{q_{n-1}}
    &= [(\pm 4k) ; (\pm 4k), 4a_{n+2}, 4a_{n+3}, \dots] + [0; 4a_{n-1}, 4a_{n-2}, \dots, 4a_1] \\
    &= (\pm 4k) + \frac{1}{(\pm 4k) + \frac{1}{\xi_{n+2}}} + \frac{q_{n-2}}{q_{n-1}}
  \end{align*}
  Summing, we have
 
  \begin{align*}
    \parens*{\xi_{n+1} + \frac{q_{n-1}}{q_n}} +  \parens*{\xi_{n} + \frac{q_{n-2}}{q_{n-1}}}
    &= (\pm 8k) + f_\pm \parens*{\frac{1}{\xi_{n+2}}} + f_\pm \parens*{\frac{q_{n-2}}{q_{n-1}}},
  \end{align*}
  
  where
  \begin{equation*}
    f_\pm (t) = t + \frac{1}{(\pm 4k) + t}. 
  \end{equation*}
  
  On the interval $[\sqrt{4k^2-1}-2k, 2k-\sqrt{4k^2-1}]$, $f_+(t)$ takes a minimum
  value of $0$ and $f_{-}(t)$ takes a maximum value of $0$, and these
  extrema are attained at the endpoints of the interval. Therefore,
  \begin{equation*}
    \abs*{\parens*{\xi_{n+1} + \frac{q_{n-1}}{q_n}} +  \parens*{\xi_{n} + \frac{q_{n-2}}{q_{n-1}}}}
    >8k,
  \end{equation*}
  and in particular
  \begin{equation*}
    \abs*{\xi_{n+1} + \frac{q_{n-1}}{q_n}} > 4k
    \quad\textrm{or}\quad
    \abs*{\xi_{n} + \frac{q_{n-2}}{q_{n-1}}} > 4k.
  \end{equation*}
  Combining these observations with~(\ref{eq:return}) proves the
  lemma.
\end{proof}

We are now ready to construct families of recurrent trajectories whose
slopes are characterized by certain continued fraction expansions. This will prove parts 2 and 3 of Theorem~\ref{thm:family}.

\begin{proof}[Proof of Theorem~\ref{thm:family}, Parts 2 and 3]
  Consider first the sequence of convergents ${p_n/q_n}$ of $\xi$. In
  both parts 2 and 3, by
  Proposition \ref{prop:foureyslopes}, the convergents are periodic
  directions. By \Cref{lem:areafour}, every cylinder in a
  periodic direction has area 4 and by Corollary~\ref{lem:midpoints}
  the length of a periodic trajectory with slope $p_n/q_n$ is
  $4\sqrt{p_n^2+q_n^2}$. As $n\to\infty$ we have
  $q_n \rightarrow \infty$, therefore, the widths of the cylinders
  converge to $0$.

  When  $\lim a_n a_{n+1}\neq -1$, by
  Lemma~\ref{lem:hurwitz}, the convergents $p_n/q_n$ satisfy
  \begin{equation}\label{eq:convergentapprox}
    q_n \,\abs*{q_n \xi - p_n} < \frac{1}{4}.
  \end{equation}
  infinitely often. In the other case, if $a_i \in \ZZ_\geq 1$ for all
  $i \geq 1$ and the sequence $(a_n)$ does not have a tail of $1s$,
  then~(\ref{eq:convapprox}) and the recurrence
  relations~(\ref{eq:recurrencerelations}) lead to
  \begin{equation*}
    q_n \,\abs*{q_n \alpha - p_n} < \frac{1}{4a_{n+1}}
  \end{equation*}
  for all $n\geq 1$. Because $a_{n+1}\geq 2$ for infinitely many $n$,
  we have
  \begin{equation}\label{eq:convergentapprox8}
    q_n \,\abs*{q_n \alpha - p_n} < \frac{1}{8}
  \end{equation}
  infinitely often. Likewise if
  $a_i \in \ZZ \setminus \set{-1,0,1}$
  for all $i \geq 1$ and the sequence $(a_n)$ does not have a tail of
  alternating $2s$ and $-2s$, using Lemma~\ref{lem:hurwitz}, we know
  $\xi$ satisfies (\ref{eq:convergentapprox8}) infinitely often.

  Now, let $\calO$ be a trajectory with slope $\xi$ and starting point
  $o$ and let $\epsilon >0$ be given. Then choose $k \in \NN$ large
  enough such that the width of every cylinder in the $p_k/q_k$
  direction is less than $\epsilon$. Furthermore, for $\xi$ satisfying
 Part \ref{part:caseConePoint} (respectively \ref{part:caseAll})
  choose $k$ such that $p_k/q_k$ satisfies \eqref{eq:convergentapprox}
  (respectively \eqref{eq:convergentapprox8}).

Next, let $C_k$ be the cylinder with slope $p_k/q_k$ containing $o$ (possibly in the boundary) and an initial, maximal closed segment of $\calO$ starting at $o$. If $o$ is in the interior of $C_k$ (for Part \ref{part:caseAll}), extend this segment back in negative time to meet the boundary of $C_k$ at a point $o_1$. Let $\calO_k$ be this segment starting at $o_1$ through $o$ and contained within $C_k$. 

Beginning at $o_1$ (which is precisely $o$ when $o$ is a cone point) consider a local developing map sending $\calO_k$ and the boundary of $C_k$ containing $o_1$ to $\RR^2$. In $\RR^2$, the image of $\calO_k$ has slope $\xi$ and the image of the boundary curve has slope $p_k/q_k$ and both line segments begin at the origin. See Figure~\ref{fig:stayingin}. If we develop the boundary
curve through its entire length, we obtain a half-open interval
spanning the origin and the point $(4q_k, 4p_k)$ (by Corollary \ref{lem:midpoints}). 

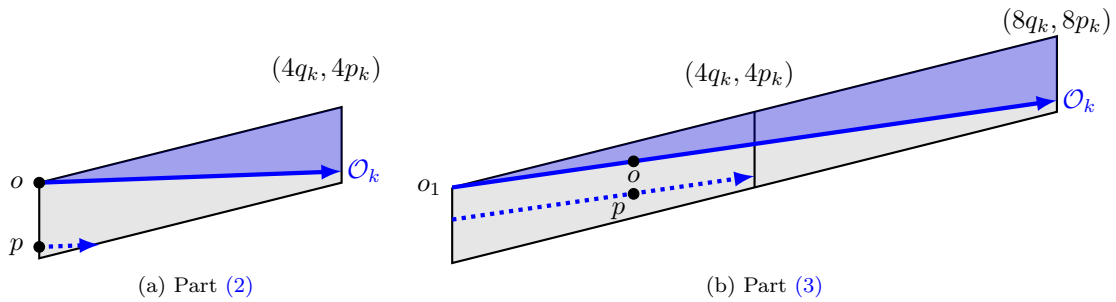
\begin{figure}[h!]
  \centering
\subfloat[Part \ref{part:caseConePoint}]{
  \begin{tikzpicture}[scale = 1]

    \fill[opacity = 0.1] (0,0) -- (4,1) -- (4, 0) -- (0, -1) -- (0,0);

    \draw[thick] (0,0) -- (4,1) -- (4, 0) -- (0, -1) -- (0,0);
    
    
    
    \fill[blue, opacity = 0.3] (0,0) -- (4,1) -- (4, 0.15) -- (0,0);
    \draw[-latex, blue, ultra thick] (0,0) -- (4, 0.15);
    \draw[-latex, blue, ultra thick, dotted] (0, -1+0.15)  --++ (0.2*4, 0.2*0.15);

    \filldraw (0,0) circle (2pt);
    \filldraw (0,-1+0.15) circle (2pt);
    
    \node[blue] at (4.3, 0.15){$\calO_k$};
    \node at (3.8, 1.5) {$(4q_k, 4p_k)$};
    \node at (-0.3, 0) {$o$};
    \node at (-0.3, -1+0.15) {$p$};

  \end{tikzpicture}
  }
  \subfloat[Part \ref{part:caseAll}]{
  \begin{tikzpicture}[scale = 1]

    \fill[opacity = 0.1] (0,0) -- (8,2) -- (8, 1) -- (0, -1) -- (0,0);

    \draw[thick] (0,0) -- (8,2) -- (8, 1) -- (0, -1) -- (0,0);
    
    \draw[thick] (4,1) --++ (0, -1);
    
    
    \fill[blue, opacity = 0.3] (0,0) -- (8,2) -- (8, 1.15) -- (0,0);
    \draw[-latex, blue, ultra thick] (0,0) -- (8, 1.15);
    \draw[-latex, blue, ultra thick, dotted] (0.5*8-4, 0.5*1.15-1)  --++ (0.5*8, 0.5*1.15);

    \filldraw (0.3*8,0.3*1.15) circle (2pt);
    \filldraw (0.3*8,0.3*1.15-0.43) circle (2pt);
    
    \node[blue] at (8.3, 1.15){$\calO_k$};
    \node at (8, 2.2) {$(8q_k, 8p_k)$};
    \node at (3.8, 1.5) {$(4q_k, 4p_k)$};
    \node at (-0.3, 0) {$o_1$};
    \node at (0.3*8,0.3*1.15-0.2) {$o$};
    \node at (0.3*8-0.2,0.3*1.15-0.65) {$p$};

  \end{tikzpicture}
  }
  
  \caption{Schematic showing $C_k$ developed in $\RR^2$ along with the segment $\calO_k$ of the trajectory $\calO$. In Part \ref{part:caseConePoint}, we ensure that the trajectory $\calO$ stays within the cylinder $C_k$ for at least one length of the cylinder such that the point $p$ is within $\epsilon$ of the cone point $o$. In Part \ref{part:caseAll}, we need to ensure that the trajectory $\calO$ stays within the cylinder for at least twice the length of the cylinder so that there is a point $p$  $\epsilon$ of $o$ that is part of the trajectory.}
  \label{fig:stayingin}
\end{figure}

When $\xi$ satisfies \ref{part:caseConePoint}, we developing the segment $\calO_k$ to where it meets the line $x=4q_k$. Then, consider the triangle formed by the origin, the end point of
this segment of $\calO_k$ and $(4q_k, 4p_k)$. The area of this triangle is 
$$8 \cdot q_k\abs{q_k \xi-p_k} < 2,$$ where the inequality follows from (\ref{eq:convergentapprox}). Since the area of $C_k$ is 4, this inequality means that the triangle is contained within $C_k$, ensuring that the segment $\calO_k$ stays inside cylinder $C_k$ for at least the length of the cylinder. Since the width of $C_k$ is less than $\epsilon$ this implies there exists a point $p \in \calO$ within $\epsilon$ of $o$  (on the Mucube), but at least $4\sqrt{p_k^2 + q_k^2}$ (the length of $C_k$) apart with respect to the intrinsic metric of the trajectory $\calO$. As $\epsilon$ was arbitrary, we conclude that the trajectory $\calO$ is recurrent.

When $\xi$ satisfies \ref{part:caseAll}, develop the segment $\calO_k$ to where it meets the line $x=8q_k$ instead, and consider the triangle formed by the origin, the end point of
this segment of $\calO_k$ and $(8q_k, 8p_k)$. In this case, the area of the triangle is
$$32 \cdot q_k\abs{q_k \alpha-p_k} < 4,$$ where the inequality follows from (\ref{eq:convergentapprox8}). Hence, the triangle is contained within \emph{two} developed copies of $C_k$, ensuring that the segment $\calO_k$ stays inside cylinder $C_k$ for at least \emph{twice} the length of the cylinder. Since the width of $C_k$ is less than $\epsilon$ this implies there exists a point $p \in \calO$ within $\epsilon$ of $o$  (on the Mucube) and we conclude recurrence in this case as well. 
\end{proof}

\begin{remark}\label{rem:HausdorffDimension}
    Let $\Xi$ be the set of slopes $\xi$ that satisfy \ref{part:caseConePoint} or \ref{part:caseAll} in Theorem ~\ref{thm:family}. It is contained in the limit set of the Fuchsian group 
    $$H = \left\langle \begin{bmatrix}1 & 4 \\ 0 & 1\end{bmatrix}, \begin{bmatrix} 0 & -1 \\ 1 & 0 \end{bmatrix} \right\rangle \subset \SL_2(\ZZ)$$ (see Section~\ref{sec:Fuchsian} for background on limit sets and Remark~\ref{rem:thm1conseq} for why $\Xi$ is contained in the limit set of $H$). The difference between $\Xi$ and the limit set of $H$ is countable: it is the set of rationals and the irrationals of the form $[4a_0;4a_1, \dots, ]$ which have a tail of $\overline{[-4,4]}$. Hence, by work of Fedosova \cite{Fedosova} on the Hausdorff dimension of limit sets of Fuchsian groups like $H$, we know that the Hausdorff dimension $\Xi$ is approximately 0.68.
\end{remark}

\part{Algebraic Characterization}\label{part:algebraicCharacterization}

In this part we study periodic directions from the algebraic point of
view.  We take a homological approach, with continued focus on
quotients of $M$, and we give an algebraic characterization of the
periodic directions in $M$. We also compute the Veech group of $Y$.

The Mucube is a square-tiled, half-translation surface. We begin by
recalling the basics of these structures, the relationship between
Veech groups of a surface and its cover, and relative homology on
square-tiled surfaces.

\section{Preliminaries for Part \ref{part:algebraicCharacterization}}\label{sec:bg}

\subsection{Translation and Half-translation Surfaces}

A \textbf{half-translation surface} can be defined geometrically as a
collection of polygons embedded in the plane and pairs of edges
identified by translation and possibly a $\pi$-rotation with the
restriction that as one moves along a glued edge, a polygon appears to
the left and to the right. These surfaces are defined up to
equivalence by cut, translate, $\pi$-rotate and paste operations.

Half-translation surfaces can also be defined as pairs $(S, \omega)$
where $\omega$ is a meromorphic quadratic differential on the Riemann
surface $S$ with at most simple poles, if any. Half-translation
surfaces are locally flat except at finitely many \textbf{cone points}
or \textbf{singularities} where they have a cone angle of $\pi \ell$
for some natural number $\ell \neq 2$. A cone point of angle of
$\pi \ell$ corresponds to a zero
 of order $\ell - 1$ when $\ell \gneqq 2$ and a simple pole
when $\ell = 1$.

For each point in $\mathbb{RP}^1$, there is an infinite family of
parallel lines in $\mathbb{R}^2$. One may intersect these lines with
the polygons in $\mathbb{R}^2$ used to define a half-translation
surface $S$.  The intersection creates a foliation of $S$ with
singularities. The singularities occur at the cone points. At times,
there are one-parameter families of pairwise homotopic leaves that are
all homeomorphic to a circle. We call the union of the leaves in one
such family a \textit{cylinder}.

The group $\SL_2(\RR)$ acts on half-translation surfaces via its
linear action on $\RR^2$---for $N \in \SL_2(\RR)$, and $S$ a
half-translation surface, $N\cdot S$ is obtained by acting on the
polygons of $S$ linearly. This preserves the number and angle of the
cone points.

A translation surface is similarly defined geometrically with the
exception that polygons are embedded in the plane up to
  translations and sides are paired via translations
only. Equivalently, they can be thought of as pairs $(S, \omega)$
where $S$ is a Riemann surface and $\omega$ a holomorphic one-form. An
analog of the Gauss-Bonnet theorem holds: the sum of the orders of the
one-form is $2g-2$ where $g$ is the genus of the underlying Riemann
surface.

\begin{figure}[h!!]
\begin{minipage}{0.20\textwidth}
\centering
\begin{tikzpicture}[oct/.style=
 {shape=regular polygon, regular polygon sides=8, draw, minimum width=2cm}]
 
 \node[oct] at (-6*22.5:1.85cm) {};

  \begin{scope}[shift={(-6*22.5:1.85cm)}]
 
  \draw (4*22.5:0.8cm) -- (4*22.5:1.05cm);
  \draw (12*22.5:0.8cm) -- (12*22.5:1.05cm);
  
  \draw [shift={(0,0.025)}] (8*22.5:0.8cm) -- (8*22.5:1.05cm);
  \draw [shift={(0,-0.025)}] (8*22.5:0.8cm) -- (8*22.5:1.05cm);
  
  \draw [shift={(0,0.025)}] (0:0.8cm) -- (0:1.05cm);
  \draw [shift={(0,-0.025)}] (0:0.8cm) -- (0:1.05cm);

  \draw [shift={(-0.05,0.05)}](2*22.5:0.8cm) -- (2*22.5:1.05cm);
 \draw (2*22.5:0.8cm) -- (2*22.5:1.05cm);
 \draw [shift={(0.05,-0.05)}](2*22.5:0.8cm) -- (2*22.5:1.05cm);

  \draw [shift={(-0.05,0.05)}](-6*22.5:0.8cm) -- (-6*22.5:1.05cm);
 \draw (-6*22.5:0.8cm) -- (-6*22.5:1.05cm);
 \draw [shift={(0.05,-0.05)}](-6*22.5:0.8cm) -- (-6*22.5:1.05cm);

 \draw [shift={(0.025,0.025)}](6*22.5:0.8cm) -- (6*22.5:1.05cm);
 \draw [shift={(-0.025,-0.025)}](6*22.5:0.8cm) -- (6*22.5:1.05cm);
 
 \draw [shift={(0.075,0.075)}](6*22.5:0.8cm) -- (6*22.5:1.05cm);
 \draw [shift={(-0.075,-0.075)}](6*22.5:0.8cm) -- (6*22.5:1.05cm);

 \draw [shift={(0.025,0.025)}](-2*22.5:0.8cm) -- (-2*22.5:1.05cm);
 \draw [shift={(-0.025,-0.025)}](-2*22.5:0.8cm) -- (-2*22.5:1.05cm);
 
 \draw [shift={(0.075,0.075)}](-2*22.5:0.8cm) -- (-2*22.5:1.05cm);
 \draw [shift={(-0.075,-0.075)}](-2*22.5:0.8cm) -- (-2*22.5:1.05cm);

  \end{scope}

\end{tikzpicture}
\end{minipage}
\begin{minipage}{0.24\textwidth}
\centering
\begin{tikzpicture}
    [sq/.style=
  {shape=regular polygon, regular polygon sides=4, draw, minimum width=1.414cm}]
  \node[sq] at (0,0){};
  \foreach \i in {1, -1}{
  	\node[sq] at (0,\i){};
  	\node[sq] at (\i,0){};
  	}
  \draw[thin] (0,1.35) -- ++(0,0.3);
  \draw[thin] (0,-1.35) -- ++(0,-0.3);
  
  \draw[thin] (-1.05,0.35) -- ++(0,0.3);
  \draw[thin] (-1.05,-0.35) -- ++(0,-0.3);
  \draw[thin] (-0.95,0.35) -- ++(0,0.3);
  \draw[thin] (-0.95,-0.35) -- ++(0,-0.3);
  
  \draw[thin] (1,0.35) -- ++(0,0.3);
  \draw[thin] (1,-0.35) -- ++(0,-0.3);
  \draw[thin] (1.05,0.35) -- ++(0,0.3);
  \draw[thin] (1.05,-0.35) -- ++(0,-0.3);
  \draw[thin] (0.95,0.35) -- ++(0,0.3);
  \draw[thin] (0.95,-0.35) -- ++(0,-0.3);
  
  \draw[thin] (-0.5,1.1) -- ++(-0.15,-0.15);
  \draw[thin] (-0.5,1.1) -- ++(0.15,-0.15);
  
  \draw[thin] (0.5,1.1) -- ++(-0.15,-0.15);
  \draw[thin] (0.5,1.1) -- ++(0.15,-0.15);
  
\draw[thin] (-0.5,-0.9) -- ++(-0.15,-0.15);
  \draw[thin] (-0.5,-0.9) -- ++(0.15,-0.15);
  \draw[thin] (-0.5,-1.1) -- ++(-0.15,-0.15);
  \draw[thin] (-0.5,-1.1) -- ++(0.15,-0.15);
  
  \draw[thin] (0.5,-0.9) -- ++(-0.15,-0.15);
  \draw[thin] (0.5,-0.9) -- ++(0.15,-0.15);
  \draw[thin] (0.5,-1.1) -- ++(-0.15,-0.15);
  \draw[thin] (0.5,-1.1) -- ++(0.15,-0.15);


\draw[thin] (-1.5,0.1) -- ++(-0.15,-0.15);
  \draw[thin] (-1.5,0.1) -- ++(0.15,-0.15);
  \draw[thin] (-1.5,0) -- ++(-0.15,-0.15);
  \draw[thin] (-1.5,0) -- ++(0.15,-0.15);
  \draw[thin] (-1.5,-0.1) -- ++(-0.15,-0.15);
  \draw[thin] (-1.5,-0.1) -- ++(0.15,-0.15);

  \draw[thin] (1.5,0.1) -- ++(-0.15,-0.15);
  \draw[thin] (1.5,0.1) -- ++(0.15,-0.15);
  \draw[thin] (1.5,0) -- ++(-0.15,-0.15);
  \draw[thin] (1.5,0) -- ++(0.15,-0.15);
  \draw[thin] (1.5,-0.1) -- ++(-0.15,-0.15);
  \draw[thin] (1.5,-0.1) -- ++(0.15,-0.15);

\end{tikzpicture}
\end{minipage}
\begin{minipage}{0.24\textwidth}
\centering
\begin{tikzpicture}
    [sq/.style=
  {shape=regular polygon, regular polygon sides=4, draw, minimum width=1.414cm}]
  \foreach \i in {0,1,2,3}{
  	\node[sq] at (\i,0){};
  	}

  \draw[thin] (0,0.5) -- ++(-0.15,0.15);
  \draw[thin] (0,0.5) -- ++(-0.15,-0.15);
  
  \draw[thin] (1,0.5) -- ++(0.15,0.15);
  \draw[thin] (1,0.5) -- ++(0.15,-0.15);

  \draw[thin] (1.9,-0.5) -- ++(-0.15,0.15);
  \draw[thin] (1.9,-0.5) -- ++(-0.15,-0.15);
   \draw[thin] (2.1,-0.5) -- ++(-0.15,0.15);
  \draw[thin] (2.1,-0.5) -- ++(-0.15,-0.15);

  \draw[thin] (2.9,-0.5) -- ++(0.15,0.15);
  \draw[thin] (2.9,-0.5) -- ++(0.15,-0.15);
   \draw[thin] (3.1,-0.5) -- ++(0.15,0.15);
  \draw[thin] (3.1,-0.5) -- ++(0.15,-0.15);
  
  \draw[thin] (1,-0.65) -- ++(0,0.3);
  \draw[thin] (2,0.65) -- ++(0,-0.3);
  
  \draw[thin] (-0.05,-0.65) -- ++(0,0.3);
  \draw[thin] (0.05,-0.65) -- ++(0,0.3);
  
  \draw[thin] (2.95,0.65) -- ++(0,-0.3);
  \draw[thin] (3.05,0.65) -- ++(0,-0.3);

\end{tikzpicture}

\end{minipage}
\begin{minipage}{0.3\textwidth}
\centering
\begin{tikzpicture}[hex/.style=
 {shape=regular polygon, regular polygon sides=6, draw, minimum width=2cm}]
 
 \node[hex] at (0:0cm) {};
 \node[hex] at (-30:1.72cm) {};
 
 
 	\draw[thin] (0,0.86) -- ++(-0.15,0.15);
  \draw[thin] (0,0.86) -- ++(-0.15,-0.15);
  
  \draw[thin] (1.5,0) -- ++(0.15,0.15);
  \draw[thin] (1.5,0) -- ++(0.15,-0.15);

 
 \draw[thin] (-0.1,-0.86) -- ++(-0.15,0.15);
  \draw[thin] (-0.1,-0.86) -- ++(-0.15,-0.15);
 \draw[thin] (0.1,-0.86) -- ++(-0.15,0.15);
  \draw[thin] (0.1,-0.86) -- ++(-0.15,-0.15);
  
  	 \draw[thin] (1.4,-1.72) -- ++(0.15,0.15);
  \draw[thin] (1.4,-1.72) -- ++(0.15,-0.15);
   \draw[thin] (1.6,-1.72) -- ++(0.15,0.15);
  \draw[thin] (1.6,-1.72) -- ++(0.15,-0.15);

\end{tikzpicture}
\end{minipage}

\caption{Examples of some translation and half-translation surfaces: The two surfaces on the left are translation surfaces and the two on the right are half-translation surfaces. Missing edge identifications denote edges identified with their opposites. }
\end{figure}
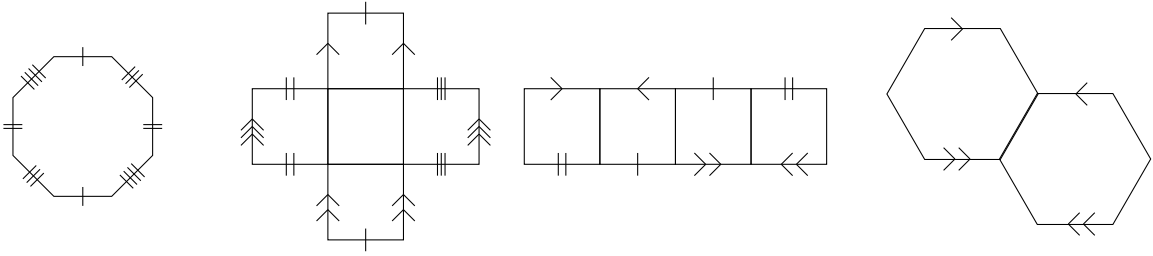

One advantage of a compact, square-tiled translation surface is it can
be easily inputted into a computer (there are SageMath packages
designed for this). Our Mucube is certainly tiled by squares, but it
is neither compact nor a translation surface. On the other hand, the
quotients $X$ and $Y$ we consider are compact, square-tiled,
half-translation surfaces. We work with translation surfaces related
to $X$ and $Y$ as a starting point to first study $X$ and $Y$, and
eventually $M$.

Given a half-translation surface $S$ defined by polygons
$P_1, \dots, P_n$, that is not a translation surface, one can build a

  \textbf{minimal translation cover}
$\tilde{S}$ as follows: Take one copy of $S$ and one of $-\Id \cdot S$
developed disjointly onto the plane $\mathbb{R}^2$. Call an edge $e$
of a half translation surface $\emph{special}$ if it is glued to its
pair via a half-translation (i.e. $\pi$ rotation + translation). For a
given edge $e$ in $S$, let $e'$ be its edge pair. Then we prescribe
the edge identifications to build $\tilde{S}$ as follows:
\begin{itemize}
\item If $e$ is special, glue $e$ to $-\Id\cdot e'$ by translation.
\item If $e$ is not special, glue to $e'$ by translation
\end{itemize}
See Figure \ref{fig:minimaltranslationcover} for an example of this
construction.

\begin{figure}
\begin{center}
\subfloat[]{
   \begin{tikzpicture}[hex/.style=
 {shape=regular polygon, regular polygon sides=6, draw, minimum width=2cm}]
 
 \node[hex, fill=gray!20] at (0:0cm) {};
 \node[hex, fill=gray!20] at (-30:1.72cm) {};
 
 
 	\draw[thick] (0,0.86) -- ++(-0.15,0.15);
  \draw[thick] (0,0.86) -- ++(-0.15,-0.15);
  
  \draw[thick] (1.5,0) -- ++(0.15,0.15);
  \draw[thick] (1.5,0) -- ++(0.15,-0.15);

 
 \draw[thick] (-0.1,-0.86) -- ++(-0.15,0.15);
  \draw[thick] (-0.1,-0.86) -- ++(-0.15,-0.15);
 \draw[thick] (0.1,-0.86) -- ++(-0.15,0.15);
  \draw[thick] (0.1,-0.86) -- ++(-0.15,-0.15);
  
  	 \draw[thick] (1.4,-1.72) -- ++(0.15,0.15);
  \draw[thick] (1.4,-1.72) -- ++(0.15,-0.15);
   \draw[thick] (1.6,-1.72) -- ++(0.15,0.15);
  \draw[thick] (1.6,-1.72) -- ++(0.15,-0.15);

\begin{scope}[shift={(-2,3.5)}]
\node[hex] at (0:0cm) {};
 \node[hex] at (-30:1.72cm) {};

 \node[] at (0, 1.1){$1$};
 \node[] at (30:1.1){$2$};
 \node[] at (-90:1.1){$7$};
 \node[] at (-150:1.1){$2$};
 \node[] at (150:1.1){$5$};

 \begin{scope}[shift={(-30:1.72cm)}]
 \node[] at (0, 1.1){$3$};
\node[] at (30:1.1){$4$};
\node[] at (-30:1.1){$5$};
\node[] at (-90:1.1){$6$};
\node[] at (-150:1.1){$4$};

 \end{scope}
 
 
  

 
  

  \begin{scope}[shift={(5,0)}, rotate=180]
    \node[hex] at (0:0cm) {};
 \node[hex] at (-30:1.72cm) {};

\node[] at (0, 1.1){$3$};
 \node[] at (30:1.1){$10$};
 \node[] at (-90:1.1){$6$};
 \node[] at (-150:1.1){$10$};
 \node[] at (150:1.1){$8$};

 \begin{scope}[shift={(-30:1.72cm)}]
 \node[] at (0, 1.1){$1$};
\node[] at (30:1.1){$9$};
\node[] at (-30:1.1){$8$};
\node[] at (-90:1.1){$7$};
\node[] at (-150:1.1){$9$};
\end{scope}


  

 
  
  
  \end{scope}
 \end{scope}

 \end{tikzpicture}

}
\hspace{0.1\textwidth}
 \subfloat[]{
 \begin{tikzpicture}
    [sq/.style=
  {shape=regular polygon, regular polygon sides=4, draw, minimum width=1.414cm}]
  \foreach \i in {0,1,2,3}{
  	\node[sq, fill=gray!20] at (\i,0){};
  	}

  \draw[thin] (0,0.5) -- ++(-0.15,0.15);
  \draw[thin] (0,0.5) -- ++(-0.15,-0.15);
  
  \draw[thin] (1,0.5) -- ++(0.15,0.15);
  \draw[thin] (1,0.5) -- ++(0.15,-0.15);

  \draw[thin] (1.9,-0.5) -- ++(-0.15,0.15);
  \draw[thin] (1.9,-0.5) -- ++(-0.15,-0.15);
   \draw[thin] (2.1,-0.5) -- ++(-0.15,0.15);
  \draw[thin] (2.1,-0.5) -- ++(-0.15,-0.15);

  \draw[thin] (2.9,-0.5) -- ++(0.15,0.15);
  \draw[thin] (2.9,-0.5) -- ++(0.15,-0.15);
   \draw[thin] (3.1,-0.5) -- ++(0.15,0.15);
  \draw[thin] (3.1,-0.5) -- ++(0.15,-0.15);
  
  \draw[thin] (1,-0.65) -- ++(0,0.3);
  \draw[thin] (2,0.65) -- ++(0,-0.3);
  
  \draw[thin] (-0.05,-0.65) -- ++(0,0.3);
  \draw[thin] (0.05,-0.65) -- ++(0,0.3);
  
  \draw[thin] (2.95,0.65) -- ++(0,-0.3);
  \draw[thin] (3.05,0.65) -- ++(0,-0.3);
  
  
  \begin{scope}[shift={(0, 2)}]
  
  \foreach \i in {0,1,2,3}{
  	\node[sq] at (\i,0){};
  	}

    \node[] at (-0.75,0){$1$};
     \node[] at (3.75,0){$1$};

\foreach \i [evaluate=\i as \j using {int(\i+2)}] in {0,1,2,3}{
  	\node[] at (\i,-0.75){\j};
  	}

\foreach \i [evaluate=\i as \j using {int(\i+6)}] in {0,1}{
  	\node[] at (\i,0.75){\j};
  	}    
      
\node[] at (2, 0.75){3};
\node[] at (3, 0.75){2};

  \end{scope}


   \begin{scope}[shift={(0, 4.5)}]
  
  \foreach \i in {0,1,2,3}{
  	\node[sq] at (\i,0){};
  	}

     \node[] at (-0.75,0){$10$};
     \node[] at (3.75,0){$10$};

\foreach \i [evaluate=\i as \j using {int(\i+8)}] in {0,1}{
  	\node[] at (\i,-0.75){\j};
  	}

\foreach \i [evaluate=\i as \j using {int(\i+4)}] in {2,3}{
  	\node[] at (\i,-0.75){\j};
  	}    

\foreach \i [evaluate=\i as \j using {int(\i+4)}] in {0,1}{
  	\node[] at (\i,0.75){\j};
  	}    
      
\node[] at (2, 0.75){9};
\node[] at (3, 0.75){8};

  \end{scope}
  
\end{tikzpicture}


\label{fig:translationcoverY}
}
\caption{Examples of half-translation surfaces and their minimal translation covers.}
 \label{fig:minimaltranslationcover}
\end{center}
\end{figure}
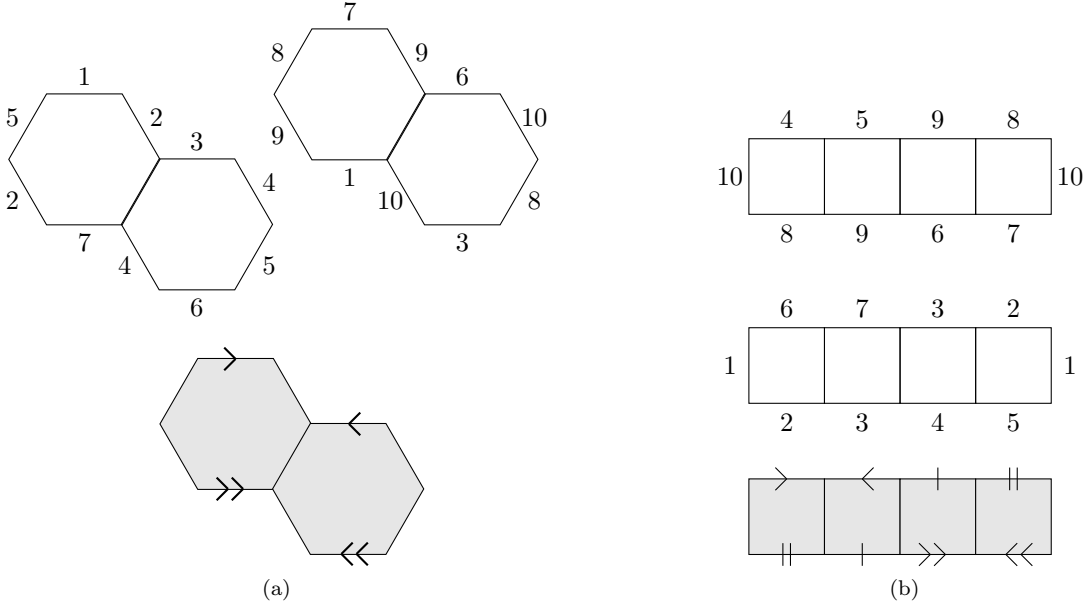

\subsection{Affine diffeomorphisms and the Veech group}

Given a half-translation surface $S$, let $\Aff(S)$ be the group of
affine diffeomorphisms of $S$. There is a derivative map
\begin{equation*}
  D: \Aff(S) \rightarrow \PSL_2(\RR)
\end{equation*}
and the \textbf{Veech group} of $S$ is defined to be the image
$V(S) := D(\Aff(S)) \leq \PSL_2(\RR)$. We note that since half-translation surfaces are equivalent to their images under $-\Id$, the derivative map $D$ takes its values in $\SL_2(\RR) / \ideal{\pm \Id}$ (which is well defined up to $\pm \Id$ for
half-translation surfaces). The translation group $\text{Aut}(S)$, defined as the group of
affine diffeomorphisms of $S$ whose derivative is in
$\{\text{Id}, -\text{Id}\}$ is the kernel of this derivative map $D$, yielding the following short exact sequence:
\[
  0 \rightarrow \text{Aut}(S) \rightarrow \text{Aff}(S)
  \xrightarrow{D} V(S) \rightarrow 0
\]
The first non-trival map in the sequence
is inclusion.

Alternately, the Veech group of $S$ is also the image in $\PSL_2(\RR)$
of the stabilizer of $S$ for the $\SL_2(\RR)$ action.

A translation surface $S$ is called a $\textbf{Veech surface}$ if the
Veech group $V(S)$ is a lattice in $\PSL_2(\RR)$.  Veech surfaces have
remarkable properties, one of which is given by the celebrated Veech
dichotomy (for finite type translation surfaces), which we state
below:

\begin{theorem}[Veech Dichotomy \cite{masur1,veech1}]
  Every straight-line flow on a Veech surface is either periodic or
  uniquely ergodic. 
\end{theorem}

\begin{remark*}
  One important consequence of the Veech dichotomy is that for a Veech
  surface, in any direction, either the surface completely decomposes
  into cylinders or there are no cylinders in that direction at
  all. Note that this is also true of the Mucube $M$, by
  Theorem~\ref{thm:cylinders}. However, $M$ does not satisfy the Veech
  dichotomy because it has directions in which the flow is not
  uniquely ergodic, to wit, the drift directions. Nor is the Veech
  group of $M$ a lattice (Theorem~\ref{thm:veechgrpM}), as we shall
  see.
\end{remark*}

In general, the affine diffeomorphism groups of a regular topological
cover $\hat{S}$ and the base space $S$ are related as follows.  Given
a (half)-translation surface $S$, denote by $S^\circ$ the punctured
surface with cone points removed. Let $N \leq \pi_1( S^\circ)$ be a
normal subgroup and consider $\pi: \hat{S} \rightarrow S^\circ$ the
associated cover. Then the following proposition gives us necessary
and sufficient conditions for affine diffeomorphisms of a surface to
lift to the cover.

\begin{proposition}[Hooper--Weiss \cite{hoopweiss}]\label{prop:HWlift}
  ~
    \begin{enumerate}
    \item An element $f$ of $\Aff(S^\circ)$ lifts to an element of
      $\Aff(\hat{S})$ iff $f$ preserves the normal subgroup of
      $\pi_1(S^\circ)$ corresponding to the cover $\hat{S}$.
    \item An element $\hat{f}$ of $\Aff(\hat{S})$ pushes down to an
      element of $\Aff(S^\circ)$ iff $\hat{f}$ acting by conjugation
      preserves the deck transformation group.
    \end{enumerate}
\end{proposition}

Following Hooper--Weiss \cite{hoopweiss}, we then define the
\textbf{affine diffeomorphism group of a cover}
$\pi:\tilde{S} \rightarrow S$ to be
$$\Aff(\tilde{S},S) =\{(\tilde{f}, f) \in \Aff(\tilde{S}) \times \Aff(S)\mid \pi \circ \tilde{f} = f \circ \pi \}.$$ 
It is known that for any $(\tilde{f}, f) \in \Aff(\tilde{S},S)$, the derivative $D(\tilde{f}) = D(f) \in \PGL_2(\RR)$. Let $V(\tilde{S}, S) = D(\Aff(M, X)) \leq \PGL_2(\RR)$ be defined as the \textbf{Veech group of the cover} $\pi: \tilde{S} \rightarrow S$.

In light of Proposition \ref{prop:HWlift}, we conclude that the Veech
group of a half-translation surface $S$ and its minimal translation
cover $\tilde{S}$ are equal.

\begin{proposition}\label{prop:transcover}
  Let $\tilde{S}$ be the minimal translation cover of a
  half-translation surface $S$. Then $V(S) = V(\tilde{S})$
\end{proposition}

\begin{proof}

Let $g \in V(S)$. We want to show that $g \in V(\tilde{S})$. Now,
write $\tilde{S} = S \cup -Id \cdot S$ where the union is with gluings
as specified above. Then,
$g \cdot \tilde{S} = g \cdot S \cup g \cdot (-Id \cdot S)$.

But note that $-Id$ commutes with $g$. So,
$g \cdot (-Id \cdot S) = -Id \cdot (g S)$.

Since $g$ is in $V(S)$, $gS$ is cut, half-translation and paste equivalent to $S$. We would like to show that $g(\tilde{S})$ is cut, translation and paste equivalent to $\tilde{S}$. We cut both copies $gS$, and $-Id \cdot (g S)$ as we would to rearrange them into $S$ and $-Id \cdot S$.

Again some of the edges the cut up half-translation surface $g S$ are \emph{special} and the others are not. We modify the reassembly of $S$ from a cut up $gS$ to give a reassembly of $\tilde{S}$. For $e$ a special edge of $gS$ and $e'$ its pair in $gS$, we have $-Id \cdot e$ and $-Id \cdot e'$, special edges in $-Id \cdot gS$. We then glue $e$ with $-Id \cdot e'$ and $e'$ with $-Id \cdot e$ via translation.

By construction, this will produce $S \cup -Id \cdot S$ developed in $\mathbb{R}^2$, which we started with. Therefore $g \in V(\tilde{S})$.

To prove the reverse inclusion: $V(\tilde{S}) \leq V(S)$ we begin with a linear transformation $g \in V(\tilde{S})$ and show that $g \in V(S)$. It is sufficient to show that $g$ normalizes the deck transformation group of the covering (see Proposition \ref{prop:HWlift}).

Constructing $\tilde{S}$ as in the proposition above, we can develop $\tilde{S}$ as $S \cup -Id \cdot S$ in $\mathbb{R}^2$. At this point it can be seen that $-Id$ acts on $S \cup -Id \cdot S$ as the generator of the deck transformation group, which is $\ZZ/2\ZZ$. Not only does $g$ normalize $-Id$, it commutes with it. Therefore $g$ is also an element of $V(S)$.
\end{proof}

With assistance from SageMath, Proposition~\ref{prop:transcover} allows us to compute the Veech groups of $X$ and $Y$ (see Section~\ref{sec:MucubeXYHalftranslation}).

\subsection{Square-tiled surfaces and homology}

\begin{wrapfigure}{R}{0.4\textwidth}
\centering
\begin{tikzpicture}
    [sq/.style=
  {shape=regular polygon, regular polygon sides=4, draw, minimum width=1.414cm}]
  \foreach \i in {0,1,2,3}{
  	\node[sq] at (\i,0){};
  	}
  	
  	\draw[->, color=green, line width = 3] (-0.5, -0.5) -- ++ (0,1);
  	\node[] at (-0.8, 0){$\tau_1$};
  	
  	\draw[->, color=green, line width = 3] (-0.5+1, -0.5) -- ++ (0,1);
  	\node[] at (-0.8+1, 0){$\tau_2$};
  	
  	\draw[->, color=green, line width = 3] (-0.5+1+1, -0.5) -- ++ (0,1);
  	\node[] at (-0.8+1+1, 0){$\tau_3$};
  	
  	\draw[->, color=green, line width = 3] (-0.5+1+2, -0.5) -- ++ (0,1);
  	\node[] at (-0.8+1+2, 0){$\tau_4$};

  	\draw[->, color=red, line width = 3] (-0.5, -0.5) -- ++ (1,0);
  	\node[] at (0, -0.9){$\sigma_1$};
  	
  	\draw[->, color=red, line width = 3] (-0.5+1, -0.5) -- ++ (1,0);
  	\node[] at (0+1, -0.9){$\sigma_2$};
  	
  	\draw[->, color=red, line width = 3] (-0.5+2, -0.5) -- ++ (1,0);
  	\node[] at (0+2, -0.9){$\sigma_3$};
  	
  	\draw[->, color=red, line width = 3] (-0.5, 0.5) -- ++ (1,0);
  	\node[] at (0, 0.9){$\sigma_4$};

  \draw[thin] (0,0.5) -- ++(-0.15,0.15);
  \draw[thin] (0,0.5) -- ++(-0.15,-0.15);
  
  \draw[thin] (1,0.5) -- ++(0.15,0.15);
  \draw[thin] (1,0.5) -- ++(0.15,-0.15);

  \draw[thin] (1.9,-0.5) -- ++(-0.15,0.15);
  \draw[thin] (1.9,-0.5) -- ++(-0.15,-0.15);
   \draw[thin] (2.1,-0.5) -- ++(-0.15,0.15);
  \draw[thin] (2.1,-0.5) -- ++(-0.15,-0.15);

  \draw[thin] (2.9,-0.5) -- ++(0.15,0.15);
  \draw[thin] (2.9,-0.5) -- ++(0.15,-0.15);
   \draw[thin] (3.1,-0.5) -- ++(0.15,0.15);
  \draw[thin] (3.1,-0.5) -- ++(0.15,-0.15);
  
  \draw[thin] (1,-0.65) -- ++(0,0.3);
  \draw[thin] (2,0.65) -- ++(0,-0.3);
  
  \draw[thin] (-0.05,-0.65) -- ++(0,0.3);
  \draw[thin] (0.05,-0.65) -- ++(0,0.3);
  
  \draw[thin] (2.95,0.65) -- ++(0,-0.3);
  \draw[thin] (3.05,0.65) -- ++(0,-0.3);

\end{tikzpicture}

\caption{Example of a half-translation square-tiled surface with a choice of relative homology generators.}
\label{fig:examplehomology}
\end{wrapfigure}
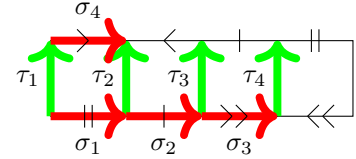

Next we recall the notion of relative and absolute homology in the
setting of surfaces tiled by squares. The first homology relative to
the vertices has two uses in our work: providing an intersection form
and a way to specify ramified covering maps from $M$ to $X$ or
$Y$. These two tools will be used heavily in
Section~\ref{sec:veechgrpchar} to give an algebraic characterization
of periodic directions on $M$.

A translation or half-translation surface is \textbf{square-tiled} if
the polygons used to build the surface are unit squares embedded in
the plane such that their sides are parallel to the axes. It is well
known that square-tiled surfaces are Veech surfaces.

For a square-tiled surface $S$, let $\Sigma$ be the set of corner
points of the squares. To understand the relative homology group
$H_1(S, \Sigma, \ZZ)$ of $S$, we start by orienting and labeling each
square edge. If two edges are glued via translation, then we assign
them the same orientation and label. The \textbf{relative homology
  group} $H_1(S, \Sigma, \ZZ)$ is the module generated by the oriented
and labeled edges along with the relations that the oriented boundary
of each square is 0. For instance, consider the surface in Figure
\ref{fig:examplehomology} with the square edges labeled and oriented
as $\sigma_i$ and $\tau_i$. The relative homology is generated by
$\{\sigma_i, \tau_i\}_{i=1}^4$ under the relations
$$\sigma_1 + \tau_2 - \sigma_4 - \tau_1 = 0; \hspace{0.5cm} \sigma_2 + \tau_3 + \sigma_4 - \tau_2 = 0; \hspace{0.5cm}\sigma_3 + \tau_4 - \sigma_2 - \tau_3 = 0; \hspace{0.5cm} -\sigma_3 + \tau_1 - \sigma_1 - \tau_4 = 0.$$

The \textbf{boundary} of an oriented cycle $\sigma$ from a point $p$ to point $q$ is denoted $\partial \sigma = q - p$. The \textbf{absolute homology} of surface $S$, denoted $H_1(S, \ZZ)$ is then defined to be the submodule of $H_1(S, \Sigma, \ZZ)$ consisting of all elements with $0$ boundary. It is known that the absolute homology $H_1(S, \ZZ)$ as an $\ZZ$-vector space has dimension twice the genus of $S$. Moreover, $H_1(S, \ZZ)$ comes equipped with a bilinear, anti-symmetric form called the algebraic intersection form which will be denoted $i_a(\cdot, \cdot)$. This gives $H_1(S, \ZZ)$ a symplectic structure. 

We  consider homology classes with integer coefficients. The homology group $H_1(S, \ZZ)$ also admits an action by the group of affine automorphisms of $S$, $\text{Aff}(S)$, which preserves the symplectic form so that we have a homomorphism 
\[\tilde{\rho}: \text{Aff}(S) \rightarrow \Sp(H_1(S, \ZZ)).\]

\subsection{Fuchsian groups and limit sets}\label{sec:Fuchsian}

Given a half-translation (or translation) surface $S$, recall that the Veech group $V(S)$ is the image in $\PSL_2(\RR)$ of the stabilizer of $S$ under the $\SL_2(\RR)$ action. Additionally, it is known that Veech groups of finite translation surfaces are discrete. Hence $V(S)$ is a Fuchsian group. As such, we conclude this section by recalling some classical background on Fuchsian groups and limit sets which will be used in Section \ref{sec:density} to prove Theorem \ref{thm:density}. See \cite{katokFuchsian, thurstonGeomTopBook} for comprehensive treatments. 

Recall that $\PSL_2(\RR)$ acts on the upper half plane and its boundary $\HH^2 \cup \partial \HH^2$ via linear fractional transformations where the action is isometric on $\HH^2$ and continuous on the boundary. Given a Fuchsian group $G$, denote by $\Lambda(G)\subseteq$ the \textbf{limit set} of $G$ which is defined to be the set of accumulation points of $G \cdot z$ for any $z \in \HH^2 \cup \{\infty\}$. Importantly, it is known that the definition of $\Lambda(G)$ is independent of the choice of $z$. It is also known that $\Lambda (G) \subseteq \partial \HH^2 \cong \RR \cup \{\infty\}$. Define a $\xi \in \partial \HH^2$ to be a \textbf{cusp point} of $G$ if there exists a parabolic element $g \in G$ such that $g \cdot \xi = \xi$. 

A Fuchsian group $G$ is called \textbf{elementary} if there exists $z \in \HH^2 \cup \partial \HH^2$ such that the $G\cdot z$ is finite, and \textbf{non-elementary} otherwise. It is known (see Exercise 3.8 of \cite{katokFuchsian}) that $G$ is non-elementary if and only if $\Lambda(G)$ contains more than 2 points. 

Additionally, we recall the following well known facts regarding limit sets of Fuchsian groups:

\begin{proposition}\label{prop:generalFuchsianprops}Let $G < \SL_2(\RR)$ be a non-elementary Fuchsian group.

\begin{enumerate}[label={(\arabic*)},ref={\theproposition~(\arabic*)}]
   
    \item \label{part:finiteindexlimitset}
    If $\HH^2/G$ has finite hyperbolic area, then $\Lambda(G) = \partial \HH^2 \simeq \RR \cup \{\infty\}$. Consequently, if $G$ is a finite index subgroup of $\PSL_2(\ZZ)$, then $\Lambda(G) = \partial \HH^2 \simeq \RR \cup \{\infty\}$.
    
    \item \label{part:fuchsianfirstkind} Assume $\Lambda(G) = \partial \HH^2$. Then $G$ is finitely generated if and only if the quotient $\HH^2 / G$ is finite volume. 
    \item \label{part:normallimitset}If $G$ is non-elementary and $N < G$ is a non-trivial normal subgroup of $G$, then $\Lambda(N) = \Lambda(G)$.

    \item \label{part:cuspsaredenseinlimitset} If $G$ contains any parabolic elements, then the closure of the set of cusps of $G$ coincides with $\Lambda(G)$.
\end{enumerate}
\end{proposition}
For \ref{part:finiteindexlimitset} and \ref{part:fuchsianfirstkind}, see Theorems 4.5.2 and 4.5.1 of \cite{katokFuchsian}. For \ref{part:normallimitset} see Corollary 8.1.3 of \cite{thurstonGeomTopBook}. For \ref{part:cuspsaredenseinlimitset}, we note that $\Lambda(G)$ by definition is the set of accumulation points of $G \cdot z$ for any $z \in \HH^2$. As the action of $G$ extends continuously to the boundary $\partial \HH^2$, $\Lambda(G)$ is also the set of accumulation points of $G \cdot \xi$ where $\xi$ is a cusp point. However, every point in $G \cdot \xi$ is a cusp point as well, so the set of cusp points of $G$ is dense in $\Lambda (G)$.

In the context of translation/half-translation surfaces, it is known (see for instance \cite{hubertschmidt}) that Veech groups are never co-compact (i.e. the quotient $\HH^2/V(S)$ is never compact). Therefore, Veech groups always have cusps.
Additionally, for a Veech surface, $S$, since the Veech group, $V(S)$ is a lattice, it is known that in any direction determined by a cusp in $\partial \HH^2$, the surface completely decomposes into cylinders of commensurable moduli. Moreover, the number of $V(S)$-equivalence classes of cusps is finite in this case.

\section{Half-translation structure on the Mucube and the quotients $X$ and $Y$}\label{sec:MucubeXYHalftranslation}

In this section we will endow the Mucube and its two relevant quotients $X$ and $Y$ with half-translation structures making the Mucube an infinite square-tiled half-translation surface and $X$ and $Y$ finite square-tiled half-translation surfaces. As a consequence, a global notion of a vertical direction will be well defined (up to $\pi$-rotation) in these three surfaces. Moreover, a half-translation structure on a surface also implies that any given affine diffeomorphism of the surface has a well-defined derivative in $\PSL_2(\RR)$, and hence a Veech group. The Veech groups of $X$ and $Y$ will play a crucial role in our algebraic characterization of periodic directions in Theorem \ref{thm:characterization}.

\begin{figure}[h!]
\centering 

\subfloat[]{
\begin{tikzpicture}[sq/.style=
  {shape=regular polygon, regular polygon sides=4, draw, minimum width=2.828cm}]

\draw[fill=gray!10] (-1,-1) -- (-1, 1) -- (-1+0.5, 1+0.5) -- (-1+0.5, -1+0.5)--cycle; 
\draw[fill=gray!10] (1,-1) -- (1+0.5, -1+0.5) --(-1+0.5,-1+0.5) -- (-1,-1) --cycle; 
\draw[fill=gray!40]  (-1, 1) -- (-1+0.5, 1+0.5) -- (1+0.5, 1+0.5) -- (1,1) --cycle; 
\draw[fill=gray!40] (1,1) -- (1+0.5, 1+0.5) --(1+0.5,-1+0.5) -- (1,-1) --cycle; 

  \node[sq] at (0,0){};
  \node[sq] at (0.5,0.5){};
  \draw[thin] (-1,-1) -- ++(0.5, 0.5);
  \draw[thin] (-1,1) -- ++(0.5, 0.5);
  \draw[thin] (1,-1) -- ++(0.5, 0.5);
  \draw[thin] (1,1) -- ++(0.5, 0.5);
  
\begin{scope}[shift = {(0.5, 2.5)}]


\draw[fill=gray!10] (-1,-1) -- (-1, 1) -- (-1+0.5, 1+0.5) -- (-1+0.5, -1+0.5)--cycle; 
\draw[fill=gray!10]  (-1+0.5, -1+0.5) -- (-1+0.5, 1+0.5) -- (1+0.5, 1+0.5) -- (1+0.5,-1+0.5) --cycle; 
\draw[fill=gray!40]  (-1, -1) -- (-1, 1) -- (1, 1) -- (1,-1) --cycle; 
\draw[fill=gray!40] (1,1) -- (1+0.5, 1+0.5) --(1+0.5,-1+0.5) -- (1,-1) --cycle; 

 \node[sq] at (0,0){};
  \node[sq] at (0.5,0.5){};
  \draw[thin] (-1,-1) -- ++(0.5, 0.5);
  \draw[thin] (-1,1) -- ++(0.5, 0.5);
  \draw[thin] (1,-1) -- ++(0.5, 0.5);
  \draw[thin] (1,1) -- ++(0.5, 0.5);
\end{scope}

\begin{scope}[shift = {(2.5, 0.5)}]


\draw[fill=gray!10]  (-1+0.5, -1+0.5) -- (-1+0.5, 1+0.5) -- (1+0.5, 1+0.5) -- (1+0.5,-1+0.5) --cycle; 
\draw[fill=gray!10] (1,-1) -- (1+0.5, -1+0.5) --(-1+0.5,-1+0.5) -- (-1,-1) --cycle; 
\draw[fill=gray!40]  (-1, 1) -- (-1+0.5, 1+0.5) -- (1+0.5, 1+0.5) -- (1,1) --cycle; 

\draw[fill=gray!40]  (-1, -1) -- (-1, 1) -- (1, 1) -- (1,-1) --cycle; 

 \node[sq] at (0,0){};
  \node[sq] at (0.5,0.5){};
  \draw[thin] (-1,-1) -- ++(0.5, 0.5);
  \draw[thin] (-1,1) -- ++(0.5, 0.5);
  \draw[thin] (1,-1) -- ++(0.5, 0.5);
  \draw[thin] (1,1) -- ++(0.5, 0.5);
\end{scope}

\draw[thin] (-0.65,0.5) -- ++(0.3, 0);
\draw[thin] (3.5-0.15,0.5) -- ++(0.3, 0);


  


\draw[thin, rotate = 90 ] (-0.15,-0.35-0.5) -- ++(0.15,-0.3);
  \draw[thin,rotate = 90 ] (-0.15+0.15,-0.35-0.3-0.5) -- ++(0.15,0.3);
  
  \draw[thin, rotate = 90 ] (-0.15+1,-0.35-0.5-1) -- ++(0.15,-0.3);
  \draw[thin,rotate = 90 ] (-0.15+0.15+1,-0.35-0.3-0.5-1) -- ++(0.15,0.3);

  \draw[thin] (7.90-8,-0.65-0.25) -- ++(0,-0.3);
  \draw[thin] (8.00-8,-0.65-0.25) -- ++(0,-0.3);
  \draw[thin] (8.10-8,-0.65-0.25) -- ++(0,-0.3);
  
  \draw[thin] (0.90,-0.65+4.5) -- ++(0,0.3);
  \draw[thin] (1.00,-0.65+4.5) -- ++(0,0.3);
  \draw[thin] (1.10,-0.65+4.5) -- ++(0,0.3);


\node[rotate = 90, thick, font=\fontsize{15}{0}\selectfont, thick] at (-1,0){{\color{olive}$>$}};
\node[rotate = 90, thick, font=\fontsize{15}{0}\selectfont, thick] at (4,1){{\color{olive}$>$}};


\node[thick, font=\fontsize{15}{0}\selectfont, thick] at (0,1){{\color{green}$>$}};
\node[ thick, font=\fontsize{15}{0}\selectfont, thick] at (1,2){{\color{green}$>$}};


\node[thick, font=\fontsize{15}{0}\selectfont, thick] at (0.5,-0.5){{\color{black}$>$}};
\node[ thick, font=\fontsize{15}{0}\selectfont, thick] at (0.5,3.5){{\color{black}$>$}};


\node[rotate = 45, thick, font=\fontsize{15}{0}\selectfont, thick] at (3.75,1.75-2){{\color{blue}$>$}};

\node[rotate = 45, thick, font=\fontsize{15}{0}\selectfont, thick] at (3.75-4,1.75+2){{\color{blue}$>$}};


\node[rotate = 45, thick, font=\fontsize{15}{0}\selectfont, thick] at (3.75-2,1.75-2){{\color{purple}$>$}};

\node[rotate = 45, thick, font=\fontsize{15}{0}\selectfont, thick] at (3.75-4+2,1.75+2){{\color{purple}$>$}};
  
\node[rotate = 45, thick, font=\fontsize{15}{0}\selectfont, thick] at (3.75,1.75){{\color{orange}$>$}};

\node[rotate = 45, thick, font=\fontsize{15}{0}\selectfont, thick] at (3.75-4,1.75){{\color{orange}$>$}};



 \begin{scope}[canvas is zy plane at x=1,transform shape]
 \draw ({-0.5},{0}) node[rotate = 90,font=\large]{$0$};
 \end{scope}

\begin{scope}[canvas is xz plane at y=0,transform shape]
\draw ({1},{2}) node[rotate = 180, font=\large]{$1$};
\end{scope}

\begin{scope}[canvas is zy plane at x=-1,transform shape]
\draw ({-0.5},{0}) node[rotate = 90, font=\large]{$2$};
\end{scope}

\begin{scope}[canvas is zx plane at y=2,transform shape]
\draw ({2},{1}) node[rotate=90, font=\large]{$3$};
\end{scope}



\begin{scope}[canvas is xz plane at y=0,transform shape]
\draw ({3},{0.5}) node[rotate = 90, font=\large]{$4$};
\end{scope}

 \begin{scope}[canvas is xy plane at z=-4,transform shape]
 \draw ({1},{-1}) node[rotate = 90, font=\large]{$5$};
 \end{scope}

 \begin{scope}[canvas is yx plane at z=-10,transform shape]
 \draw ({-3},{-1}) node[rotate = 180, font=\large]{$7$};
 \end{scope}

\begin{scope}[canvas is zx plane at y=2,transform shape]
\draw ({0.5},{3}) node[rotate=180,font=\large]{$6$};
\end{scope}


\begin{scope}[canvas is xy plane at z=-4,transform shape]
 \draw ({-1},{1}) node[font=\large]{$8$};
 \end{scope}

 \begin{scope}[canvas is yx plane at z=-10,transform shape]
 \draw ({-1},{-3}) node[rotate =270, font=\large]{$10$};
 \end{scope}

 \begin{scope}[canvas is yz plane at x=1,transform shape]
 \draw ({2},{-2}) node[rotate = 270,font=\large]{$9$};
 \end{scope}

  \begin{scope}[canvas is zy plane at x=-1,transform shape]
 \draw ({-2},{2}) node[font=\large]{$11$};
 \end{scope}
\end{tikzpicture}
\label{fig:threecubehandleX}
}\subfloat[]{\begin{tikzpicture}
    [sq/.style=
  {shape=regular polygon, regular polygon sides=4, draw, minimum width=1.414cm}]

  \foreach \i\j\k in {0/4/8,1/5/9,2/6/10,3/7/11}{
  	\node[sq] at (\i,0){\i};
  	\node[sq] at (\i+3, -2){\j};
  	\node[sq] at (\i+6 ,0){\k};
  	}
  	
  	\draw[thin, color=red, line width = 2] (4,-1.5) -- ++(-0.15,0.15);
  \draw[thin, color=red, line width = 2] (4,-1.5) -- ++(-0.15,-0.15);
  
  \draw[thin, color=red, line width = 2] (0,0.5) -- ++(0.15,0.15);
  \draw[thin, color=red, line width = 2] (0,0.5) -- ++(0.15,-0.15);
  
  	\draw[thin, line width = 2] (6,0.5) -- ++(-0.15,0.15);
  \draw[thin, line width = 2] (6,0.5) -- ++(-0.15,-0.15);
  
  \draw[thin, line width = 2] (1,0.5) -- ++(0.15,0.15);
  \draw[thin, line width = 2] (1,0.5) -- ++(0.15,-0.15);

  	\draw[thin, color=purple, line width = 2] (7,0.5) -- ++(-0.15,0.15);
  \draw[thin, color=purple, line width = 2] (7,0.5) -- ++(-0.15,-0.15);
  
  \draw[thin, color=purple, line width = 2] (3,-1.5) -- ++(0.15,0.15);
  \draw[thin, color=purple, line width = 2] (3,-1.5) -- ++(0.15,-0.15);

  	\draw[thin, color=olive, line width = 2] (2,-0.5) -- ++(-0.15,0.15);
  \draw[thin, color=olive, line width = 2] (2,-0.5) -- ++(-0.15,-0.15);
  
  \draw[thin, color=olive, line width = 2] (6,-2.5) -- ++(0.15,0.15);
  \draw[thin, color=olive, line width = 2] (6,-2.5) -- ++(0.15,-0.15);

  	\draw[thin, color=orange, line width = 2] (5,-2.5) -- ++(-0.15,0.15);
  \draw[thin, color=orange, line width = 2] (5,-2.5) -- ++(-0.15,-0.15);
  
  \draw[thin, color=orange, line width = 2] (9,-0.5) -- ++(0.15,0.15);
  \draw[thin, color=orange, line width = 2] (9,-0.5) -- ++(0.15,-0.15);
  
  	\draw[thin, color=green, line width = 2] (3,-0.5) -- ++(-0.15,0.15);
  \draw[thin, color=green, line width = 2] (3,-0.5) -- ++(-0.15,-0.15);
  
  \draw[thin, color=green, line width = 2] (8,-0.5) -- ++(0.15,0.15);
  \draw[thin, color=green, line width = 2] (8,-0.5) -- ++(0.15,-0.15);

  \draw[thin] (4,-2.65) -- ++(0,0.3);
  \draw[thin ] (2,0.65) -- ++(0,-0.3);
  
  
  \draw[thin] (5.95,-0.65) -- ++(0,0.3);
  \draw[thin] (6.05,-0.65) -- ++(0,0.3);
  
  \draw[thin] (2.95,0.65) -- ++(0,-0.3);
  \draw[thin] (3.05,0.65) -- ++(0,-0.3);
  
  \draw[thin] (7.90,0.65) -- ++(0,-0.3);
  \draw[thin] (8.00,0.65) -- ++(0,-0.3);
  \draw[thin] (8.10,0.65) -- ++(0,-0.3);
  
  \draw[thin] (0.90,-0.65) -- ++(0,0.3);
  \draw[thin] (1.00,-0.65) -- ++(0,0.3);
  \draw[thin] (1.10,-0.65) -- ++(0,0.3);
  

  
  \draw[thin ] (8.80,0.65) -- ++(0,-0.3);
  \draw[thin ] (8.85,0.65) -- ++(0.15,-0.3);
  \draw[thin ] (8.85+0.15,0.65-0.3) -- ++(0.15,0.3);
  
  \draw[thin ] (2.80,-2.35) -- ++(0,-0.3);
  \draw[thin ] (2.85,-2.35) -- ++(0.15,-0.3);
  \draw[thin ] (2.85+0.15,-2.35-0.3) -- ++(0.15,0.3);

  \draw[thin ] (-0.15,-0.35) -- ++(0.15,-0.3);
  \draw[thin ] (-0.15+0.15,-0.35-0.3) -- ++(0.15,0.3);

  \draw[thin ] (5.85,-1.35) -- ++(0.15,-0.3);
  \draw[thin ] (5.85+0.15,-1.35-0.3) -- ++(0.15,0.3);

  \draw[thin ] (6.85,-0.35) -- ++(0.15,-0.3);
  \draw[thin ] (6.85+0.15,-0.35-0.3) -- ++(0.15,0.3);
  \draw[thin ] (7.20,-0.65) -- ++(0,0.3);

  \draw[thin ] (4.85,-1.35) -- ++(0.15,-0.3);
  \draw[thin ] (4.85+0.15,-1.35-0.3) -- ++(0.15,0.3);
  \draw[thin ] (5.20,-1.65) -- ++(0,0.3);

%
%
  
  %

\end{tikzpicture}  	
\label{fig:halftranslationX}
}
\caption{Square-tiled representation of the genus 3 surface $X$}
\end{figure}
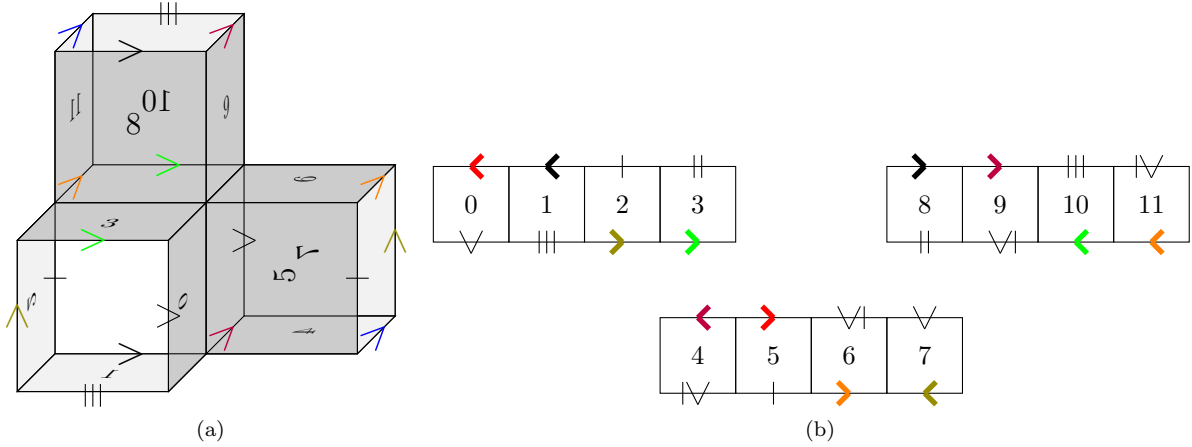

\begin{wrapfigure}{r}{0.3\textwidth}
\centering 
\subfloat[]{\begin{tikzpicture}
    [sq/.style=
  {shape=regular polygon, regular polygon sides=4, draw, minimum width=2.828cm}]

\draw[fill=gray!10] (-1,-1) -- (-1, 1) -- (-1+0.5, 1+0.5) -- (-1+0.5, -1+0.5)--cycle; 
\draw[fill=gray!10] (1,-1) -- (1+0.5, -1+0.5) --(-1+0.5,-1+0.5) -- (-1,-1) --cycle; 
\draw[fill=gray!40]  (-1, 1) -- (-1+0.5, 1+0.5) -- (1+0.5, 1+0.5) -- (1,1) --cycle; 
\draw[fill=gray!40] (1,1) -- (1+0.5, 1+0.5) --(1+0.5,-1+0.5) -- (1,-1) --cycle; 

  \node[sq] at (0,0){};
  \node[sq] at (0.5,0.5){};
  \draw[thin] (-1,-1) -- ++(0.5, 0.5);
  \draw[thin] (-1,1) -- ++(0.5, 0.5);
  \draw[thin] (1,-1) -- ++(0.5, 0.5);
  \draw[thin] (1,1) -- ++(0.5, 0.5);


\node[thick, font=\fontsize{15}{0}\selectfont, thick] at (0,1){{\color{green}$>$}};
\node[rotate=90, thick, font=\fontsize{15}{0}\selectfont, thick] at (-0.5,0.5){{\color{green}$>$}};

\node[rotate = 90, thick, font=\fontsize{15}{0}\selectfont, thick] at (-1,0){{\color{olive}$>$}};
\node[thick, font=\fontsize{15}{0}\selectfont, thick] at (0,-1){{\color{olive}$>$}};

\node[rotate=90, thick, font=\fontsize{15}{0}\selectfont, thick] at (1.5,0.5){{\color{blue}$>$}};
\node[thick, font=\fontsize{15}{0}\selectfont, thick] at (0.25,1.5){{\color{blue}$>$}};

\node[rotate=90, thick, font=\fontsize{15}{0}\selectfont, thick] at (1,0){{\color{black}$>$}};
\node[thick, font=\fontsize{15}{0}\selectfont, thick] at (0.5,-0.5){{\color{black}$>$}};

 \begin{scope}[canvas is zy plane at x=1,transform shape]
 \draw ({-0.5},{0}) node[rotate = 90,font=\large]{$0$};
 \end{scope}

\begin{scope}[canvas is xz plane at y=0,transform shape]
\draw ({1},{2}) node[rotate = 180, font=\large]{$1$};
\end{scope}

\begin{scope}[canvas is zy plane at x=-1,transform shape]
\draw ({-0.5},{0}) node[rotate = 90, font=\large]{$2$};
\end{scope}

\begin{scope}[canvas is zx plane at y=2,transform shape]
\draw ({2},{1}) node[rotate=90, font=\large]{$3$};
\end{scope}
\end{tikzpicture}
\label{fig:Yembedded}
}
\hspace{1cm}
\subfloat[]{\begin{tikzpicture}
    [sq/.style=
  {shape=regular polygon, regular polygon sides=4, draw, minimum width=1.414cm}]
  \foreach \i in {0,1,2,3}{
  	\node[sq] at (\i,0){\i};
  	}


\node[rotate=180, thick, font=\fontsize{15}{0}\selectfont, thick] at (0,-0.5){{\color{black}$>$}};

\node[thick, font=\fontsize{15}{0}\selectfont, thick] at (1,0.5){{\color{black}$>$}};

\node[rotate=180, thick, font=\fontsize{15}{0}\selectfont, thick] at (0,0.5){{\color{blue}$>$}};

\node[thick, font=\fontsize{15}{0}\selectfont, thick] at (3,0.5){{\color{blue}$>$}};

\node[thick, font=\fontsize{15}{0}\selectfont, thick] at (2,0.5){{\color{green}$>$}};

\node[thick, font=\fontsize{15}{0}\selectfont, thick] at (3,-0.5){{\color{green}$>$}};

\node[thick, font=\fontsize{15}{0}\selectfont, thick] at (2,-0.5){{\color{olive}$>$}};

\node[rotate=180, thick, font=\fontsize{15}{0}\selectfont, thick] at (1,-0.5){{\color{olive}$>$}};

\end{tikzpicture}
\label{fig:halftranslationY}
}
\caption{Square-tiled representation of the genus 1 surface $Y$. }
\end{wrapfigure}
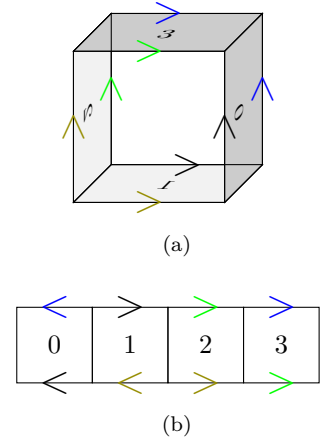

Consider first the surface $X$ pictured in Figure \ref{fig:mucubeFD}. After some cuts, translations by vectors in $2\ZZ$ and pasting, we see that $X$ is made out of unit squares (see Figure \ref{fig:threecubehandleX}). Moreover, when the squares are laid out on the plane as in Figure \ref{fig:halftranslationX}, we see that $X$ is a square-tiled half-translation surface with eight cone points each of angle $3 \pi$. The half-translation structure of $X$ lifts to the Mucube $M$, making it a square-tiled infinite genus half-translation surface with all cone angles measuring $3\pi$.

Referring to Figure~\ref{fig:halftranslationX}, we draw attention to two sets of simple closed curves on $X$ that play an important role in understanding the periodic geodesics of $M$. There are three \textit{horizontal} simple closed geodesics avoiding all cone points that cut every square they meet in half. One traverses squares $0, 1, 2, 3$ and then returns to $0$. Another traverses squares $4, 5, 6, 7$ before returning to square $4$. The final horizontal simple closed curve we consider traverses squares $8, 9, 10, 11$ and then returns to square $8$. We note that these are precisely the core curves of the cylinders in the horizontal cylinder decomposition of $X$.

Similarly, there are three \textit{vertical} simple closed geodesics, avoiding all cone points, that cut every square they meet in half. One traverses $0, 5, 2, 7,$ and then returns to square $0$. Another traverses squares $1, 8, 3, 10$, and the final vertical curve traverses squares $4, 9, 6, 11$ in order. These are precisely the core curves of the cylinders in the vertical cylinder decomposition of $X$.

Next, recall from Section \ref{sec:descentToY} that the surface $Y$ is the quotient of $X$ under the rotational symmetry about the axis $L$ and point $p$ shown in Figure \ref{fig:rotation}. In the polygonal representation of $X$ as in Figure \ref{fig:halftranslationX}, the order 3 rotation about the axis $L$ amounts to cyclically permuting the horizontal cylinders. From the polygonal representation of $Y$ shown in Figure \ref{fig:halftranslationY} we can see that $Y$ is a square-tiled half-translation surface of genus one.

The Veech groups of finite square-tiled translation surfaces can be easily computed using the \verb|surface_dynamics| package in SageMath \cite{surfacedynamics}. To compute the Veech group of $Y$, we first construct its minimal translation cover (See Figure \ref{fig:translationcoverY}). 
Using \verb|surface_dynamics| we compute the Veech groups of the minimal translation covers and using Proposition \ref{prop:transcover} we then have the Veech group of $Y$ as well:

\begin{proposition}\label{prop:veechgrpY} The Veech groups $V(X)$ and $V(Y)$ of $X$ and $Y$ are equal and is the image of the matrix group 
$$\hat{V} := \langle \Theta , A, B | \Theta^4, \Theta^2 B\Theta^2B^{-1}, \Theta^2 A\Theta^{-2}A^{-1}\rangle$$ in $\PSL_2(\ZZ)$ where 
$$\Theta = \begin{bmatrix} 0 & -1 \\ 1 & 0 \end{bmatrix},\, A =  \begin{bmatrix} 1 & 4 \\ 0 & 1 \end{bmatrix},\text{ and }\, B = \begin{bmatrix} 5 & -8 \\ 2 & -3 \end{bmatrix}.$$
\end{proposition}

Subsequently, we will utilize some properties of the group $\hat{V}$ (and consequently $V(Y) = \P\hat{V}$) which we state below:

\begin{proposition}\label{prop:YVeechGroupFacts}
The group $\hat{V}$ has finite index in $\SL_2(\ZZ)$ and has 3 cusp equivalence classes with representatives $\infty$, $1$ and $2$. Moreover, the $\infty$-cusp represents all the one-cylinder directions on $Y$.
\end{proposition}

\section{Fundamental group of the Mucube}\label{sec:MucubeFundamentalGroup}

\begin{figure}[h!]
\centering
\includegraphics[width=3.25in]{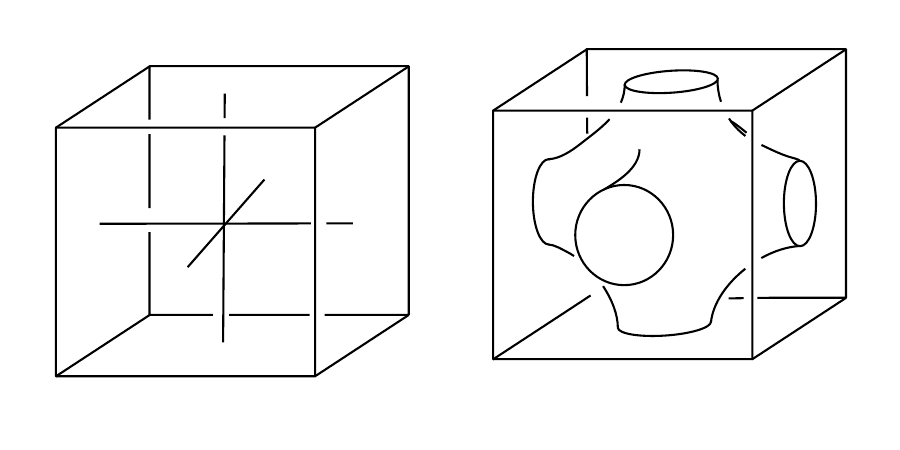}
\caption{Topological surface corresponding to the building block of the Mucube}
\label{fig:mucubetop}
\end{figure}

In subsequent sections we will use Proposition \ref{prop:HWlift} to
understand the affine diffeomorphisms of the Mucube, $M$, via the
affine diffeomorphisms of $X$ and $Y$. In order to do so, we need some
understanding of the fundamental group of $M$, which we develop in
this section.

We start by revisiting the construction of the Mucube in Section
\ref{sec:intro} through a topological lens. Recall that we embedded
a surface with boundary inside the cube $\cube$ with vertices
$(\pm 1, \pm 1, \pm 1)$. Now, suppose that opposite faces of $\cube$
were identified to one another (front face identified to back, top
identified to bottom, and right identified to left) to obtain a
3-manifold. This $3$-manifold is none other than the $3$-torus,
$S^1 \times S^1 \times S^1$, with a flat metric. The surface embedded
in the interior of $\cube$ (after the drilling) would then become the
surface $X$ of genus three embedded inside the $3$-torus (see Figure
\ref{fig:mucubetop}). If we then take the universal cover of the
$3$-torus, the lift of the genus three surface embedded within, will
give us the Mucube.

Let $\TT^3$ denote the $3$-torus. The embedding
$\iota: X\rightarrow \TT^3$ induces a group homomorphism
$\iota_*:\pi_1(X) \rightarrow \pi_1 (\TT^3)$, where
$\pi_1(\TT^3) \simeq \ZZ^3$. Let $\mathcal{K} \leq \pi_1(X)$ be the
kernel of the homomorphism $\iota_*$. Then then $M$ is the cover of
$X$ corresponding to $\mathcal{K}$. The set of three horizontal and
three vertical geodesics on $X$ mentioned in Section
\ref{sec:MucubeXYHalftranslation} can be lifted to $M$, and we call
their lifts \textit{horizontal} and \textit{vertical} geodesics in
$M$. The following proposition explains the significance of these
curves in relation to the fundamental group of $M$.


\begin{proposition}\label{prop:genforfundgroupM}
  There exists a generating set of $\pi_1(M)$ that consists of curves
  that are freely homotopic to horizontal or vertical geodesics.
\end{proposition}

\begin{proof}
  We begin by finding a generating set of $\pi_1(X)$ that consists of
  elements that are freely homotopic to horizontal or vertical
  geodesics.

  Consider the six based loops $a,b,c,x,y,z \subset X$ shown in
  Figure~\ref{fig:generatorsX}.  An Euler characteristic computation verifies that
  \begin{equation*}
    X - a\cup b\cup c\cup x\cup y\cup z
  \end{equation*}
  is a topological disk. Therefore, the set $\set{a,b,c,x,y,z}$ generates
  $\pi_1(X)$.

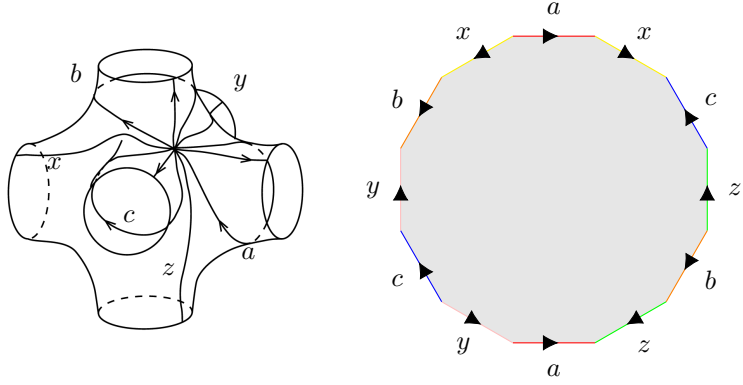
\begin{figure}
\centering
\begin{tikzpicture}
\node (img) at (0,0) {\includegraphics[scale=0.145]{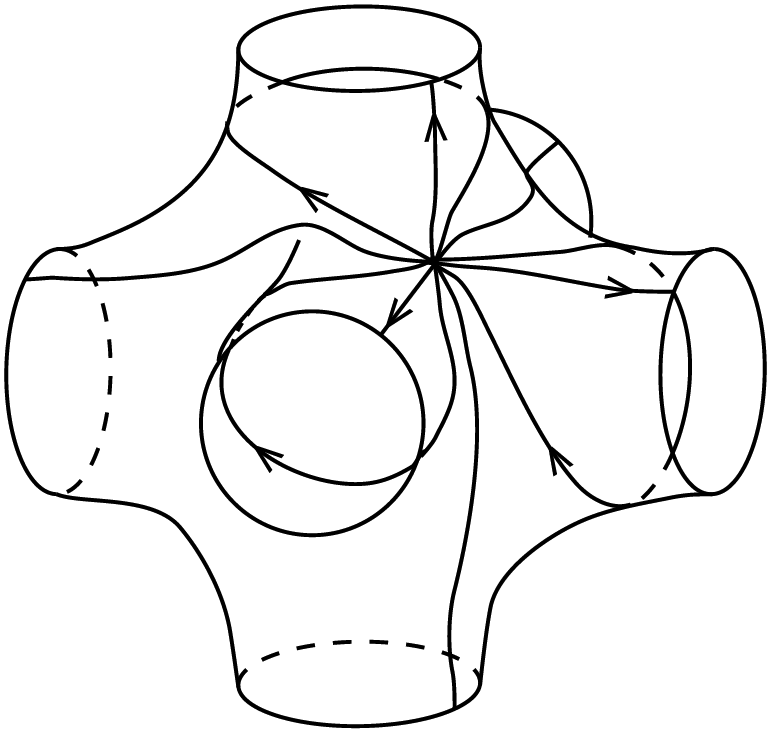}};
\node at (35pt, -24pt) {$a$};
\node at (-30pt, 44pt) {$b$};
\node at (-10pt, -10pt) {$c$};
\node at (-38pt, 10pt) {$x$};
\node at (32pt, 40pt) {$y$};
\node at (5pt, -30pt) {$z$};
\begin{scope}[xshift=150, xscale=0.7, yscale=0.7]
    [
  ultra thick,
  every node/.style={font=\small}
]

  \coordinate (O) at (0,0);

  \foreach \i [evaluate=\i as \ang using {105-30*(\i-1)}] in {1,...,12} {
    \coordinate (P\i) at (\ang:3);
  }

  \fill[gray!20] (P1) -- (P2) -- (P3) -- (P4) -- (P5) -- (P6)
                 -- (P7) -- (P8) -- (P9) -- (P10) -- (P11) -- (P12) -- cycle;

  \draw[red] (P1) -- (P2);
  \draw[-{Latex[length=6pt,width=6pt]}] ($(P1)!0.44!(P2)$) -- ($(P1)!0.56!(P2)$);
  \node at ($(O)!1.18!($(P1)!0.5!(P2)$)$) {$a$};

  \draw[yellow] (P2) -- (P3);
  \draw[-{Latex[length=6pt,width=6pt]}] ($(P2)!0.44!(P3)$) -- ($(P2)!0.56!(P3)$);
  \node at ($(O)!1.18!($(P2)!0.5!(P3)$)$) {$x$};

  \draw[blue] (P3) -- (P4);
  \draw[-{Latex[length=6pt,width=6pt]}] ($(P3)!0.56!(P4)$) -- ($(P3)!0.44!(P4)$);
  \node at ($(O)!1.18!($(P3)!0.5!(P4)$)$) {$c$};

  \draw[green] (P4) -- (P5);
  \draw[-{Latex[length=6pt,width=6pt]}] ($(P4)!0.56!(P5)$) -- ($(P4)!0.44!(P5)$);
  \node at ($(O)!1.18!($(P4)!0.5!(P5)$)$) {$z$};

  \draw[orange] (P5) -- (P6);
  \draw[-{Latex[length=6pt,width=6pt]}] ($(P5)!0.44!(P6)$) -- ($(P5)!0.56!(P6)$);
  \node at ($(O)!1.18!($(P5)!0.5!(P6)$)$) {$b$};

  \draw[green] (P6) -- (P7);
  \draw[-{Latex[length=6pt,width=6pt]}] ($(P6)!0.44!(P7)$) -- ($(P6)!0.56!(P7)$);
  \node at ($(O)!1.18!($(P6)!0.5!(P7)$)$) {$z$};

  \draw[red] (P7) -- (P8);
  \draw[-{Latex[length=6pt,width=6pt]}] ($(P7)!0.56!(P8)$) -- ($(P7)!0.44!(P8)$);
  \node at ($(O)!1.18!($(P7)!0.5!(P8)$)$) {$a$};

  \draw[pink] (P8) -- (P9);
  \draw[-{Latex[length=6pt,width=6pt]}] ($(P8)!0.56!(P9)$) -- ($(P8)!0.44!(P9)$);
  \node at ($(O)!1.18!($(P8)!0.5!(P9)$)$) {$y$};

  \draw[blue] (P9) -- (P10);
  \draw[-{Latex[length=6pt,width=6pt]}] ($(P9)!0.44!(P10)$) -- ($(P9)!0.56!(P10)$);
  \node at ($(O)!1.18!($(P9)!0.5!(P10)$)$) {$c$};

  \draw[pink] (P10) -- (P11);
  \draw[-{Latex[length=6pt,width=6pt]}] ($(P10)!0.44!(P11)$) -- ($(P10)!0.56!(P11)$);
  \node at ($(O)!1.18!($(P10)!0.5!(P11)$)$) {$y$};

  \draw[orange] (P11) -- (P12);
  \draw[-{Latex[length=6pt,width=6pt]}] ($(P11)!0.56!(P12)$) -- ($(P11)!0.44!(P12)$);
  \node at ($(O)!1.18!($(P11)!0.5!(P12)$)$) {$b$};

  \draw[yellow] (P12) -- (P1);
  \draw[-{Latex[length=6pt,width=6pt]}] ($(P12)!0.56!(P1)$) -- ($(P12)!0.44!(P1)$);
  \node at ($(O)!1.18!($(P12)!0.5!(P1)$)$) {$x$};

\end{scope}
\end{tikzpicture}
\captionof{figure}{Generators of the fundamental group of $X$}
\label{fig:generatorsX}
\end{figure}

Recall that $\iota:X \rightarrow \TT^3$ is the embedding of $X$ into
the three-dimensional torus, shown in Figure~\ref{fig:mucubetop}. Let
$\iota_* : \pi_1(X) \rightarrow \pi_1(\TT^3)$ be the homomorphism
induced on fundamental groups by $\iota$. Note that
$\iota_*(a)=\iota_*(b)=\iota_*(c)= \mathds{1}$ and that
$\iota_*(x)=x, \iota_*(y)=y$ and $\iota_*(z)=z$, where
$\pi_1(\TT^3)= \mathbb{Z}^3 = \langle x, y, z | [x,y] = [y,z] =[x,z] = \mathds{1} \rangle$. The Mucube is the cover of $X$
corresponding to the kernel $\mathcal{K}$ of $\iota_*$, therefore,
$\pi_1(M)= \mathcal{K}$. We will show that $\mathcal{K}$ is generated
by elements that are freely homotopic to horizontal and vertical
curves.

From the inclusion $\iota: X \rightarrow \TT^3$, it is clear that
$a, b, c$ are in $\mathcal{K}$, and $x, y, z$ are not. Since $\pi_1(\TT^3)$ is a free abelian
group and kernels are normal, $\mathcal{K}$ must be the normal
closure of $a, b, c$ and the commutator subgroup of
$\langle x, y, z \rangle \leq \pi_1 (X)$. In other words,
$\mathcal{K} = \langle \langle a, b, c, [x,y], [y,z], [x,z] \rangle
\rangle$. The generators $a, b, c$ are clearly freely homotopic to
horizontal geodesics. It remains to show that the commutator subgroup
of $\langle x,y,z \rangle$ is generated by
elements that are freely homotopic to horizontal and vertical curves.

Consider the element $[z,x]c = zxz^{-1}x^{-1}c \in \mathcal{K}$. Conjugating by $x^{-1}z^{-1}$, we obtain   $z^{-1}x^{-1}czx \in \mathcal{K}$. This element is homotopic to a vertical cylinder. See Figure~\ref{verticalgenerator}, where a lift of this curve to $M$ is shown. In this setting, the free homotopy to a vertical curve is evident. This homotopy is then projected to $X$ to give the desired result.

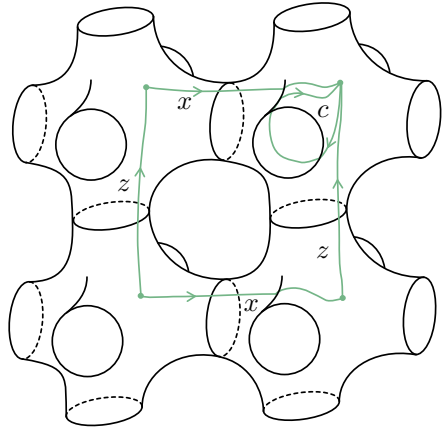
\begin{figure}
\centering
\begin{tikzpicture}
\node (img) at (0,0) {\includegraphics[scale=0.7]{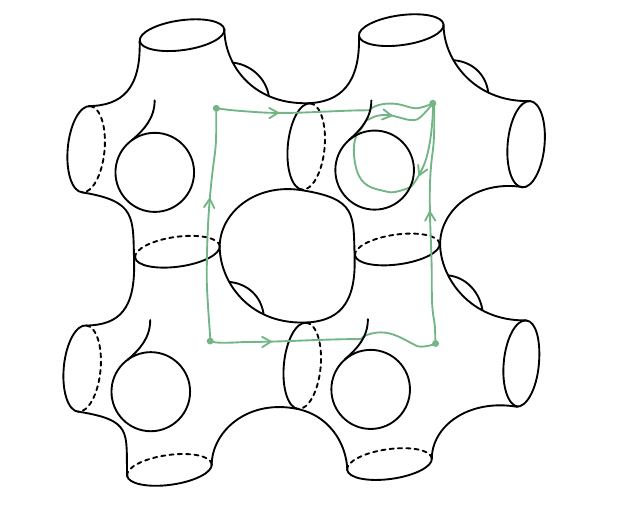}};
\node at (32pt, -10pt) {$z$};
\node at (-20pt, 46pt) {$x$};
\node at (-43pt, 15pt) {$z$};
\node at (32pt, 43pt) {$c$};
\node at (5pt, -30pt) {$x$};
\end{tikzpicture}
\caption{A lift of $z^{-1}x^{-1}czx$ is shown in $M$. Reading concatenation from right to left and starting at the bottom left corner, one can see the shown curve matches the given element.}
\label{verticalgenerator}
\end{figure}

Similarly, the element $yxy^{-1}x^{-1}b^{-1}$ is in $\mathcal{K}$. When conjugated by
$y$ it becomes $xy^{-1}x^{-1}b^{-1}y$ which is in $\mathcal{K}$ as well,
since $\mathcal{K}$ is normal. This element is homotopic to a vertical
curve. Finally, $z^{-1}y^{-1}a^{-1}zy$ is in
$\mathcal{K}$ and is homotopic to a vertical cylinder. Let
$k=xy^{-1}x^{-1}b^{-1}y$, $l=z^{-1}y^{-1}czy$, and
$m=z^{-1}a^{-1}x^{-1}zx$.

We have just shown that $\mathcal{K}\geq \langle \langle a, b, c, k, l, m \rangle \rangle$.
We will now show
that the the normal closure
$\langle \langle a, b, c, k, l, m \rangle \rangle$ is equal to
$\mathcal{K}$. To show that $\mathcal{K}$ is a subgroup of
$\langle \langle a, b, c, k, l, m\rangle \rangle$, we show containment
of each of the generators.

The first three generators $a, b$ and $c$ are contained trivially. The generator $[x, y]$ can be written as $(yky^{-1}b^{-1})^{-1}$, which consists of the product of conjugates of $ a, b, c, k, l, m$. Therefore $[x,y]$ must be contained in the normal closure $\langle \langle a, b, c, k, l, m \rangle \rangle$. Similarly $[y,z]$ can be written as $((zy) l (zy)^{-1}c^{-1})^{-1}$, and $[x,z]$ can be written as $(x(azmz^{-1})x^{-1})^{-1}$. Since the $\langle \langle a, b, c, k, l, m \rangle \rangle$ is normal, it must contain the smallest normal group containing $a, b, c, [x,y], [y,z], [x,z]$ (i.e. the normal closure, $\mathcal{K}$). Thus $\mathcal{K} = \langle \langle a, b, c, k, l, m \rangle \rangle$.
\end{proof}


\section{Characterization of periodic directions on $M$ via $V(Y)$}\label{sec:veechgrpchar}

In this section we will prove Theorem \ref{thm:veechgrpchar}, which gives us a characterization of periodic directions on the Mucube, $M$, via an infinitely generated group $\Gamma$. From Lemma \ref{prop:veechgrpY}, we know that the $V(Y) =  \P\ideal{\Theta , A, B} < \SL_2(\ZZ)$ where $\Theta = \begin{bmatrix}0 & -1 \\ 1 & 0 \end{bmatrix}$, $A =  \begin{bmatrix}1 & 4 \\ 0 & 1 \end{bmatrix}$, and $B=  \begin{bmatrix}5 & -8 \\ 2 & -3 \end{bmatrix}$. Recall that for a group $G < \SL_2\RR$, $\P G$ is its projective image in $\PSL_2\RR$. Define,
    $$\Gamma = \ideal{A, \{h \Theta h^{-1} : h \in \ideal{\Theta, A, B}\}}$$
Note that $\P\Gamma < V(Y)$. 

We then have the following characterization:
\begin{customtheorem}{1}
A slope $(p,q)$ is periodic on $M$ if and only if there exists $N \in \Gamma$ such that $N\cdot (1,0) = (p,q)$.
\end{customtheorem}

To prove the theorem, we take a closer look at the absolute homology $H_1(Y, \ZZ)$.

\subsection{Homology of $Y$}

The homology of $Y$ relative to $\Sigma$ denoted $H_1(Y, \Sigma, \ZZ)$ is the module generated by $\{\sigma_i, \tau_i\}_{i=1}^4$ under the relations
\begin{align}\label{eqn:homologyrelations}
\begin{split}
\sigma_1 + \tau_2 &= \tau_1 + \sigma_4,\\
\sigma_2 + \tau_3 &= \tau_2 - \sigma_4,\\
\sigma_3+ \tau_4 &= \tau_3 + \sigma_2,\\
-\sigma_3 + \tau_1 &= \tau_4 + \sigma_1.\\
\end{split}
\end{align}
These relations represent the fact that any curve that bounds a disk is trivial in homology.

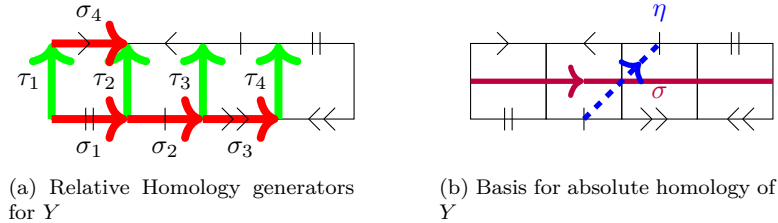
\begin{figure}[h!]
\begin{center}
\subfloat[Relative Homology generators for $Y$]
{
\begin{tikzpicture}
    [sq/.style=
  {shape=regular polygon, regular polygon sides=4, draw, minimum width=1.414cm}]
  \foreach \i in {0,1,2,3}{
  	\node[sq] at (\i,0){};
  	}
  	
  	\draw[->, color=green, line width = 3] (-0.5, -0.5) -- ++ (0,1);
  	\node[] at (-0.8, 0){$\tau_1$};
  	
  	\draw[->, color=green, line width = 3] (-0.5+1, -0.5) -- ++ (0,1);
  	\node[] at (-0.8+1, 0){$\tau_2$};
  	
  	\draw[->, color=green, line width = 3] (-0.5+1+1, -0.5) -- ++ (0,1);
  	\node[] at (-0.8+1+1, 0){$\tau_3$};
  	
  	\draw[->, color=green, line width = 3] (-0.5+1+2, -0.5) -- ++ (0,1);
  	\node[] at (-0.8+1+2, 0){$\tau_4$};

  	\draw[->, color=red, line width = 3] (-0.5, -0.5) -- ++ (1,0);
  	\node[] at (0, -0.9){$\sigma_1$};
  	
  	\draw[->, color=red, line width = 3] (-0.5+1, -0.5) -- ++ (1,0);
  	\node[] at (0+1, -0.9){$\sigma_2$};
  	
  	\draw[->, color=red, line width = 3] (-0.5+2, -0.5) -- ++ (1,0);
  	\node[] at (0+2, -0.9){$\sigma_3$};
  	
  	\draw[->, color=red, line width = 3] (-0.5, 0.5) -- ++ (1,0);
  	\node[] at (0, 0.9){$\sigma_4$};

  \draw[thin] (0,0.5) -- ++(-0.15,0.15);
  \draw[thin] (0,0.5) -- ++(-0.15,-0.15);
  
  \draw[thin] (1,0.5) -- ++(0.15,0.15);
  \draw[thin] (1,0.5) -- ++(0.15,-0.15);

  \draw[thin] (1.9,-0.5) -- ++(-0.15,0.15);
  \draw[thin] (1.9,-0.5) -- ++(-0.15,-0.15);
   \draw[thin] (2.1,-0.5) -- ++(-0.15,0.15);
  \draw[thin] (2.1,-0.5) -- ++(-0.15,-0.15);

  \draw[thin] (2.9,-0.5) -- ++(0.15,0.15);
  \draw[thin] (2.9,-0.5) -- ++(0.15,-0.15);
   \draw[thin] (3.1,-0.5) -- ++(0.15,0.15);
  \draw[thin] (3.1,-0.5) -- ++(0.15,-0.15);
  
  \draw[thin] (1,-0.65) -- ++(0,0.3);
  \draw[thin] (2,0.65) -- ++(0,-0.3);
  
  \draw[thin] (-0.05,-0.65) -- ++(0,0.3);
  \draw[thin] (0.05,-0.65) -- ++(0,0.3);
  
  \draw[thin] (2.95,0.65) -- ++(0,-0.3);
  \draw[thin] (3.05,0.65) -- ++(0,-0.3);

\end{tikzpicture}}
\hspace{1cm}
\subfloat[Basis for absolute homology of $Y$]
{\begin{tikzpicture}
    [sq/.style=
  {shape=regular polygon, regular polygon sides=4, draw, minimum width=1.414cm}]
  \foreach \i in {0,1,2,3}{
  	\node[sq] at (\i,0){};
  	}

  \draw[thin] (0,0.5) -- ++(-0.15,0.15);
  \draw[thin] (0,0.5) -- ++(-0.15,-0.15);
  
  \draw[thin] (1,0.5) -- ++(0.15,0.15);
  \draw[thin] (1,0.5) -- ++(0.15,-0.15);

  \draw[thin] (1.9,-0.5) -- ++(-0.15,0.15);
  \draw[thin] (1.9,-0.5) -- ++(-0.15,-0.15);
   \draw[thin] (2.1,-0.5) -- ++(-0.15,0.15);
  \draw[thin] (2.1,-0.5) -- ++(-0.15,-0.15);

  \draw[thin] (2.9,-0.5) -- ++(0.15,0.15);
  \draw[thin] (2.9,-0.5) -- ++(0.15,-0.15);
   \draw[thin] (3.1,-0.5) -- ++(0.15,0.15);
  \draw[thin] (3.1,-0.5) -- ++(0.15,-0.15);
  
  \draw[thin] (1,-0.65) -- ++(0,0.3);
  \draw[thin] (2,0.65) -- ++(0,-0.3);
  
  \draw[thin] (-0.05,-0.65) -- ++(0,0.3);
  \draw[thin] (0.05,-0.65) -- ++(0,0.3);
  
  \draw[thin] (2.95,0.65) -- ++(0,-0.3);
  \draw[thin] (3.05,0.65) -- ++(0,-0.3);
  
  \node[] at (-0.8, 0){};
  \node[] at (0, -0.9){{\color{white}$\sigma_1$}};
  
\node[color=purple] at (2,-0.15){$\sigma$};  
  \draw[->, line width = 2, color = purple] (-0.5, 0) -- ++ (1.5, 0);
  \draw[line width = 2, color = purple] (1, 0) -- ++ (2.5, 0);
  
  \node[color=blue] at (2, 0.9){$\eta$};
  \draw[->, line width = 2, dashed, color = blue] (1, -0.5) -- (1.75, 0.25);
  \draw[dashed, line width = 2, color = blue] (1.75, 0.25) -- (2, 0.5);

\end{tikzpicture}}
\caption{Homology of $Y$}
\label{fig:Yhomologyetasigma}
\end{center}

\end{figure}

Now let $\sigma$ be the homology class of $\gamma_{(1,0)}$, the core curve of the horizontal cylinder. Note that $\sigma \in H_1(Y, \ZZ)$. Let $\eta \in H_1(Y, \ZZ)$ be the homology class of core curve of the area 1 cylinder in the (1,1) direction as shown in Figure \ref{fig:Yhomologyetasigma}. The following proposition gives a basis for $H_1(Y, \ZZ)$. 

\begin{proposition}\label{prop:Yhomologybasis} The set $\mathcal{B} := \{\sigma, \eta\}$ is a basis for $H_1(Y, \ZZ)$.
\end{proposition}
\begin{proof}
First, $H_1(Y, \ZZ)$ is two dimensional since $Y$ is genus 1. So, it suffices to show that $\sigma$ and $\eta$ are linearly independent. Assume to the contrary that $c_1\sigma = c_2  \eta$ for some $c_i \in \ZZ$. Then, using the 
intersection form $i_a$ (which is symplectic), we get
$$ 0 = i_a(\sigma, c_1\sigma) = i_a(\sigma, c_2 \eta) = c_2 \cdot i_a(\sigma, \eta) = c_2.$$
The final equality is observed from the representative for $\sigma$ and $\eta$ seen in Figure~\ref{fig:Yhomologyetasigma} (b). Likewise, $c_1 = 0$ since we assumed $c_1\sigma = c_2\eta$ and $\sigma$ is non-trivial. Thus $\mathcal{B}$ forms a basis for $H_1(Y, \ZZ)$. 
\end{proof}

We claim that $\text{Aut}(Y) \cong \ZZ_2$. Any element of $\text{Aut}(Y)$ must permute the cone points. There are exactly two cone points of angle $\pi$ on $Y$, call them $\alpha$ and $\beta$. Let $f \in \text{Aut}(Y)$. If $\alpha$ is fixed by $f$, then $f$ restricted to a neighbourhood of $\alpha$ must be the identity, in order to preserve the half-translation structure. Since $f$ is an isometry, this implies $f$ is the identity globally. Similarly if $f$ swaps $\alpha$ and $\beta$, this information alone determines $f$ entirely. We note that in this case, $f$ may be visualized as rotation by $\pi$ around the midpoint of $\tau_3$. This is an involution, and thus $\text{Aut}(Y) \cong \ZZ_2$.

Further, the generator of $\text{Aut}(Y)$ (given by rotation around the midpoint of $\tau_3$ by $\pi$) acts on $\sigma$ and $\eta$ by sending them to $-\sigma$ and $-\eta$ respectively. Consequently, the action of $N \in \SL_2(\RR)$ on $Y$ gives two homeomorphisms $f_N^+: Y \to Y$ and $f_N^-: Y \to Y$, which induce actions on the homology defined by $N \cdot \sigma = {f_N^{\pm}}_*(\sigma)$. Thus, we get the following action.
\[\rho : V(Y)(=\text{Aff}(Y)/\text{Aut}(Y)) \rightarrow \Sp(H_1(Y, \ZZ))/\ideal{ -\text{Id}}\]
In what follows, we drop the $\pm$, with the understanding that $N \cdot \sigma$ is only defined up to sign.

After choosing the oriented basis $\mathcal{B}$ for $H_1(Y, \ZZ)$, we can identify $\Sp(H_1(Y, \ZZ))$ with $\Sp(2, \ZZ) = \SL_2(\ZZ)$ and $\Sp(H_1(Y, \ZZ)/\ideal{ -\text{Id}}$ with $\PSL_2(\ZZ)$. Under this identification, we have the following proposition:

\begin{proposition}The image \label{prop:veechgrprep}$\rho(V(Y)) = \P\left\langle\begin{bmatrix} 1 & 1 \\ 0 & 1 \end{bmatrix},\begin{bmatrix}3 & -1\\ 4 & -1 \end{bmatrix}\right\rangle\leq \PSL_2(\ZZ)$. 
\end{proposition}
\begin{proof}
Since $V(Y) \leq \PSL_2(\ZZ)$ is generated by images in $\PSL_2(\ZZ)$ of matrices $\Theta = \begin{bmatrix}0 & -1 \\ 1 & 0 \end{bmatrix}$, $A =  \begin{bmatrix}1 & 4 \\ 0 & 1 \end{bmatrix}$, and $B=  \begin{bmatrix}5 & -8 \\ 2 & -3 \end{bmatrix}$, it suffices to show $\rho([\Theta])$, $\rho([A])$ and $\rho([B])$ generate the given subgroup. To begin, note that $\sigma = \sigma_1 + \sigma_2$ and $\eta = \sigma_2 + \tau_3 = \tau_4+\sigma_3$.

First, we show that $\rho([\Theta]) = \text{Id}$ by showing that an affine diffeomorphism with derivative of $\pm\Theta$ preserves the homology classes of $\sigma$ and $\eta$ or sends $\sigma$ to $-\sigma$ and $\eta$ to $-\eta$.

Consider the affine diffeomorphism $f_{\Theta}$ given by rotating the left most square of $Y$ (in its representation in Figure~\ref{fig:Yhomologyetasigma} for example) by $\frac{\pi}{2}$ and isometrically continuing the map to the rest of $Y$. Then the derivative of $f_{\Theta}$ is $[\Theta]$. Let $\tau$ be the homology class of $\gamma_{(0,1)}$, the core curve of the vertical cylinder. Note that $ {f_{\Theta}}_*(\sigma) = \tau = \tau_1 - \tau_3 = \sigma_1+\sigma_2=\sigma$ where the two final equalities are seen using the relations (\ref{eqn:homologyrelations}). 
Similarly, using the relations (\ref{eqn:homologyrelations}), we have ${f_{\Theta}}_*(\eta) =-\sigma_1 + \tau_1 =  \tau_4+ \sigma_3 = \eta.$  Thus $\rho([\Theta]) = \text{Id} \in \PSL_2(\ZZ)$.

Next, note that $f_A$ defined as the affine map that accomplished a right Dehn twist about the curve $\sigma$ has constant derivative $[A]$. Thus ${f_A}_*(\sigma) = \sigma$. Since the algebraic intersection $i_a(\sigma, \eta) = 1$ we have ${f_A}_* (\eta) = \sigma + \eta$. Therefore $\rho([A]) \in \PSL_2(\ZZ)$ is the projective image of $\begin{bmatrix}1 & 1 \\ 0 & 1\end{bmatrix}$ in  $\PSL_2(\ZZ)$.

To compute $\rho([B])$, we consider the affine map $f_B$ given by simultaneously applying a left Dehn twist around each of the cylinders in the $(2,1)$ direction. One can show that $[B]$ is the derivative of $f_B$ by noting that $B$ is the composition of rotating clockwise until the $(2,1)$ direction is horizontal, the horizontal shear $\begin{bmatrix}1 & -10 \\ 0 & 1\end{bmatrix}$, followed by rotation counterclockwise, returning the the horizontal to the $(2,1)$ direction. Noting that the modulus of the cylinders is $10$, we see this corresponds to applying a Dehn twist along them.

The understanding of $f_B$ as left Dehn twisting around each of the two cylinders in the $(2,1)$ direction allows for quick computation of ${f_B}_*( \sigma)$. Let $\gamma$ be the core curve of the cylinder in the (2,1) direction that intersects $\sigma_1$ and $\sigma_2$. Note that 
$$\gamma = \sigma_1 + \sigma_2 + \tau_3 + \sigma_2 + \sigma_3 + \tau_4 = \sigma + 2 \eta $$
where the last equality uses the relations \ref{eqn:homologyrelations}. 

Since the algebraic intersection $i_a(\gamma, \sigma_1) = i_a(\gamma, \sigma_2) = -1$, the left Dehn twist $f_B$ gives the action of $B$ on $\sigma_1$ and $\sigma_2$ as, 
$${f_B}_*(\sigma_1) = \sigma_1 + \sigma + 2 \eta \text{ and } {f_B}_* (\sigma_2) = \sigma_2 + \sigma + 2 \eta$$

Hence, we get the image of $\sigma$ under $B$ as,
$$ {f_B}_*(\sigma) = {f_B}_*( \sigma_1) + {f_B}_*(\sigma_2) = \sigma_1 + \sigma_2 + 2(\sigma + 2 \eta) = 3 \sigma+ 4 \eta$$

Finally, since $i_a(\gamma, \eta) = 1$, the left Dehn twist $f_B$ acts on $\eta$ to give the image of $\eta$ under $B$ as,
$${f_B}_*(\eta) = \eta - \gamma = \eta - 2 \eta - \sigma = -\sigma - \eta$$
Therefore, $\rho([B]) \in \PSL_2(\ZZ)$ is the projective image of $ \begin{bmatrix}3 & -1\\ 4 & -1 \end{bmatrix}$ in $\PSL_2(\ZZ)$. \end{proof}

\begin{remark}\label{rem:SLvsPSL}
    Going forward, given any element $M \in V(Y) \leq \PSL_2(\ZZ)$, for convenience of notation, we will work a pre-image of it in $\SL_2(\ZZ)$. Since there are only two choices of pre-images for a given $M \in V(Y)$ and they differ from each other by multiplication by $-\text{Id}$ (which commutes with matrix multiplication and only changes the action on homology by $-\text{Id}$), we will choose a pre-image freely without loss of generality in our arguments. We will denote the pre-image in $SL_2(\ZZ)$ of an element $M \in \PSL_2(\ZZ)$ by $\hat{M}$.

\end{remark}

Since we now know how $V(Y)$ acts on homology, as a corollary, we obtain the image of $\Gamma$ under the representation $\rho$:

\begin{corollary}The image \label{cor:Gammarep}$\rho(\P\Gamma) = \P\left\{\begin{bmatrix} 1 & n \\ 0 & 1\end{bmatrix}\mid n \in \ZZ\right\} < \PSL_2(\ZZ)$.
\end{corollary}
\begin{proof}
This follows directly from the computation of $\rho([\Theta])$ and $\rho([A])$ in the proof of Proposition \ref{prop:veechgrprep}.
\end{proof}
Henceforth, let $$\mathcal{P} := \P \left\{\begin{bmatrix} 1 & n \\ 0 & 1\end{bmatrix}\mid n \in \ZZ\right\}.$$
Next, motivated by Theorem \ref{thm:characterization}, we define the following submodule of $H_1(Y, \ZZ)$:
$$ W := \{\gamma \in H_1(Y, \ZZ)| i_a(\gamma, \sigma) = 0\}$$
Note that since $\sigma \in W$ and $\eta \not \in W$, and $H_1(Y, \ZZ)$ is of rank 2, we must have $W = \ideal{\sigma}$. Moreover, the stabilizer of $W$ in $V(Y)$ acts as $\mathcal{P}$ on homology.

\begin{proposition}\label{prop:parabolicstabilizer}
The image $\rho(\Stab_{V(Y)}(W)) = \mathcal{P}$ where $\Stab_{V(Y)}(W)$ is the stabilizer in $V(Y)$ of the submodule $W$. 
\end{proposition}
\begin{proof}
Let $M \in \rho(\Stab_{V(Y)}(W))$ and let $\hat{M} = \begin{bmatrix}a & b \\ c & d\end{bmatrix}$ be one of its pre-images in $\SL_2(\ZZ)$. By definition, $M \cdot \sigma \in W$. Hence, $c=0$.
But since $\rho(V(Y)) \leq \PSL_2(\ZZ)$, this means $a = \pm 1$, $b \in \ZZ$ and $d = \pm 1$ such that $\hat{M} =  \pm \begin{bmatrix} 1 & n \\ 0 & 1\end{bmatrix}$ for some $n \in \ZZ$ and $\rho(\Stab_{V(Y)}(W)) < \mathcal{P}$.

On the other hand, we would like to show that  $\mathcal{P} \subseteq \rho(\Stab_{V(Y)}(W))$. Let $N \in \mathcal{P}$. Then, $\hat{N} = \begin{bmatrix} 1 & n \\ 0 & 1\end{bmatrix}$ for some $n \in \ZZ$ (without loss of generality, by Remark \ref{rem:SLvsPSL}). Recall from the proof of Proposition~\ref{prop:veechgrprep} that $A \cdot \sigma = \sigma$ so that $\rho([A^n]) = [\hat{N}] = N$. Hence, $N \in \rho(\Stab_{V(Y)}(W))$.

\end{proof}

Next, we show that $\rho(\P\Gamma)$ is the free group on two generators, which will be helpful in precisely establishing the stabilizer of $W$ in $V(Y)$.

\begin{lemma}\label{lem:free}The subgroup $\ideal{\rho([A]), \rho([B])} \leq  \PSL_2(\ZZ)$ is isomorphic to $F_2$. Consequently, $\rho(V(Y)) \simeq F_2$. 
\end{lemma}

\begin{proof}
Let $\rho_A = \begin{bmatrix}1 & 1 \\ 0 & 1 \end{bmatrix}$ and $\rho_B = \begin{bmatrix}3 & -1 \\ 4 & -1 \end{bmatrix}$. With this notation, recall from Proposition \ref{prop:veechgrprep} that $\rho([A]) = [\rho_A]$ and $\rho([B]) = [\rho_B]$. Note then, the following two equations
  \begin{equation*}
    \begin{bmatrix}
      1 & 1 \\ 0 & 2
    \end{bmatrix}^{-1}\cdot \rho_A \cdot
    \begin{bmatrix}
      1 & 1 \\ 0 & 2
    \end{bmatrix}
    =
    \begin{bmatrix}
      1 & 1 \\ 0 & 2
    \end{bmatrix}^{-1}
    \cdot
    \begin{bmatrix}
      1 & 1 \\ 0 & 1
    \end{bmatrix}
    \cdot
    \begin{bmatrix}
      1 & 1 \\ 0 & 2
    \end{bmatrix}
    = 
    \begin{bmatrix}
      1 & 2 \\ 0 & 1
    \end{bmatrix}
  \end{equation*}
  
  \begin{equation*}
    \begin{bmatrix}
      1 & 1 \\ 0 & 2
    \end{bmatrix}^{-1}\cdot \rho_B \cdot
    \begin{bmatrix}
      1 & 1 \\ 0 & 2
    \end{bmatrix}
    =
    \begin{bmatrix}
      1 & 1 \\ 0 & 2
    \end{bmatrix}^{-1}
    \cdot
    \begin{bmatrix}
      3 & -1 \\ 4 & -1
    \end{bmatrix}
    \cdot
    \begin{bmatrix}
      1 & 1 \\ 0 & 2
    \end{bmatrix}
    = 
    \begin{bmatrix}
      1 & 0 \\ 2 & 1
    \end{bmatrix}
  \end{equation*}
So, $\langle\rho([A]), \rho([B])\rangle$, as a subgroup of
  $\PGL(2,\RR)$, is conjugate to the group $P \ideal{\begin{bmatrix} 1 & 2 \\ 0 & 1
  \end{bmatrix}, \begin{bmatrix}
    1 & 0 \\ 2 & 1 \end{bmatrix}}$. This group is isomorphic to $F_2$. Consequently, since $\rho([\Theta]) = \text{Id}$, we see that $\rho(V(Y)) = \ideal{\rho([A]), \rho([B])} \simeq F_2$.
\end{proof}

Putting the last two results together, we find $\rho(\Stab_{V(Y)}(W)) = \rho(\P\Gamma)$. We shall see below that in fact the stabilizer of $W$ in $V(Y)$ is precisely $\P\Gamma$.

\begin{corollary}\label{cor:stabilizergamma}The stabilizer of $W$ in $V(Y)$ is $\Stab_{V(Y)}(W) = \P\Gamma$. 

\end{corollary}

\begin{proof}

We will show that the pre-image set $\rho^{-1}(\mathcal{P}) = \P\Gamma$ and then apply Proposition \ref{prop:parabolicstabilizer} to obtain the statement $\rho^{-1} \circ \rho (\Stab_{V(Y)}(W)) = \P\Gamma$. Since the defining property of elements of $\Stab_{V(Y)}(W)$ depends on their image by $\rho$, $\Stab_{V(Y)}(W) = \rho^{-1} \circ \rho (\Stab_{V(Y)}(W))$. This gives us the desired equality $\Stab_{V(Y)}(W) = \P\Gamma$.

First note that $\P\Gamma \leq \rho^{-1}(\mathcal{P})$ by Corollary \ref{cor:Gammarep}. For the reverse inclusion, we will show $\rho^{-1}(\calP) \leq \P\Gamma$.

Let $M \in \rho^{-1}(\mathcal{P})$ and consider $\hat{M}$, a pre-image in $\SL_2(\ZZ)$. Let $\rho_A = \begin{bmatrix} 1 & 1 \\ 0 & 1  \end{bmatrix}$. Note that $\rho(M) = [(\rho_A)^n]$.

Since $M \in V(Y)$, $\hat{M} \in \hat{V}$ and it can be written as a finite product of matrices from the set $\{A, B, \Theta\}$. Let $w (A, B, \Theta)$ be the reduced word given by the sequence of matrices whose product is $\hat{M}$. To show $\hat{M} \in \Gamma$ (equivalent to showing $M \in \P\Gamma$), we will rearrange $w(A,B, \Theta)$ to an equivalent word given by the product of an element of $\langle A \rangle$ and a conjugate of $\Theta$. We use the following finite process:
\begin{enumerate}
    \item Moving right to left in $w(A, B, \Theta)$, find the maximal disjoint subwords containing each power of $\Theta$ so that each is a conjugate of $\Theta^n$ for some $n\in \mathbb{N}$. Note that a conjugate of a power of $\Theta$ is equal to the product of conjugates of $\Theta$ ($X\Theta^n X^{-1} =(X \Theta X^{-1})^n$). Thus, each of these conjugates is an element of $\Gamma$. If the union of the subwords forms all of $w(A, B, \Theta)$, we are finished. If not, we continue to the next step.
    \item Let $C$ be a conjugate of theta. Note that for any matrix $X$,  $CX = XX^{-1}CX = XD$ where $D=X^{-1}CX$ is another conjugate of $\Theta$. Thus we can ``move'' $X$ from the right to the left of $C$, at the cost of replacing $C$ with a different conjugate of $\Theta$. In this way, we "move" any matrices not part of a conjugate obtained in step (1) to the left. At this point, $w(A, B, \Theta)$ has been transformed so that it consists of a word $w'$ in the letters $A, B$, followed by a word $w''$ in conjugates of $\Theta$ ($w(A,B, \Theta)= w'(A,B)w''(X\Theta X^{-1})$).
\end{enumerate}
To finish the proof, we argue that $w'$ is actually an element of $\langle A \rangle$. From the proof of Proposition~\ref{prop:veechgrprep}, we know that for $C$, a conjugate of $\Theta$, $\rho([C]) = \mathrm{Id}$. Thus $\rho([w''])= \mathrm{Id}$. This gives us
    \[[(\rho_A)^n]=\rho(M)=\rho([w]) = \rho([w'])\rho([w''])=\rho([w']).\]
    The proof of Proposition~\ref{prop:veechgrprep} shows us that $\rho([A^n]) = \rho ([w'])$. Since $[w']=[w'(A,B)]$ and $\rho$ is a homomorphism, we have $\rho([A])^n = [w'(\rho(A),\rho(B))]$. Since $\langle \rho([A]), \rho([B]) \rangle$ is free (Lemma~\ref{lem:free}), we must have $[w'(A,B)]=[A^n]$. Therefore $M \in \P\Gamma$, completing the proof.


\end{proof}

\subsection{Proof of Theorem \ref{thm:veechgrpchar}}

We are now ready to prove Theorem \ref{thm:veechgrpchar}.

For ease of reading, we start by recalling and setting some notation. For any one-cylinder direction $(p,q)$ in $Y$, let $\gamma_{(p,q)}$ denote the homology class in $H_1(Y, \ZZ)$ of the core curve of the cylinder in direction $(p,q)$. Recall from Proposition \ref{prop:Yhomologybasis}, we denote $\gamma_{(1,0)} = \sigma$. For any $N \in \hat{V} = \ideal{\Theta, A, B}$ and any one-cylinder direction $(p,q)$ in $Y$, note that $N \cdot (p,q)$ is a one-cylinder direction as well, as affine diffeomorphisms preserve one-cylinder directions. So, let $N \cdot \gamma_{(p,q)}:= \gamma_{N \cdot(p,q)}$. Recall that $i_a(\gamma_1, \gamma_2)$ denotes the algebraic intersection between homology classes of curves $\gamma_1$ and $\gamma_2$. Finally, recall that 
$W = \{\gamma \in H_1(Y, \ZZ)| i_a(\gamma, \sigma) = 0\} = \ideal{\sigma} \subseteq H_1(Y, \ZZ).$

\begin{proof}[Proof of Theorem \ref{thm:veechgrpchar}]
$(\Leftarrow)$ Let $N \in \Gamma$. We want to show that $N \cdot (1,0)$ is a periodic direction in $M$. Using Theorem \ref{thm:characterization}, it suffices to show that 
\begin{itemize}
\item $N\cdot (1,0)$ is a one-cylinder direction on $Y$ and
\item $i_a(\gamma_{N\cdot (1,0)}, \gamma_{(1,0)}) = 0$.
\end{itemize}
First, $N \cdot (1,0)$ is a one-cylinder direction on $Y$ since $[N] \in \P\Gamma < V(Y)$ and $(1,0)$ is a one-cylinder direction on $Y$. 
Secondly, by Corollary \ref{cor:stabilizergamma}, note that $[N] \in \Stab_{V(Y)}(W)$. This implies that $\gamma_{N \cdot (1,0)} = N \cdot \gamma_{(1,0)} \in W$ as $\gamma_{(1,0)} = \sigma \in W$. Hence, by definition of $W$, $i_a(\gamma_{N \cdot (1,0)}, \gamma_{(1,0)}) = 0$.

$(\Rightarrow)$ Let $(p,q)$ be a periodic direction in $M$. By Theorem \ref{thm:characterization}, we know that $(p,q)$ is a one-cylinder direction for $Y$ and that $i_a(\gamma_{(p,q)}, \gamma_{(1,0)}) = 0$. Since $(p,q)$ is a one-cylinder direction in $Y$, by Proposition \ref{prop:YVeechGroupFacts}, we can write $(p,q) = N \cdot (1,0)$ for some $N \in \ideal{\Theta, A, B}$.
By Theorem \ref{thm:characterization}, the surface $Y$ decomposes into a single cylinder in the $N \cdot (1,0)$ direction and that the core curve of this cylinder has intersection 0 with the horizontal core curve, i.e. $\gamma_{N \cdot (1,0)} \in W$. But $\gamma_{N \cdot (1,0)} = N \cdot \gamma_{(1,0)} = \sigma$ implying that $[N] \in \Stab_{V(Y)}(W)$. But by Corollary~\ref{cor:stabilizergamma}, $[N] \in \P\Gamma$ which implies $N \in \Gamma$.
\end{proof}

  \begin{remark}\label{rem:thm1conseq}
    Theorem~\ref{thm:veechgrpchar} provides and alternate proof of
    Proposition \ref{prop:foureyslopes}: A rational number of the form
  \begin{equation*}
    \frac{p}{q} = [4a_0; 4a_1, 4a_2, \dots,4a_n]
  \end{equation*}
  where $a_0\in\ZZ$ and $a_i \in\ZZ\setminus\set{0}$ for all
  $1\leq i \leq n$ , has its numerator and denominator in the first column of the matrix
  \begin{align*} &\Theta A^{-a_0} \Theta A^{a_1} \Theta A^{-a_2}
    \Theta \dots \Theta A^{(-1)^{n+1} a_n} \Theta\\
    =& \begin{bmatrix}0 & -1\\ 1 & 0 \end{bmatrix}\begin{bmatrix}1 &
      -4a_0 \\ 0 & 1 \end{bmatrix} \begin{bmatrix}0 & -1\\ 1 &
      0 \end{bmatrix}\begin{bmatrix}1 & 4a_1 \\ 0 &
      1 \end{bmatrix} \begin{bmatrix}0 & -1\\ 1 & 0 \end{bmatrix}
    \dots \begin{bmatrix}0 & -1\\ 1 & 0 \end{bmatrix} \begin{bmatrix}1
      & (-1)^{n+1}4a_n \\ 0 & 1 \end{bmatrix}\begin{bmatrix}0 & -1\\ 1
      & 0 \end{bmatrix}.
  \end{align*}
  By Theorem \ref{thm:veechgrpchar}, trajectories with slope $\frac{p}{q}$ are periodic. 
\end{remark}

\begin{remark}
  Rational numbers of the form
  \begin{equation*}
    \frac{p}{q} = [4a_0; 4a_1, 4a_2, \dots,4a_n]
  \end{equation*}
  are bounded away from $1$, and experimental searches for
  periodic directions reveal what appear to be gaps in the set of
  periodic directions. (See~\Cref{fig:experimentalperiodicsoncircle}.)
  In Section~\ref{sec:density} we show that in fact there are no
  gaps---the set of periodic directions is dense. Already, a
  calculation shows that for $n\geq 0$,
  \begin{equation*}
    (B^{-1}A\Theta)^n B^{-1}A \Theta A^{-1}B (B^{-1}A \Theta)^{-n} = \begin{bmatrix}18n^2 + 36 n + 18 & -18n^2-42n-25 \\ 18n^2 +30n + 13 & -18n^2 - 36n - 18 \end{bmatrix}.
  \end{equation*}
  By Theorem \ref{thm:veechgrpchar},
  $(18n^2 + 36 n + 18, 18n^2 +30n + 13)$ is a periodic
  direction. Hence we obtain
$$\frac{p_n}{q_n} := \frac{18n^2 + 30n + 13}{18n^2 + 36n + 18}$$ 
as a family of periodic slopes that approaches $1$ as
$n \rightarrow \infty$, showing that $1$ is not isolated, as experimental evidence and Fourey fractions would suggest.
\end{remark}

\begin{figure}[h!]
\centering
    \includegraphics[scale=0.25]{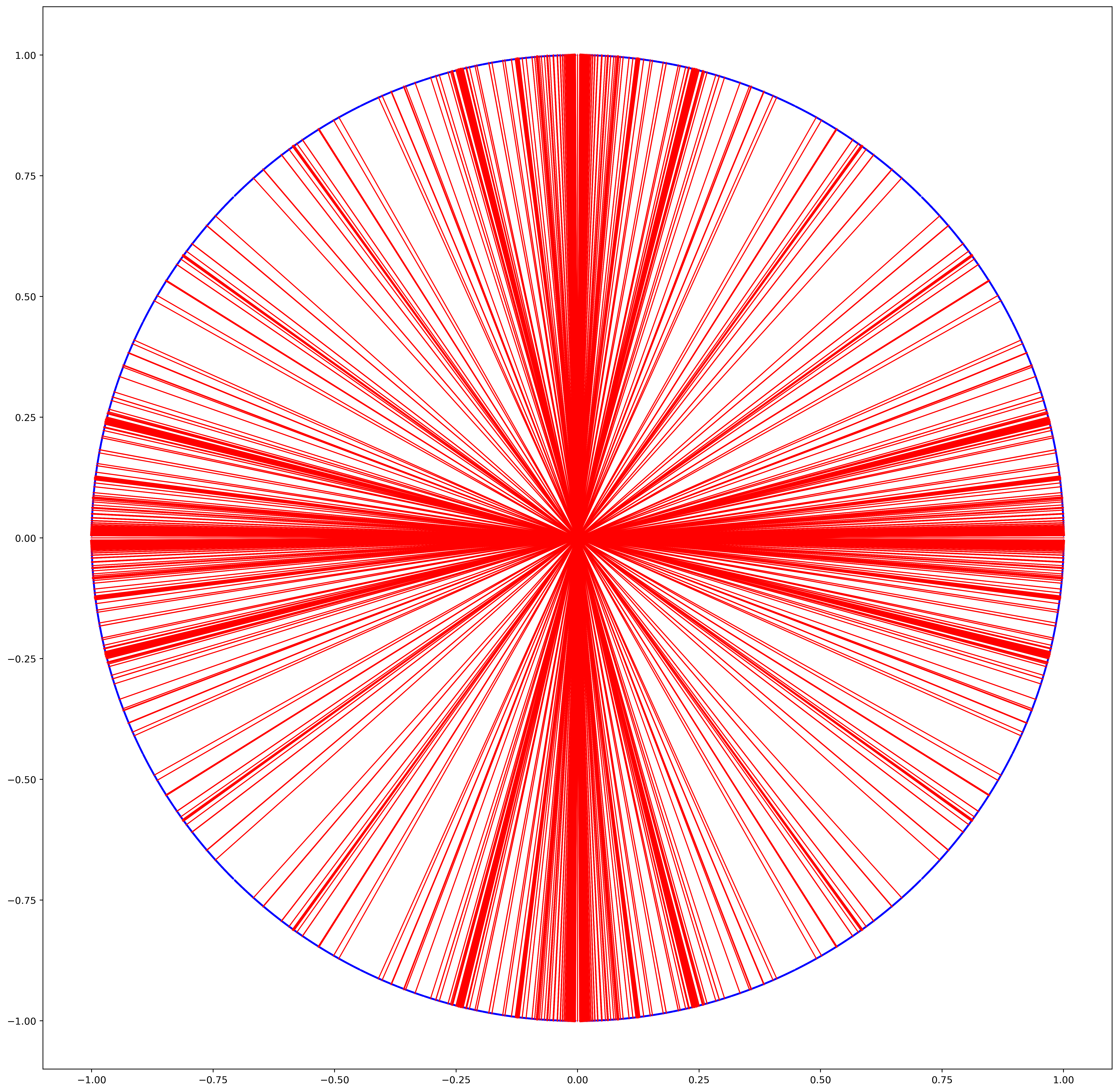}
    \caption{Rational slopes $\frac{p}{q}$ (with $|p|,|q| \leq 230$) on the unit circle. Each ray from the center of the disk represents a periodic slope while the points not associated to rays represent drift-periodic rational slopes.}
    \label{fig:experimentalperiodicsoncircle}
\end{figure}

\section{The Veech group of $M$}\label{sec:veech}

As another consequence of Theorem \ref{thm:veechgrpchar}, we are able to determine $V(M)$, the Veech group of the Mucube.

Recall that $M$ is a $\ZZ^3$ cover of $X$ defined by the map  $\iota_*: \pi_1(X) \rightarrow \pi_1(\TT^3) \simeq \ZZ^3$ induced by the inclusion map $\iota: X \rightarrow \TT^3$ so that $\pi_1(M) = \ker(\iota_*)$. Recall also that $\Aff(M,X)$ denotes the affine diffeomorphism group of the covering, comprised of elements $(\tilde{f}, f) \in \Aff(M) \times \Aff(X)$ such that $\pi \circ \tilde{f} = f \circ \pi$. Recall that $V(M, X) := D(\Aff(M, X))$ denotes the Veech group of the cover (i.e. the group of derivatives, well-defined up to multiplication by $-\text{Id}$ of the affine diffeomorphisms of the covering).

We then have the following corollary stating that the Veech group of the covering $\pi:M \rightarrow X$ is precisely $\P\Gamma$.

\begin{corollary}\label{cor:veechgroup} The Veech group $V(M,X)$ of the covering $\pi:M \rightarrow X$ is $\P\Gamma$. In particular, $\P\Gamma < V(M)$.
\end{corollary}
\begin{proof}
First we note that $V(M, X) = D(\Aff(M, X)) < V(X) = V(Y)$ since for any $(\tilde{f}, f) \in \Aff(M,X)$, $D(\tilde{f}) = D(f) \in V(X) = V(Y)$. Similarly, we note that $V(M,X) < V(M)$.
Now, let $N \in V(M,X)$. Since $(1,0)$ is a periodic direction in $M$ and $V(M,X) < V(M)$, $\hat{N} \cdot (1,0)$ is periodic direction in $M$. By Theorem \ref{thm:veechgrpchar}, $\hat{N} \in \Gamma$ so that $N \in \P\Gamma$. So, $V(M,X) <\P\Gamma$.

On the other hand, let $N \in \P\Gamma$ and let $f \in \Aff(X)$ such that $D(f) = N$. For $\iota_*: \pi_1(X) \rightarrow \ZZ^3$ defined above, we will show that 
$f_*(\ker(\iota_*)) < \ker(\iota_*)$. Then, using Proposition \ref{prop:HWlift}, we will conclude that $N \in D(\Aff(M, X) = V(M,X)$. 

Towards this, let $\alpha \in \ker(\iota_*) = \pi_1(M) < \pi_1(X)$. 

By Proposition \ref{prop:genforfundgroupM}, we know that each element of $\ker(\iota_*)$ can be written as a product of horizontal and vertical cylinder curves (defined in Section \ref{sec:MucubeFundamentalGroup}) and since $f_*$ is a homomorphism, it suffices to show that $f_*(\alpha) \in \ker(\iota_*)$ for any $\alpha$ that is freely homotopic to a horizontal or vertical cylinder curve.

So, assume $\alpha$ is freely homotopic to a horizontal or vertical cylinder curve in $X$. This means $f(\alpha)$ is freely homotopic to a cylinder curve $c$ in $X$ since $f$ is an affine diffeomorphism. Assume for the sake of contradiction that $f_*(\alpha) \not \in \ker(\iota_*)$. Then the cylinder curve $c$ does not lift to a closed curve in $M$. Let $(p,q)$ be the direction of this curve and note that $\hat{N} \cdot (1,0) = \pm (p,q)$ as $N = D(f)$. As $c$ does not lift to a closed curve in $M$, $\hat{N} \cdot (1,0) = \pm(p,q)$ is not a periodic direction for $M$. This means $N \not \in \P\Gamma$ by Theorem \ref{thm:veechgrpchar}, a contradiction.
\end{proof}

Before we obtain the Veech group of $M$, we first have the following basic fact about $(1,n)$ directions where $n \in \ZZ$.

\begin{lemma}\label{lem:1n} A direction $(1,n)$ for $n \in \ZZ$ is periodic if and only if $n \equiv 0 \mod 4$.
\end{lemma}
\begin{proof}
Due to Corollary \ref{cor:oddodd} and Proposition \ref{prop:foureyslopes}, it suffices to argue that $(1,4k+2)$ is not periodic for $k \in \ZZ$. But note that $(1,4k+2) = \begin{bmatrix}1 &  0 \\ 4k & 1 \end{bmatrix}\cdot(1, 2)$. As the projective image of $\begin{bmatrix}1 &  0 \\ 4k & 1 \end{bmatrix}$ is in $\P\Gamma < V(M)$ and $(1,2)$ is not periodic in $M$, $(1,4k+2)$ is not is periodic for any $k \in \ZZ$. 
\end{proof}

Finally, we obtain the entire Veech group of $M$:

\veechgrpM*

\begin{proof}

By Corollary \ref{cor:veechgroup},  $\P\Gamma = D(\Aff(M, X))< V(M)$. So, it remains to show $V(M) < \P\Gamma$.

Let $Q  \in V(M)$ and $\hat{Q} =  \begin{bmatrix}a & b \\ c & d \end{bmatrix}$ be a pre-image in $\SL_2(\ZZ)$. Since $(1,0)$ is a periodic direction in $M$ and $Q \in V(M)$, $\hat{Q}.(1,0) = (a,c)$ is a periodic direction in $M$. By Theorem \ref{thm:veechgrpchar}, there exists $N = \begin{bmatrix} a & x \\ c & y \end{bmatrix} \in \Gamma$ where $x, y \in \ZZ$. 

Now, consider 
$$Q' := \Theta^{-1} N^{-1} \hat{Q}\Theta = \begin{bmatrix}0 & 1 \\ -1 & 0 \end{bmatrix}N^{-1}  \hat{Q} \begin{bmatrix}0 & -1 \\ 1 & 0 \end{bmatrix}  = \begin{bmatrix} ad-bc & 0 \\ dx-by & ay-cx \end{bmatrix} = \begin{bmatrix} 1 & 0 \\ dx-by & 1 \end{bmatrix}.$$

Since $(1,0)$ is a periodic direction, $Q'\cdot(1,0) = (1, dx-by)$ is a periodic direction for $M$. By Lemma \ref{lem:1n}, we have $dx - by = 4k$ for some $k \in \ZZ$. 
Next, consider 
$$Q'' := N^{-1}\hat{Q} = \begin{bmatrix} 1 & by-dx \\ 0 & 1\end{bmatrix}$$
Since $dx - by = 4k$, we have $N^{-1}\hat{Q} = \begin{bmatrix} 1 & -4k \\ 0 & 1\end{bmatrix} = \left(\begin{bmatrix}1 & 4 \\ 0 & 1 \end{bmatrix}\right)^{-k}$. This means $[N^{-1}\hat{Q}] \in \P\Gamma$ since $\begin{bmatrix}1 & 4 \\ 0 & 1 \end{bmatrix} = A \in \Gamma$. Finally, as $[N] \in \P\Gamma$ we must have $[\hat{Q}] = Q \in \P\Gamma$ so that $V(M) < \P\Gamma$. 
\end{proof}

\section{Density of Periodic Directions}\label{sec:density}

The characterization of periodic directions via Theorem \ref{thm:veechgrpchar} allows us to prove the following result about the density of the periodic directions of the Mucube in the set of all directions:

\densityofperiodics*

Before proving Theorem~\ref{thm:density}, we comment on ergodic directions. The decomposition of $M$ into vertical and horizontal cylinders makes evident that $M$ is a Hooper--Thurston--Veech surface (see \cite{Hooper-Thurston-Veech}). Using the work of Hooper (See \cite{Hooper-Thurston-Veech}, Theorem G.3 and Corollary G.4) together with Theorem~\ref{thm:density}, one can then deduce the following corollary. 

\ergodicdirections*

\begin{remark} We make a few remarks on Corollary~\ref{cor:ergodic directions}.
\begin{enumerate}
    \item The Lebesgue measure on $M$ referred here is the $\chi$-Maharam measure with $\chi$ being the trivial homomorphism from $\mathbb{Z}^3$ to $\mathbb{R}$.
    \item The lower bound on the Hausdorff dimension comes from the fact that using Hooper's work it can be shown that non-periodic directions in the limit set of the group $\left\langle\begin{bmatrix}1 & 4 \\ 0 & 1\end{bmatrix}, \begin{bmatrix}1 & 0 \\ 4 & 1\end{bmatrix}\right\rangle$ are ergodic directions for the straight-line flow. As this group is normal in $\left\langle\begin{bmatrix}1 & 4 \\ 0 & 1\end{bmatrix}, \begin{bmatrix}0 & -1 \\ 1 & 0\end{bmatrix}\right\rangle$, the limit sets coincide. Then, using Fedosova's work \cite{Fedosova}, we see that the Hausdorff dimension is bounded below by 0.68. 
    \end{enumerate}
\end{remark}

We now turn our attention to proving Theorem \ref{thm:density}. Since the periodic directions of the Mucube are characterized using the group $\Gamma$, we approach Theorem \ref{thm:density} by first considering its properties as a Fuchsian group. Recall that $\Gamma = \ideal{A,\{g \Theta g^{-1} | g \in \hat{V}\}}$
where $\Theta = \begin{bmatrix}0 & -1 \\ 1 & 0\end{bmatrix}$, $A = \begin{bmatrix} 1 & 4 \\ 0 & 1\end{bmatrix}$, $B = \begin{bmatrix}5 & -8 \\ 2 & - 3 \end{bmatrix}$ and $\hat{V} =\ideal{\Theta, A, B | \Theta^4, \Theta^2 B\Theta^2B^{-1}, \Theta^2 A\Theta^{-2}A^{-1}}.$

The next proposition outlines the properties of $\Gamma$ we will utilize.

\begin{proposition}[Properties of $\Gamma$]\label{prop:gammaproperties} The group $\Gamma$ has the following properties:

\begin{enumerate}
\item It has infinite index in $\SL_2(\ZZ)$.
\item The limit set of $\Gamma$, $\Lambda(\Gamma) = \partial \HH^2$ and consequently, $\Gamma$ is infinitely generated.
\item The set of cusps of $\Gamma$ is dense in $\partial \HH^2 \cong \RR \cup \{\infty\}$.
\end{enumerate}
\end{proposition}
\begin{proof} 
\begin{enumerate} 
\item First note that $[\SL_2(\ZZ): \hat{V}] < \infty$. So it suffices to show that $[\hat{V}: \Gamma] = \infty$. 

We will show that $B^n \Gamma = \Gamma$ implies $n = 0$. Assume $B^n \Gamma = \Gamma$ for some $n$. As, $A \in \Gamma$ we can write $A = B^n g$ for some $g \in \Gamma$. Since $\Gamma$ is generated by $A$ and conjugates of $\Theta$, together with the presentation of $\hat{V}$ we conclude that the sum of exponents of $B$ in $g$ must be 0. This then implies $n = 0$ as $\langle A, B \rangle \simeq F_2$. 
\item Note that the group $\mathcal{N} := \ideal{\{g \Theta  g^{-1}| g \in \hat{V}\}} < \hat{V}$ is a normal subgroup of $G$. By Proposition \ref{part:normallimitset}, the limit sets of these two groups are equal, i.e. $\Lambda(\mathcal{N}) = \Lambda (\hat{V})$. Note that since $\hat{V}$ is finitely generated and finite index in $\SL_2(\ZZ)$, by Proposition \ref{part:finiteindexlimitset}  we have $\Lambda(\hat{V}) = \Lambda(\SL_2(\ZZ)) =  \partial \HH^2$.  But $\mathcal{N} < \Gamma$ which implies $\Lambda (\hat{V}) < \Lambda(\Gamma)$. Hence, $\Lambda(\Gamma) = \partial \HH^2$. By Proposition \ref{part:fuchsianfirstkind} we also conclude that $\Gamma$ is infinitely generated. 

\item By Proposition \ref{part:cuspsaredenseinlimitset}, the cusps of $\Gamma$ are dense in $\Lambda(\Gamma)$ which we just proved is equal to $\partial \HH^2 \cong \RR \cup \{\infty\}$.
\end{enumerate}

\end{proof}

We now use Proposition \ref{prop:gammaproperties} to prove Theorem \ref{thm:density}.

\begin{proof}[Proof of Theorem \ref{thm:density}]
We first show that the set of periodic directions of the Mucube is dense in the set of cusps of $\Gamma$. 

Let $\xi \in \partial \HH^2$ be a cusp of $\Gamma$. 
Since $\Gamma < \SL_2(\ZZ)$, we can write $\xi = \frac{p}{q}$ for some $\frac{p}{q} \in \QQ$.
Since $\xi$ is  cusp of $\Gamma$, there exists a parabolic element $N \in \Gamma$ such that $N \xi = \xi$. By Theorem \ref{thm:veechgrpchar}, $N^m \cdot (1,0)$ is a periodic direction for any $m \in \ZZ$. However, since $N$ is a parabolic element $\lim_{m \rightarrow \infty} N^m \cdot (1,0) = \xi$. This proves that the set of periodic directions is dense in the set of cusps of $\Gamma$.

Finally, by Proposition \ref{prop:gammaproperties}, we know that the set of cusps of $\Gamma$ is dense in $\partial \HH^2 \cong \RR \cup \{\infty\}$. Therefore, the set of periodic directions must be dense in $\RR \cup \{\infty\}$ as well. 
\end{proof}

\bibliographystyle{abbrv}
\bibliography{references}

\end{document}